\documentclass[12pt,twoside]{article}
\usepackage{amsmath,amsfonts,amssymb,amsthm,array,mathdots}
\usepackage{a4,enumitem,indentfirst}
\usepackage{etoolbox} 

\usepackage{stmaryrd} 

\usepackage{url} 
\usepackage{epsfig, color} 
\usepackage{tikz}
\usepackage{pgfplots}\pgfplotsset{compat=1.13}
\usetikzlibrary{arrows,positioning,calc,patterns}

\theoremstyle{plain}

\newtheorem{theo}{Theorem}
\newtheorem{prop}{Proposition}[section]
\newtheorem{conj}[prop]{Conjecture}
\newtheorem{coro}[prop]{Corollary}
\newtheorem{fact}[prop]{Fact}
\newtheorem{lemma}[prop]{Lemma}

\theoremstyle{definition}
\newtheorem{example}[prop]{Example}
\AtEndEnvironment{example}{\null\hfill$\diamond$}%

\newtheorem{rem}[prop]{Remark}
\AtEndEnvironment{rem}{\null\hfill$\diamond$}%

\newcommand\Frenet{\mathfrak{F}}

\newcommand{\Word}{{\mathbf W}}

\newcommand{\conv}{\operatorname{convex}}
\newcommand{\nconv}{\operatorname{non-convex}}

\newcommand{\tok}{\preceq}

\newcommand{\bfx}{{\mathbf x}}
\newcommand{\bfy}{{\mathbf y}}

\newcommand{\Ac}{\operatorname{Ac}}
\newcommand{\Sing}{\operatorname{Sing}}
\newcommand{\sing}{\operatorname{sing}}

\newcommand\SO{\operatorname{SO}}

\newcommand\GL{\operatorname{GL}}
\newcommand\Lo{\operatorname{Lo}}
\newcommand\Up{\operatorname{Up}}

\newcommand\so{\operatorname{\mathfrak{so}}}
\newcommand\lo{\operatorname{\mathfrak{lo}}}

\newcommand\Spin{\operatorname{Spin}}
\newcommand\spin{\mathfrak{spin}}
\newcommand\Cliff{\operatorname{Cl}}
\newcommand\inv{\operatorname{inv}}

\newcommand\Diag{\operatorname{Diag}}
\newcommand\diag{\operatorname{diag}}
\newcommand\B{\operatorname{B}}
\newcommand\Quat{\operatorname{Quat}}
\newcommand\jacobi{\lambda}
\newcommand\fg{\mathfrak{g}}

\newcommand{\longacute}{\operatorname{acute}}
\newcommand{\longgrave}{\operatorname{grave}}

\newcommand{\longhat}{\operatorname{hat}}
\newcommand{\chop}{\operatorname{chop}}
\newcommand{\adv}{\operatorname{adv}}
\newcommand{\iti}{\operatorname{iti}}
\newcommand{\pathiti}{\operatorname{path}}
\newcommand{\Pathiti}{\operatorname{Path}}
\newcommand{\jet}{\operatorname{jet}}

\newcommand{\nmesmo}{\llbracket  n \rrbracket}
\newcommand{\nmaisum}{\llbracket n+1 \rrbracket}

\newcommand{\mult}{\operatorname{mult}}

\newcommand{\NN}{{\mathbb{N}}}
\newcommand{\ZZ}{{\mathbb{Z}}}

\newcommand{\RR}{{\mathbb{R}}}

\newcommand{\Ss}{{\mathbb{S}}}
\newcommand{\BB}{{\mathbb{B}}}
\newcommand{\DD}{{\mathbb{D}}}

\newcommand{\cH}{{\cal H}}
\newcommand{\cL}{{\cal L}}

\newcommand{\cU}{{\cal U}}

\newcommand{\cA}{{\cal A}}
\newcommand{\cB}{{\cal B}}

\newcommand{\cD}{{\cal D}}

\newcommand{\cJ}{{\cal J}}

\newcommand{\cW}{{\cal W}}

\newcommand{\bj}{{\mathbf{j}}}
\newcommand{\bL}{{\mathbf{L}}}
\newcommand{\bQ}{{\mathbf{Q}}}

\newcommand{\fa}{{\mathfrak a}}

\newcommand{\fh}{{\mathfrak h}}

\newcommand{\fl}{{\mathfrak l}}
\newcommand{\fn}{{\mathfrak n}}

\newcommand{\nondeg}{\mathcal{L}\mathbb{S}^n}
\newcommand{\Pos}{\operatorname{Pos}}
\newcommand{\Neg}{\operatorname{Neg}}
\newcommand{\Bru}{\operatorname{Bru}}
\newcommand{\cLjojo}{\cL_n^{\diamond}}
\newcommand{\Brujojo}{\operatorname{Bru}_{\acute\eta}^{\diamond}}
\newcommand{\Bruadv}{\operatorname{Bru}_{\acute\eta}^{0}}
\newcommand{\Bruchop}{\operatorname{Bru}_{\acute\eta}^{1}}

\begin{document}

\title{
Stratification  of spaces of \\ locally convex curves by itineraries}
\author{Victor Goulart
\and Nicolau C. Saldanha}


\date{\today}

\maketitle

\begin{abstract}

Locally convex (or nondegenerate) curves in the sphere $\Ss^n$
(or the projective space)
have been studied for several reasons,
including the study of linear ordinary differential equations
of order $n+1$.
Taking Frenet frames allows us to translate such curves $\gamma$
into corresponding curves $\Gamma$ 
in the flag space,
the  orthogonal group $\SO_{n+1}$ or its double cover $\Spin_{n+1}$.
Determining the homotopy type of the space of such closed curves or,
more generally, of spaces of such curves with prescribed
initial and final jets appears to be a hard problem,
which has been solved for $n=2$ but otherwise remains open.
This paper is a step towards solving the problem
for larger values of $n$.
In the process, we prove a related conjecture of B. Shapiro and M. Shapiro
regarding the behavior of fundamental systems of solutions 
to linear ordinary differential equations. 

We define the \emph{itinerary} 
of a locally convex curve $\Gamma:[0,1]\to\Spin_{n+1}$   
as a (finite) word $w$ in the 
alphabet $S_{n+1}\smallsetminus\{e\}$ of non-trivial permutations. 
This word encodes the succession of 
non-open Bruhat cells of $\Spin_{n+1}$
pierced by $\Gamma(t)$ as $t$ ranges from $0$ to $1$. 
We prove that, for each word $w$, the subspace of curves 
of itinerary $w$ is an embedded contractible 
(globally collared topological) submanifold of finite codimension,
thus defining a stratification of the space of curves. 
We show how to obtain explicit (topologically) transversal sections 
for each of these submanifolds. 
We study both a space of curves 
with minimum regularity hypotheses, 
where only topological transversality applies, 
and spaces of sufficiently regular curves, 
where transversality has the usual meaning. 
In both cases we also study 
the adjacency relation between strata. 

This is an important step in the construction of  
CW cell complexes mapped into
the original space of curves by weak homotopy equivalences.
Our stratification is not as nice as might be desired,
lacking for instance the Whitney property.
Somewhat surprisingly, 
the differentiability class of the curves
affects some properties of the stratification.
The necessary ingredients for the construction
of a dual CW complex are proved.

\medskip 

\end{abstract}

\section{Introduction}
\label{sect:intro}

For a fixed integer $n\geq 2$,
let $\Spin_{n+1}$ be the universal covering
of the orthogonal group $\SO_{n+1}$.  
For $j\in\nmesmo=\{1,2,\ldots,n\}$, 
consider the skew-symmetric tridiagonal matrices 
$\fa_j=e_{j+1}e_j^\top-e_je_{j+1}^\top\in\so_{n+1}$ 
as elements of the Lie algebra $\spin_{n+1}$, via the 
isomorphism of Lie algebras induced by the covering map $\Pi$. 

A map $\Gamma:J\to\Spin_{n+1}$ 
defined on an interval $J\subseteq\RR$ 
is called a \emph{locally convex curve} 
if it is absolutely continuous 
(hence differentiable almost everywhere) and 
its logarithmic derivative is almost everywhere of the form 
\begin{equation}
\label{equation:locallyconvex}
(\Gamma(t))^{-1}\Gamma'(t)=
\sum_{j\in\nmesmo}\kappa_j(t)\fa_j, 
\end{equation} 
for positive functions $\kappa_1,\ldots,\kappa_n:J\to(0,+\infty)$. 

Given a smooth locally convex curve $\Gamma$,  
the smooth curve $\gamma:J\to\Ss^n$ defined by 
$\gamma(t)=\Pi(\Gamma(t))e_1$ 
satisfies 
$\det(\gamma(t),\gamma'(t),\ldots,\gamma^{(n)}(t))>0$ 
for all $t\in J$.
A parametric curve $\gamma:J\to\RR^{n+1}$ of class $C^n$ 
satisfying the inequality above is also called 
\emph{(positive) locally convex} 
\cite{Alves-Saldanha, Saldanha3, Saldanha-Shapiro}
or 
\emph{(positive) nondegenerate} 
\cite{Goulart-Saldanha, Khesin-Ovsienko, Khesin-Shapiro2, Little}. 
Such a curve $\gamma$ can be lifted to a locally convex curve 
$\Frenet_{\gamma}$ in $\SO_{n+1}$ 
(and therefore in $\Spin_{n+1}$) of class $C^1$ 
by taking the orthogonal matrix $\Frenet_{\gamma}(t)$ 
whose column-vectors are the result of applying the 
Gram-Schmidt algorithm to the ordered basis 
$(\gamma(t),\gamma'(t),\ldots,\gamma^{(n)}(t))$ of $\RR^{n+1}$. 
The orthogonal basis of $\RR^{n+1}$ thus obtained
is the (generalized) Frenet frame of the space curve $\gamma$. 
By the classical Frenet-Serret formulae, 
the coefficients $\kappa_1,\ldots,\kappa_n$
of the logarithmic derivative of $\Frenet_\gamma$ 
admit geometric interpretations: 
$\kappa_1=v_\gamma=|\gamma'|$ is  
the velocity of $\gamma$; 
$\kappa_2=v_\gamma\varkappa_1$, 
where $\varkappa_1$ is the geodesic curvature of $\gamma$; 
$\kappa_3=v_\gamma\varkappa_2$, 
where $\varkappa_2$ is the geodesic torsion of $\gamma$, 
and so on  
\cite{Klingenberg1, Novikov-Yakovenko}.
The term locally convex comes from the fact that 
a nondegenerate curve $\gamma:J\to\RR^{n+1}$ 
can be partitioned into finitely many 
\emph{convex} arcs, i.e., arcs that intersect any $n$-dimensional subspace of $\RR^{n+1}$ at most $n$ times 
(with multiplicities taken into account); 
see Subsection \ref{subsect:convex}.

Given $r\in\NN^\ast$ and $z_0,z_1\in\Spin_{n+1}$, 
let $\cL_n^{[C^r]}(z_0;z_1)$ denote the space of 
locally convex curves $\Gamma:[0,1]\to\Spin_{n+1}$ 
of differentiability class $C^r$ with endpoints 
$\Gamma(0)=z_0$ and $\Gamma(1)=z_1$. 
We endow this space with the usual  
$C^r$ topology and consider the problem 
of describing its homotopy type. 
This is equivalent to the problem of studying the homotopy 
type of the space $\nondeg(z_0;z_1)$ of 
nondegenerate spherical curves $\gamma:[0,1]\to\Ss^n$ 
satisfying $\Frenet_\gamma(0)=z_0$ and $\Frenet_\gamma(1)=z_1$ 
with the subspace topology inherited from 
$C^n([0,1],\RR^{n+1})$ 
(see Subsection \ref{subsect:Hilbert}). 
Some historical motivation for this problem 
is given at the end of this introduction. 
Of course, we have the natural homeomorphism 
$\cL_n^{[C^r]}(z_0;z_1)\approx\cL_n^{[C^r]}(1;z_0^{-1}z_1)$. 
The present paper provides an important preliminary step 
for the construction of an abstract cell complex $\cD_n(z)$
weak homotopy equivalent to 
$\cL_{n}^{[C^r]}(z)=\cL_{n}^{[C^r]}(1;z)$. 
The existence of $\cD_n(z)$ 
and the construction of its lowest dimensional 
skeletons is addressed in \cite{Goulart-Saldanha-cw},
which can be considered a sequel to the present paper
(see also 
\cite{Goulart-Saldanha} 
for a preliminary version). 

In Subsection \ref{subsect:Hilbert} 
we endow the space $\cL_{n}^{[C^r]}(z_0;z_1)$ 
with a convenient Banach manifold atlas. 
For technical reasons, 
we also consider alternate Hilbert manifold versions 
$\cL_n^{[H^{r}]}(z_0;z_1)$ of these spaces. 
A particularly interesting case is 
$\cL^{[H^1]}_n(z_0;z_1)$, where we are allowed 
to perform certain constructions that 
violate the continuity of the logarithmic derivative 
(e.g. the homotopies in the proofs of 
Lemmas \ref{lemma:convex2} and \ref{lemma:tok0}).  
Our approach is reminiscent of the construction 
in \cite{Klingenberg2} 
of the Hilbert manifold $H^1([0,1],M)$ of absolutely 
continuous curves in a compact Riemannian manifold $M$. 
General results from the homotopy theory of infinite-dimensional 
manifolds 
\cite{Burghelea-Henderson, Burghelea-Saldanha-Tomei1, Henderson, Palais} 
(explicitly, Facts \ref{fact:BST} and \ref{fact:BH} below) 
imply that, for fixed $z_0,z_1\in\Spin_{n+1}$, 
all versions of $\cL_n^\ast(z_0;z_1)$ 
we are interested in 
are indeed homeomorphic: 
these include $\ast=H^1$ and $\ast=H^r, C^r$ for large $r$. 
This justifies dropping the distinctive superscripts $\ast$ 
and referring simply to the spaces $\cL_n(z_0;z_1)$  
when no serious confusion is likely to arise. 
In some situations, though, the distinction is crucial   
(e.g. Section \ref{sect:Hk}). 

In order to state our main results, 
we rely on the Bruhat stratification 
of the spin group, studied in \cite{Goulart-Saldanha0}. 
We invoke many notations and results directly from that paper. 
For convenience, 
some recollection is provided in Subsection~\ref{subsect:Bruhat}. 

In a nutshell, $\Spin_{n+1}$ is a disjoint union 
of unsigned Bruhat cells $\Bru_\sigma$, indexed on the 
symmetric group $S_{n+1}$ of permutations of the set 
$\nmaisum$. 
Each unsigned Bruhat cell $\Bru_\sigma$ has 
exactly $2^{n+1}$ connected components 
$\Bru_{q\acute\sigma}$, 
called signed Bruhat cells, 
each one an embedded submanifold of $\Spin_{n+1}$ 
diffeomorphic to $\RR^d$, where 
$d=\inv(\sigma)$ is the number of inversions 
of the permutation $\sigma$ 
(Corollary 1.2 of \cite{Goulart-Saldanha0}).  
The collection of signed Bruhat cells is indexed by the group 
$\widetilde\B^+_{n+1}\subset\Spin_{n+1}$, 
the lift of the subgroup $\B^+_{n+1}\subset\SO_{n+1}$ 
of signed permutation matrices with positive determinant. 
The acute map  
$\sigma\in S_{n+1}\mapsto \acute\sigma\in\widetilde\B^+_{n+1}$, 
which is not a homomorphism, 
is a right inverse to the natural projection 
$\widetilde \B^+_{n+1}\to S_{n+1}$ and is
defined in Equation \eqref{equation:acutegravei} 
in Subsection~\ref{subsect:Bruhat}  
(also, Equation (2)  
of \cite{Goulart-Saldanha0}). 
We can then write each element $z\in\widetilde\B^+_{n+1}$ as 
$z=q\acute\sigma$ for unique $\sigma\in S_{n+1}$ and 
$q\in\Quat_{n+1}$. 
Here, $\Quat_{n+1}\subset\widetilde\B^+_{n+1}$ 
is the lift to $\Spin_{n+1}$ of the subgroup 
$\Diag^+_{n+1}\subset\B^+_{n+1}$ of diagonal matrices. 

Let $\eta\in S_{n+1}$ have the maximum number of inversions, 
$\inv(\eta)=n(n+1)/2$, i.e., let $\eta:j\mapsto n+2-j$. 
The element $\eta$ is called the top permutation of $S_{n+1}$ 
and is often denoted in the literature by $w_0$.
The cells $\Bru_{q\acute\eta}$, $q\in\Quat_{n+1}$, are open, 
and their union $\Bru_\eta$ 
is a dense open subspace of the spin group. 
We call the complement 
$\Sing_{n+1}=\Spin_{n+1}\smallsetminus\Bru_\eta$ 
the \emph{singular set} of the spin group;  
this is a singular variety of codimension one. 
Accordingly, we define the \emph{singular set} of a 
locally convex curve $\Gamma:[t_0,t_1]\to\Spin_{n+1}$ as 
$\sing(\Gamma)=
\Gamma^{-1}[\Sing_{n+1}]\smallsetminus\{t_0,t_1\}
\subset(t_0,t_1)$;
the elements of $\sing(\Gamma)$ are sometimes called
the moments of non-transversality between 
the osculating flag of $\gamma=(\Pi\circ\Gamma)e_1$
and the standard complete flag of $\RR^{n+1}$ 
\cite{Shapiro-Shapiro3}.
Theorem 3 of \cite{Goulart-Saldanha0} 
implies that nondegenerate curves $\Gamma\in\cL_n(z_0;z_1)$ 
have finite singular sets $\sing(\Gamma)\subset (0,1)$.  

Recall that the Hausdorff distance
\cite{Falconer}
between two nonempty compact sets $X, Y \subset [0,1]$ is:
\[ d_{\cH}(X,Y) = \max \left\{
(\sup_{x\in X} \inf_{y \in Y} |x-y|),
(\sup_{y\in Y} \inf_{x \in X} |x-y|) \right\}; \]
we also define $d_{\cH}(\emptyset,X) = 1$ for $X \ne \emptyset$
and  $d_{\cH}(\emptyset,\emptyset) = 0$.
Let $\cH([0,1]) \subset 2^{[0,1]}$ be the 
set of compact subsets of $[0,1]$;
this is a complete metric space with the Hausdorff distance where
the empty set is an isolated point.

\begin{theo}
\label{theo:Hausdorff}
Given $z_0,z_1\in\Spin_{n+1}$, 
the map $\sing: \cL_n(z_0;z_1) \to \cH([0,1])$ defined by  
$\Gamma\mapsto\sing(\Gamma)$ is continuous.
\end{theo}

This is obtained as an immediate consequence of 
Lemma \ref{lemma:novanishingletter}. 
These results imply in particular that 
when a locally convex curve is deformed 
(while remaining in $\cL_n(z_0;z_1)$), 
points in the singular set
may join or split but never vanish or appear out of nowhere. 
This is closely related to 
the known fact \cite{Anisov, Shapiro-Shapiro, Shapiro} 
that convex curves 
(when they exist)  
form a connected component $\cL_{n,\conv}(z_0;z_1)$ 
of $\cL_n(z_0;z_1)$ (Lemma \ref{lemma:convex2}).
A locally convex curve $\Gamma:[t_0,t_1]\to\Spin_{n+1}$ 
is said to be \emph{(globally) convex} if 
$\sing((\Gamma(t_0))^{-1}\Gamma) = \emptyset$. 
We review in Subsection \ref{subsect:convex} 
the equivalence between this notion of convexity and
the geometric, more classical one 
\cite{Khesin-Shapiro2, Saldanha3, Shapiro-Shapiro, Shapiro}, 
introduced above in terms of the 
nondegenerate curve $\gamma=(\Pi\circ\Gamma)e_1$. 
Another closely related result is Lemma \ref{lemma:conjecture}, 
which proves Conjecture 2.6 of \cite{Shapiro-Shapiro3}.


Given a locally convex curve 
$\Gamma:[t_0,t_1]\to\Spin_{n+1}$, 
write $\sing(\Gamma)=\{\tau_1<\cdots<\tau_\ell\}\subset(t_0,t_1)$ 
and, for each $j\in\llbracket\ell\rrbracket$, let  
$\Gamma(\tau_j)\in\Bru_{\eta\sigma_j}$, 
$\sigma_j\in S_{n+1}\smallsetminus\{e\}$. 
Let $\Word_{n}$ be the set of finite words in the 
alphabet $S_{n+1}\smallsetminus\{e\}$.   
We define the \emph{itinerary} of $\Gamma$ by 
$\iti(\Gamma)=(\sigma_1,\ldots,\sigma_\ell)\in\Word_n$. 

We define our working space of 
locally convex curves as
\begin{equation}
\label{equation:cL}
\cL_{n}=\bigsqcup_{q\in\Quat_{n+1}}\cL_{n}(1;q).
\end{equation}
Given $z_0,z_1\in\Spin_{n+1}$, 
we can determine explicitly an 
element $q\in\Quat_{n+1}$ such that 
the spaces $\cL_{n}(z_0;z_1)$ and 
$\cL_{n}(q)=\cL_{n}(1;q)$ 
are homeomorphic. 
Therefore, in order to understand 
all the spaces $\cL_{n}(z_0;z_1)$, one may restrict attention 
to the disjoint union of $2^{n+1}$ spaces 
in Equation \eqref{equation:cL}. 
The problem of determining whether the spaces
$\cL_{n}(q_0)$ and $\cL_n(q_1)$ are homeomorphic
(where $q_0,q_1 \in \Quat_{n+1}$,  $q_0 \ne q_1$)
has been considered in \cite{Alves-Saldanha,Saldanha-Shapiro};
Corollary 1.1 
in \cite{Goulart-Saldanha-cw} gives partial results.
For $n = 3$,
our methods allow for a rather complete discussion in
\cite{Alves-Goulart-Saldanha},
culminating in the computation of the homotopy type
of $\cL_3(z)$ for all $z \in \Spin_4$:
see Corollary \ref{coro:center} and Theorem \ref{theo:L3}
in the Final Remarks.


For $w =(\sigma_1,\ldots,\sigma_\ell)\in\Word_n$, set
\begin{equation}
\label{equation:word}
\hat w=\hat\sigma_1\cdots\hat\sigma_\ell\in\Quat_{n+1}, 
\qquad 
\dim(w)=\dim(\sigma_{1})+\cdots+\dim(\sigma_{\ell}), 
\end{equation}
where $\sigma\in S_{n+1}\mapsto\hat\sigma\in\Quat_{n+1}$ is 
the $\longhat$ map defined in 
Equation \eqref{equation:acutegravei} 
and $\dim(\sigma)=\inv(\sigma)-1$, for all $\sigma\in S_{n+1}$. 
Notice that $\acute\eta\hat w\acute\eta\in\Quat_{n+1}$ for 
all words $w\in\Word_n$.
Let $\cL_n[w]\subset\cL_n$ be the 
subset of curves with itinerary $w$.

\begin{rem}
\label{rem:collared}
We recall the concepts of tubular neighborhood and collared topological submanifold. 
Let $M_0$ be a (finite or infinite dimensional) manifold
and $M_1 \subseteq M_0$:
the subset $M_1$ is called
a (globally) \emph{collared topological submanifold of codimension $d$}
if and only if there exists an open set $\hat A_0$,
$M_1 \subseteq \hat A_0 \subseteq M_0$,
which is a \emph{tubular neighborhood} of $M_1$
(based on \cite{Brown}).
We say that $\hat A_0$ as above is a tubular neighborhood if
there exist
an open ball $B \subseteq \RR^d$, $0 \in B$,
a continuous projection $\Pi: \hat A_0 \to M_1 \subseteq \hat A_0$
and a continuous map $\hat F: \hat A_0 \to B$
such that the map $(\Pi,\hat F): \hat A_0 \to M_1 \times B$ is a homeomorphism.
Embedded $C^2$ submanifolds of Hilbert spaces 
with finite codimension 
are collared topological submanifolds:
in this case $\Pi$ can be taken to be the normal projection. 
\end{rem}

We are mostly interested in the cases $H^1$ 
(essentially no regularity hypotheses) and 
$H^r$ for large $r$. 

\begin{theo}
\label{theo:stratification}
For $r\in\NN^\ast$ and $w\in\Word_{n}$, we have 
$\cL^{[H^r]}_{n}[w]\subseteq
\cL^{[H^r]}_n(\acute\eta\hat w\acute\eta)$; also:
\begin{enumerate}
\item\label{item:theo:stratification:H1}
{The set $\cL^{[H^1]}_{n}[w]$
is a contractible globally collared topological submanifold of 
$\cL^{[H^1]}_{n}(\acute\eta\hat w\acute\eta)$ 
of codimension $\dim(w)$.} 
\item\label{item:theo:stratification:Hklarge}
{If $r\geq3$, then $\cL_n^{[H^r]}[w]$ is 
an embedded $C^{r-1}$ submanifold of 
$\cL^{[H^r]}_{n}(\acute\eta\hat w\acute\eta)$ 
of codimension $\dim(w)$ 
(with tubular neighborhood fibred by normal balls).
}   
\end{enumerate}
\end{theo}

In particular, 
all words $w\in\Word_n$ 
are realizable as itineraries of locally convex curves,  
the empty word $(\,)\in\Word_n$ being the itinerary 
of the convex curves. 
In fact, it follows from Lemma \ref{lemma:convex2} that  
$\cL_{n,\conv}=\cL_n[(\,)]$ is a contractible connected 
component of $\cL_n$ contained in $\cL_n(\hat\eta)$, 
where $\hat\eta=\acute\eta^2$,
consistently with known results \cite{Anisov, Shapiro-Shapiro, Shapiro}. 
The proof of Theorem \ref{theo:stratification} 
is presented 
in Section \ref{sect:paths}. 
Some preliminary steps are covered in 
Sections \ref{sect:acctriangle} and \ref{sect:acc}. 

We have thus defined the \emph{itinerary stratification}  
that gives this paper its title: 
\begin{equation}
\label{equation:stratification} 
\cL_n = \bigsqcup_{w \in \Word_n} \cL_n[w];
\qquad
\cL_n[w] = 
\{ \Gamma \in \cL_n \;|\; \iti(\Gamma) = w \}.
\end{equation} 

We devote the last part of this paper 
to investigate how 
these strata fit together. 
Explicit parameterizations of transversal sections 
of $\cL_n[w]$ are constructed in Section~\ref{sect:transversal} 
with this goal in mind.
Unlike the homotopy type of the spaces $\cL_n(z_0;z_1)$,
this turns out to be sensitive to the regularity class, i.e.,
on which version $\cL_n^\ast$ we are actually using 
(see Section~\ref{sect:Hk}).

We produce a simple, visual example below. 
For $n\leq4$, we use the simplified notation 
$a=a_1$, $b=a_2$, $c=a_3$, $d=a_4$ for the 
Coxeter generators of $S_{n+1}$. 
We also write a word in $\Word_n$ as a string of letters, 
as in, say, $ab[ab]abb[aba][ab]=(a,b,ab,a,b,b,aba,ab)$. 
Square brackets are used to avoid confusion between, say, 
$a[ba]=(a,ba)$, $[aba]=(aba)$ and $aba=(a,b,a)$, 
of respective lengths $2$, $1$ and $3$. 

\begin{figure}[ht]
\def\svgwidth{8cm}
\centerline{%
\begingroup%
  \makeatletter%
  \providecommand\color[2][]{%
    \errmessage{(Inkscape) Color is used for the text in Inkscape, but the package 'color.sty' is not loaded}%
    \renewcommand\color[2][]{}%
  }%
  \providecommand\transparent[1]{%
    \errmessage{(Inkscape) Transparency is used (non-zero) for the text in Inkscape, but the package 'transparent.sty' is not loaded}%
    \renewcommand\transparent[1]{}%
  }%
  \providecommand\rotatebox[2]{#2}%
  \ifx\svgwidth\undefined%
    \setlength{\unitlength}{624.78230975bp}%
    \ifx\svgscale\undefined%
      \relax%
    \else%
      \setlength{\unitlength}{\unitlength * \real{\svgscale}}%
    \fi%
  \else%
    \setlength{\unitlength}{\svgwidth}%
  \fi%
  \global\let\svgwidth\undefined%
  \global\let\svgscale\undefined%
  \makeatother%
  \begin{picture}(1,1.00163023)%
    \put(0,0){\includegraphics[width=\unitlength,page=1]{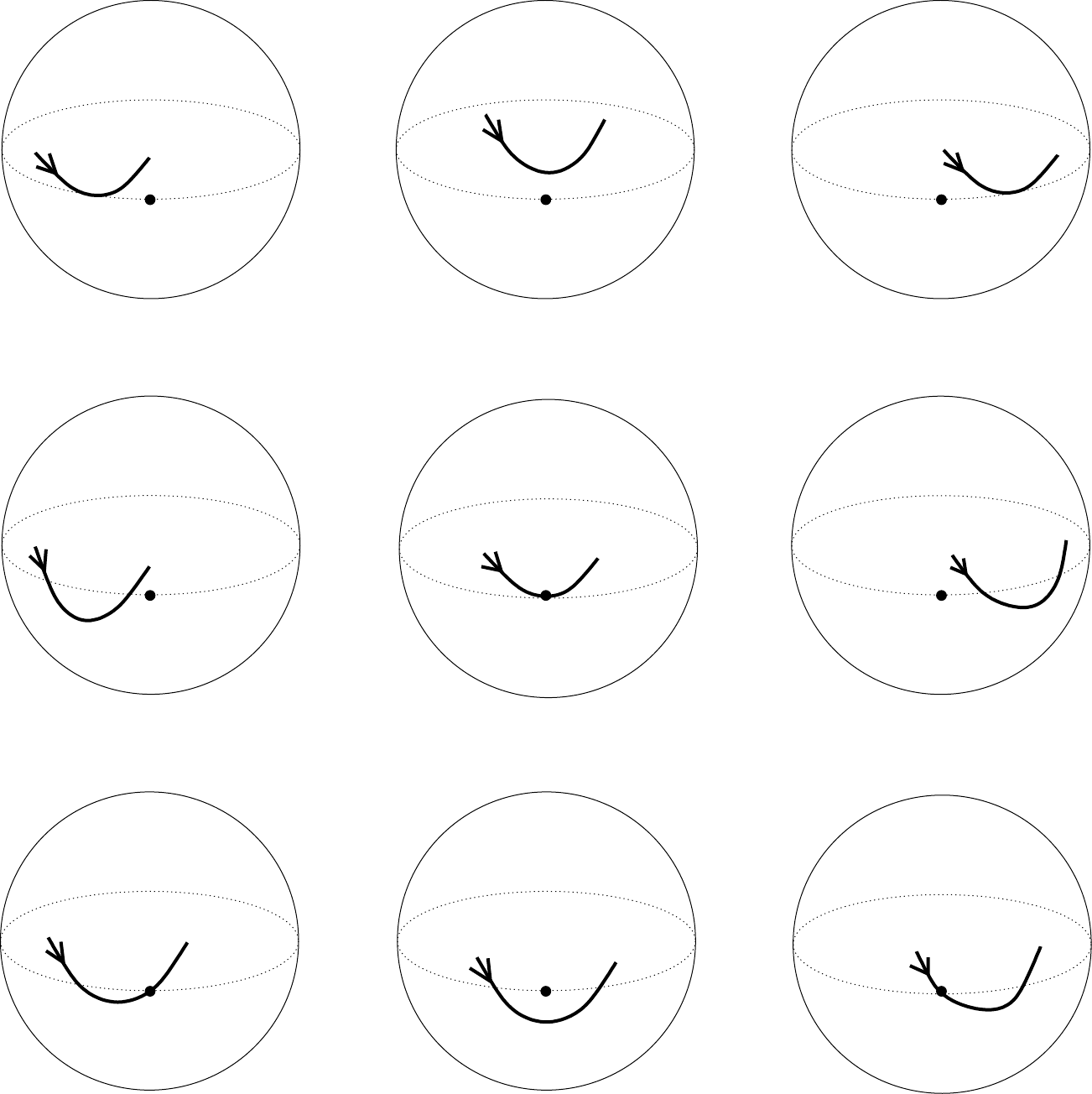}}%
    \put(0.12332229,0.94577198){\color[rgb]{0,0,0}\makebox(0,0)[lb]{\smash{}}}%
    \put(0.10680052,0.93225411){\color[rgb]{0,0,0}\makebox(0,0)[lb]{\smash{}}}%
    \put(0.09628655,0.93676007){\color[rgb]{0,0,0}\makebox(0,0)[lb]{\smash{$[ab]b$}}}%
    \put(0.47779116,0.9337561){\color[rgb]{0,0,0}\makebox(0,0)[lb]{\smash{$bb$}}}%
    \put(0.82775397,0.93676007){\color[rgb]{0,0,0}\makebox(0,0)[lb]{\smash{$b[ab]$}}}%
    \put(0.09478453,0.57027542){\color[rgb]{0,0,0}\makebox(0,0)[lb]{\smash{$abab$}}}%
    \put(0.46577517,0.56877339){\color[rgb]{0,0,0}\makebox(0,0)[lb]{\smash{$[aba]$}}}%
    \put(0.82625185,0.57177734){\color[rgb]{0,0,0}\makebox(0,0)[lb]{\smash{$baba$}}}%
    \put(0.10229445,0.20529258){\color[rgb]{0,0,0}\makebox(0,0)[lb]{\smash{$a[ba]$}}}%
    \put(0.47478712,0.20529267){\color[rgb]{0,0,0}\makebox(0,0)[lb]{\smash{$aa$}}}%
    \put(0.82475011,0.2067946){\color[rgb]{0,0,0}\makebox(0,0)[lb]{\smash{$[ba]a$}}}%
  \end{picture}%
\endgroup%
}
\caption{A family of curves in $\cL_2$.
The equator is dashed and the fat dot indicates $e_1$.
The vector $e_2$ is pointing to the right.}
\label{fig:aba}
\end{figure}

\begin{example}
\label{example:aba}
Let $n = 2$.
In Figure \ref{fig:aba},
we draw the nondegenerate curve $\gamma: [t_0,t_1] \to \Ss^2$, 
$\gamma(t)=\Pi(\Gamma(t))e_1$, 
as a visual representation of the corresponding locally convex curve
$\Gamma = \Frenet_{\gamma}: [t_0,t_1] \to \Spin_3$.
A letter $a = a_1$ in $\iti(\Gamma)$ 
corresponds to the curve $\gamma$ 
transversally crossing the equator 
(i.e., the great circle $x_3 = 0$)
at a point different from $\pm e_1$.
A letter $b = a_2$ occurs when the tangent geodesic (great circle)
to $\gamma$ at $t$ includes the points $\pm e_1$
but the $x_3$-coordinate of $\gamma(t)$ is non-zero.
A letter $[ab]$ indicates that the curve is tangent to the equator,
but not at $\pm e_1$.
A letter $[ba]$ declares that the curve crosses the equator transversally
at $\pm e_1$.
Finally, $[aba]$ proclaims that the curve is tangent to the equator
at $\pm e_1$.
Figure \ref{fig:aba} shows a two-parameter family
of (arcs of) curves in $\cL_2$ illustrating all these cases. 
The reader may want to compare this with the explicit 
parameterization of a tranversal section of $\cL_2[[aba]]$ 
obtained in Example~\ref{example:transversalsectionaba}. 
\end{example}

\begin{figure}[ht]
\def\svgwidth{10cm}
\centerline{%
\begingroup%
  \makeatletter%
  \providecommand\color[2][]{%
    \errmessage{(Inkscape) Color is used for the text in Inkscape, but the package 'color.sty' is not loaded}%
    \renewcommand\color[2][]{}%
  }%
  \providecommand\transparent[1]{%
    \errmessage{(Inkscape) Transparency is used (non-zero) for the text in Inkscape, but the package 'transparent.sty' is not loaded}%
    \renewcommand\transparent[1]{}%
  }%
  \providecommand\rotatebox[2]{#2}%
  \ifx\svgwidth\undefined%
    \setlength{\unitlength}{452.14337261bp}%
    \ifx\svgscale\undefined%
      \relax%
    \else%
      \setlength{\unitlength}{\unitlength * \real{\svgscale}}%
    \fi%
  \else%
    \setlength{\unitlength}{\svgwidth}%
  \fi%
  \global\let\svgwidth\undefined%
  \global\let\svgscale\undefined%
  \makeatother%
  \begin{picture}(1,0.34282213)%
    \put(0,0){\includegraphics[width=\unitlength,page=1]{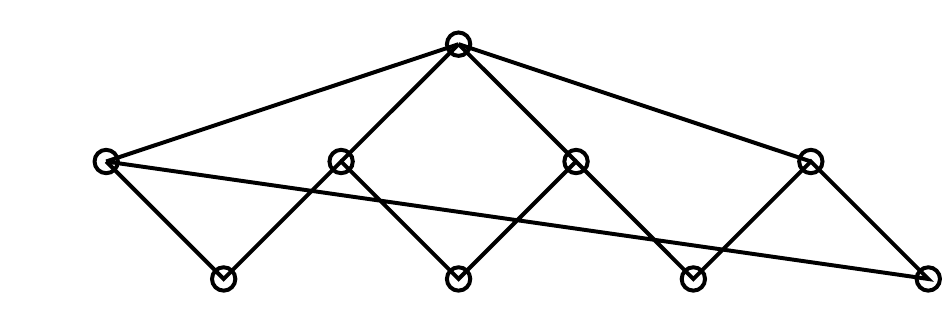}}%
    \put(0.37454972,0.32086658){\color[rgb]{0,0,0}\makebox(0,0)[lb]{\smash{\textbf{\textit{$[aba]$}}}}}%
    \put(0.00059906,0.18376892){\color[rgb]{0,0,0}\makebox(0,0)[lb]{\smash{\textbf{\textit{$[ba]a$}}}}}%
    \put(0.25618919,0.18376892){\color[rgb]{0,0,0}\makebox(0,0)[lb]{\smash{\textbf{\textit{$a[ba]$}}}}}%
    \put(0.87952825,0.18997065){\color[rgb]{0,0,0}\makebox(0,0)[lb]{\smash{\textbf{\textit{$b[ab]$}}}}}%
    \put(0.63013984,0.18376892){\color[rgb]{0,0,0}\makebox(0,0)[lb]{\smash{\textbf{\textit{$[ab]b$}}}}}%
    \put(0.18757439,0.00919702){\color[rgb]{0,0,0}\makebox(0,0)[lb]{\smash{\textbf{\textit{$aa$}}}}}%
    \put(0.38708511,0.00919702){\color[rgb]{0,0,0}\makebox(0,0)[lb]{\smash{\textbf{\textit{$abab$}}}}}%
    \put(0.68001752,0.0029953){\color[rgb]{0,0,0}\makebox(0,0)[lb]{\smash{\textbf{\textit{$bb$}}}}}%
    \put(0.88572997,0.00919702){\color[rgb]{0,0,0}\makebox(0,0)[lb]{\smash{\textbf{\textit{$baba$}}}}}%
  \end{picture}%
\endgroup%
}
\caption{The adjacency relation between the itineraries 
below $[aba]$.}
\label{fig:subaba}
\end{figure}

We define a partial order in $\Word_n$ by
\begin{equation} 
\label{equation:poset}
w_0\tok w_1 \quad \Leftrightarrow \quad
\cL_n^{[H^1]}[w_1]\subseteq\overline{\cL_n^{[H^1]}[w_0]}.
\end{equation}
The Hasse diagram in Figure \ref{fig:subaba}  
represents the above partial order restricted to
$\{ w \in \Word_2 \;|\; w \tok [aba] \} =
\{[aba], a[ba], [ba]a, b[ab], [ab]b, aa, abab, baba, bb\}$.

Equation \eqref{equation:poset} 
defines a poset structure in $\Word_n$ that 
inherits (so to speak) some features from 
the strong Bruhat order $\leq$ in $S_{n+1}$:  
recall that 
$\sigma_0\leq\sigma_1$ in $S_{n+1}$ 
if and only if 
$\Bru_{\sigma_0}\subseteq\overline{\Bru_{\sigma_1}}$ 
in $\Spin_{n+1}$ 
(see, for instance, Corollary 1.1 
of \cite{Goulart-Saldanha0}; 
notice the reversion of the indices
in relation to Equation \eqref{equation:poset}). 


\begin{theo}
\label{theo:poset}
For $w_0,w_1=(\sigma_1,\ldots,\sigma_\ell)\in\Word_n$, 
$w_0\tok w_1$ is equivalent to each one 
of the following conditions:
\begin{enumerate}[label=(\roman*)]
\item\label{item:inclusion}{$\cL_n^{[H^1]}[w_1]\subseteq
\overline{\cL_n^{[H^1]}[w_0]}$;}
\item\label{item:intersection}{$\cL_n^{[H^1]}[w_1]\cap
\overline{\cL_n^{[H^1]}[w_0]}\neq\emptyset$;}
\item\label{item:open}{given $\Gamma_1 \in \cL_n^{[H^1]}[w_1]$, $\epsilon > 0$,  
$\sing(\Gamma_1) = \{t_1<\cdots<t_\ell\}$
and an open neighborhood $U \subset \cL_n^{[H^1]}$ of $\Gamma_1$ there exists
$\Gamma \in U \cap \cL_n^{[H^1]}[w_0]$
with $\Gamma$ and $\Gamma_1$
coinciding outside 
$\cup_{i\in\llbracket\ell\rrbracket}(t_i-\epsilon,t_i+\epsilon)$;}
\item\label{item:subwords}{there exist nonempty words 
$\tilde w_1,\ldots,\tilde w_\ell\in\Word_n$ such that 
$w_0=\tilde w_1\cdots\tilde w_\ell$ and,  
for all $i\in\llbracket\ell\rrbracket$, 
$\tilde w_i\tok (\sigma_i)$.}
\end{enumerate}
\end{theo}

The empty word $(\,)\in\Word_n$ is an isolated point.
We prove Theorem \ref{theo:poset} in Section \ref{sect:poset}. 
The corresponding statement is false for $\ast = [H^r]$, 
$r$ large;
indeed, we shall see in Section \ref{sect:Hk} that 
the Whitney condition fails:
\begin{equation}
\label{equation:acbHk}
\begin{gathered}
\cL^{[H^1]}_3[[acb]] \subset \overline{\cL^{[H^1]}_3[cabca]}, \\
\cL^{[H^r]}_3[[acb]] \not\subset \overline{\cL^{[H^r]}_3[cabca]}, \qquad
\cL^{[H^r]}_3[[acb]] \cap \overline{\cL^{[H^r]}_3[cabca]} \ne \emptyset, \quad r\geq 3.
\end{gathered}
\end{equation}
Thus, in the equivalent statement to Theorem \ref{theo:poset} 
for $r\geq 3$, the equivalence between conditions 
\ref{item:inclusion} and \ref{item:intersection} 
does not hold; 
likewise, condition \ref{item:intersection} and 
\ref{item:open} are not equivalent in the $r\geq 3$ case. 
In Lemma \ref{lemma:subwordsHk} we state and prove the 
equivalence between \ref{item:intersection} and 
a version of \ref{item:subwords} for arbitrary large $r$.

As we write this paper, 
some natural and rather basic questions concerning the partial order
$\tok$ are still open; 
Conjecture \ref{conj:multtok} below
is essentially equivalent to Conjecture 2.4 
in \cite{Shapiro-Shapiro3}. 
For $\sigma \in S_{n+1}$, set 
$\mult(\sigma) = (\mult_1(\sigma), \ldots, \mult_n(\sigma)) \in \NN^n$,
where, for each $j\in\nmesmo$, we have 
\begin{equation}
\label{equation:mult}
\mult_j(\sigma) = (1^\sigma + \cdots + j^\sigma) - (1 + \cdots + j)
\in\NN=\{0,1,2,\ldots\}, 
\end{equation} 
as in Theorem 4 of \cite{Goulart-Saldanha0}.
For $w = (\sigma_1, \ldots, \sigma_\ell) \in \Word_n$, define
\[ \mult(w) = \mult(\sigma_1) + \cdots + \mult(\sigma_\ell) \in \NN^n. \]
For $u, v \in \NN^n$, we write $u \le v$ if and only if 
$u_j \le v_j$ for all $j \in \nmesmo$.

\begin{conj}
\label{conj:multtok}
Given $w_0, w_1 \in \Word_n$, 
if $w_0 \tok w_1$
then $\mult(w_0) \le \mult(w_1)$.
\end{conj}

Given $n\in\NN$, $n\geq2$, let
\[ r_{\bullet}(n) = \left\lfloor \left(\frac{n+1}{2}\right)^2 \right\rfloor =
\max\{ \mult_j(\sigma); \sigma \in S_{n+1}, j \in \nmesmo \}. \]

\begin{theo}
\label{theo:multHk}
For $w_0,w_1\in\Word_n$,
if $r > r_{\bullet}(n)$ and
$\cL_n^{[H^{r}]}[w_1]\cap
\overline{\cL_n^{[H^{r}]}[w_0]}\neq\emptyset$, 
then $\mult(w_0) \le \mult(w_1)$.
\end{theo}

Notice that these conditions imply $w_0 \tok w_1$. 
We prove Theorem \ref{theo:multHk} in Section \ref{sect:Hk}. 
In \cite{Goulart-Saldanha-cw} 
we apply an argument similar to Poincar\'e duality  
to obtain from the itinerary stratification 
(Equation \eqref{equation:stratification})  
a CW complex $\cD_n$
with a cell of dimension $\dim(w)$ for each $w \in \Word_n$
and a weak homotopy equivalence 
$c:\cD_n\to\cL_n$
(see also \cite{Goulart-Saldanha}
for an older version of this construction).
A similar finite dimensional construction is presented in 
\cite{Alves-Saldanha2} with the aim of computing  
the homotopy type of the intersections of two real Bruhat cells. 
In both cases, the Whitney condition is violated 
(see Equation \eqref{equation:acbHk}).
Theorem \ref{theo:multHk} allows us to circumvent,
in the construction of $\cD_n$,
the inconvenient fact that Conjecture \ref{conj:multtok}
remains an open problem (for $n>4$;
but see \cite{Saldanha-Shapiro-Shapiro}
and \cite{Saldanha-Shapiro-Shapiro1}
for some recent partial results).

\goodbreak

The space $\cL\Ss^2(I)\approx\cL_{2}(-1)\sqcup\cL_{2}(1)$ 
of closed nondegenerate curves in $\Ss^2$ 
was originally studied by J. Little in the 
seventies \cite{Little}, and shown to have three 
connected components. 
These are, in our notation: 
$\cL_{2}(+1)$, containing 
curves with an odd number of self-intersections 
(counted with multiplicity); 
$\cL_{2,\conv}(-1)$, 
the subspace of simple curves; and 
$\cL_{2,\nconv}(-1)$, 
containing curves with positive even number 
of self-intersections (again with multiplicity). 
The works of B. Khesin, B. Shapiro and M. Shapiro in the nineties 
$\cite{Khesin-Shapiro2, Shapiro-Shapiro, Shapiro}$ 
extended this result for $n$ and $z\in\Spin_{n+1}$ 
arbitrary, showing that $\cL_{n}(z)$ has one or 
two connected components: one if and only if it does not 
contain convex curves 
and two otherwise, one of 
them being the contractible subspace $\cL_{n,\conv}(z)$.
In \cite{Saldanha3} the spaces $\cL_{2}(z)$ were 
completely classified into three homotopy types explicitly described. 
Our approach via the CW complex $\cD_n$ 
has already allowed further progress on 
the problem of describing the homotopy 
types of the spaces $\cL_{n}(z)$ for $n>2$.  
We expect to present our new results,  
in particular solving the problem for $n=3$. 
For more information, see our final remarks in Section \ref{section:final}.

Our problem is related to the study of 
linear ordinary differential operators.
This point of view was the original motivation of V. Arnold, B. Khesin, 
V. Ovsienko, B. Shapiro and M. Shapiro for considering this class of questions 
in the early nineties 
\cite{Khesin-Ovsienko, Khesin-Shapiro1, Khesin-Shapiro2, Shapiro-Shapiro3}.
Conjectures 2.4 and 2.6 of \cite{Shapiro-Shapiro3} 
(mentioned earlier in this introduction) are related to 
an attempt at a generalized (multiplicative) Sturm 
Theory for linear ordinary differential equations of order $n+1>2$, 
the case $n=1$ standing for the classical (additive) one. 
The first of these conjectures has been proved 
for $n\leq4$ in \cite{Saldanha-Shapiro-Shapiro, Shapiro-Shapiro4}, 
but the general case remains open;  
the second one is essentially our Lemma \ref{lemma:conjecture}. 
The second author was first led to consider  
this subject while studying the critical 
sets of nonlinear differential operators with periodic 
coefficients, in a series of works with D. Burghelea 
and C. Tomei 
\cite{Burghelea-Saldanha-Tomei1,Burghelea-Saldanha-Tomei2,Burghelea-Saldanha-Tomei3,Saldanha-Tomei}. 

In Section \ref{sect:gso} we review notation and results about
Bruhat cells and related topics (Subsection \ref{subsect:Bruhat}),
Hilbert and Banach manifolds of curves (Subsection \ref{subsect:Hilbert}),
and convex curves (Subsection \ref{subsect:convex}).
Section \ref{sect:sing} is dedicated to the proof of
Theorem \ref{theo:Hausdorff}.
The concept of accessibility
is discussed in Section \ref{sect:acctriangle}
for the lower triangular group and
in Section \ref{sect:acc} for the spin group.
Section \ref{sect:paths} is dedicated to the proof of
Theorem \ref{theo:stratification}.
In Section \ref{sect:transversal} we present
transversal sections to the strata.
Section \ref{sect:poset} is dedicated to the proof of
Theorem \ref{theo:poset}.
Section \ref{sect:Hk} begins by proving Theorem \ref{theo:multHk};
we then discuss in detail the neighborhood of the stratum
defined by the permutation $acb = [3142] \in S_4$.
Some final remarks 
are given in Section \ref{section:final}:
Theorem \ref{theo:L3} (which is proved in \cite{Alves-Goulart-Saldanha})
describes the homotopy type of
the spaces $\cL_3(q)$, $q \in \Quat_4$
(locally convex curves in $\Ss^3$).

We would like to thank:  
Em\'ilia Alves, 
Boris Khesin, 
Ricardo Leite, 
Carlos Gustavo Moreira, 
Paul Schweitzer, 
Boris Shapiro, 
Michael Shapiro, 
Carlos Tomei, 
David Torres, 
Cong Zhou 
and 
Pedro Z\"{u}lkhe 
for helpful conversations
and the referee for a careful review of the text. 
We also thank
the University of Toronto and the University of Stockholm 
for the hospitality during our visits. 
Both authors thank CAPES, CNPq and FAPERJ (Brazil) for financial support.
More specifically, the first author benefited from
CAPES-PDSE grant 99999.014505/2013-04 
during his Ph. D. and also 
CAPES-PNPD post-doc grant 88882.315311/2019-01.

\section{Notations and facts} 
\label{sect:gso}

\subsection{Bruhat cells}
\label{subsect:Bruhat}

We briefly recall some definitions from \cite{Goulart-Saldanha0}. 
For $j\in\nmesmo$, set 
$\alpha_j:\RR\to\Spin_{n+1}$, 
\[\alpha_j(\theta)=\exp(\theta\fa_j), \qquad 
\fa_j=e_{j+1}e_j^\top-e_je_{j+1}^\top\in\so_{n+1}=\spin_{n+1}.\] 
The spin group $\Spin_{n+1}$ is the universal double cover
of $\SO_{n+1}$ and is contained in the Clifford algebra $\Cliff_{n+1}^{0}$.
Also, set $\acute a_j=\alpha_j(\pi/2), 
\grave a_j=(\acute a_j)^{-1}\in
\widetilde\B^+_{n+1}\subset\Spin_{n+1}$. 
Recall that the group $\widetilde\B^+_{n+1}$ is 
the lift to the spin group of the group 
$\B^+_{n+1}\subset\SO_{n+1}$ of signed permutation 
matrices with positive determinant; 
the elements $\acute a_j$ are generators 
of $\widetilde\B^+_{n+1}$.    

A \emph{reduced word} for a permutation $\sigma\in S_{n+1}$ 
is an expression $\sigma=a_{i_1}\cdots a_{i_k}$ 
of minimum length $k=\inv(\sigma)$; 
here, $a_j=(j,j+1)\in S_{n+1}$, 
$j\in\nmesmo$, are the Coxeter generators of the symmetric group. 
Given a reduced word as above, set   
\begin{equation}
\label{equation:acutegravei}
\begin{gathered}
\acute \sigma= \longacute(\sigma) = \acute a_{i_1}\cdots \acute a_{i_k},
\quad
\grave \sigma= \longgrave(\sigma) =
\grave a_{i_1}\cdots \grave a_{i_k}\in\widetilde\B^+_{n+1}, \\
\hat \sigma=\longhat(\sigma) = \acute\sigma(\grave\sigma)^{-1}\in\Quat_{n+1}\subset\widetilde\B^+_{n+1}.
\end{gathered}
\end{equation}
Let $\vartriangleleft$ be the covering relation for the Bruhat order; 
thus, $\sigma_0\vartriangleleft\sigma_1$ implies 
$\inv(\sigma_1)=1+\inv(\sigma_0)$. 
An important special case is when 
$\sigma_{j-1}\vartriangleleft\sigma_j=\sigma_{j-1} a_{i_j}$: 
in this case, there exists reduced words 
$\sigma_{j-1}=a_{i_1}\cdots a_{i_{j-1}}$ and 
$\sigma_j=a_{i_1}\cdots a_{i_{j-1}}a_{i_j}$.

Let $\Lo^1_{n+1}$ be the group of lower triangular matrices with 
unit diagonal entries and $\lo^1_{n+1}$ be its Lie algebra. 
Notice that $\Lo^1_{n+1}$ is nilpotent and contractible.
Let $\cU_I\subset\SO_{n+1}$ be the open 
contractible set of orthogonal matrices 
$Q\in\SO_{n+1}$ that 
admit an $LU$ decomposition $Q=LU$ with 
$L\in\Lo^1_{n+1}$ and $U\in\Up^+_{n+1}$. 
Here, $\Up^+_{n+1}$ is the group of upper triangular 
matrices with positive diagonal entries. 
Also, let $\cU_1\subset\Spin_{n+1}$ be 
the connected component of the identity 
in the subset $\Pi^{-1}[\cU_I]\subset\Spin_{n+1}$. 
The diffeomorphism $\bL:\cU_1\to\Lo^1_{n+1}$ 
is defined by taking the $L$-part $\bL(z)$ in the 
$LU$ decomposition of the matrix $\Pi(z)$.
Each $z_0\in\Spin_{n+1}$ has an open neighborhood 
$\cU_{z_0}=z_0\,\cU_1\subset\Spin_{n+1}$ diffeomorphic to $\Lo^1_{n+1}$. 
We also consider the set of matrices 
$z_0\Lo^1_{n+1}=\Pi(z_0)\Lo^1_{n+1}\subset\GL^+_{n+1}$  
and the diffeomorphism 
$\bL_{z_0}:\cU_{z_0}\to z_0\Lo^1_{n+1}$, 
$\bL_{z_0}(z)=z_0\bL(z_0^{-1}z)$, with inverse 
$\bQ_{z_0} = \bL_{z_0}^{-1}:z_0\Lo^1_{n+1}\to\cU_{z_0}$, 
where $\bQ_{z_0}(M)\in\cU_{z_0}$ is the lift to $\Spin_{n+1}$ 
of the $Q$-part of the $QR$ decomposition of $M$. 
We are particularly interested in the case 
$z_0\in\widetilde\B^+_{n+1}$. 
In this situation, 
the matrices in $z_0\Lo^1_{n+1}$ are, up to signs, triangular 
matrices with rows shuffled by the underlying permutation 
of $z_0$. 
We call $(\cU_{z_0},\bL_{z_0})$ 
a \emph{triangular system of coordinates}. 
When $z_0=1$, we write simply $\bQ=\bQ_1$, 
in accordance with $\bL=\bL_1$. 

For each $j\in\nmesmo$, the Lie algebra element 
$\fa_j\in\spin_{n+1}$ is taken to a positive multiple of 
$\fl_j=e_{j+1}e_j^\top\in\lo^1_{n+1}$ 
by the derivative of the map $\bL$ 
at the identity (Lemma 4.1 of \cite{Goulart-Saldanha0}). 
Moreover, the arcs of the curves $\alpha_j$ contained in $\cU_1$ 
are taken by $\bL$ into (orientation-preserving) 
reparameterizations of $\jacobi_j(t)=\exp(t\fl_j)$, $t\in\RR$. 
We say that an absolutely continuous 
map $\Gamma:J\to\Lo^1_{n+1}$, 
defined on an interval $J\subseteq\RR$, 
is a \emph{convex curve} if and only if 
its logarithmic derivative is given almost everywhere by 
\[(\Gamma(t))^{-1}\Gamma'(t)=\sum_{j\in\nmesmo}\beta_j(t)\fl_j, \]
for positive functions $\beta_1,\ldots,\beta_n:J\to(0,+\infty)$. 
In other words, $\Gamma:J\to\Lo^1_{n+1}$ is a convex curve 
if and only if $\bQ\circ\Gamma:J\to\Spin_{n+1}$ is 
a locally convex curve. 
Notice that $(\bQ\circ\Gamma)[J]\subset\cU_1$; 
in Subsection \ref{subsect:convex} we review the fact that a 
locally convex curve $\Gamma:J\to\Spin_{n+1}$ is 
strictly convex if and only if 
$\Gamma[J]\subset\cU_{z_0}$ for some $z_0\in\Spin_{n+1}$. 
Notice that if $\Gamma:[t_0,t_1]\to\Spin_{n+1}$ is strictly convex 
then $\Gamma$ is globally convex, i.e., 
$\sing((\Gamma(t_0))^{-1}\Gamma)=\emptyset$; 
the reciprocal is not quite true.

Some distinguished Lie algebra elements are 
\begin{equation}
\label{equation:nfhLfh}
\begin{gathered}
\fn=\sum_{j\in\nmesmo}\fl_j,\quad
\fh_L=\sum_{j\in\nmesmo}\sqrt{j(n+1-j)}\;\fl_j\in\lo^1_{n+1}, \\
\fh=\sum_{j\in\nmesmo}\sqrt{j(n+1-j)}\;\fa_j\in\spin^1_{n+1}. 
\end{gathered}
\end{equation}
For arbitrary elements $g_0\in G$,  
$\mathfrak{v}\in\mathfrak{g}$ of 
a Lie group and its Lie algebra, denote by 
$\Gamma_{g_0;\mathfrak{v}}:\RR\to G$ 
the smooth parametric curve 
$\Gamma_{g_0;\mathfrak{v}}(t)=g_0\exp(t\mathfrak{v})$. 
The smooth curves 
$\Gamma_{L_0;\fn}$, $\Gamma_{L_0;\fh_L}$ 
and $\Gamma_{z_0;\fh}$,  
studied in Example 4.2 of 
\cite{Goulart-Saldanha0}, are particularly useful. 
The first two are convex and the third one is locally convex. 

We denote by $\Pos_\eta\subset\Lo^1_{n+1}$ the 
open subset of totally positive matrices 
\cite{Berenstein-Fomin-Zelevinsky}. 
For a reduced word $\eta=a_{i_1}\cdots a_{i_m}$ 
($m=n(n+1)/2$) for the Coxeter element of $S_{n+1}$, 
the map 
$(0,+\infty)^m \to \Pos_\eta$,
$(t_1,\ldots,t_m)
\mapsto\jacobi_{i_1}(t_1)\cdots\jacobi_{i_m}(t_m)$, 
is a diffeomorphism. 
More generally, there are embedded submanifolds 
$\Pos_\sigma, \Neg_\sigma\subset\Lo^1_{n+1}$, 
$\sigma\in S_{n+1}$, 
such that, given a reduced word $\sigma=a_{i_1}\cdots a_{i_k}$, 
$k=\inv(\sigma)$, the maps 
$(0,+\infty)^k \to \Pos_\sigma$,
$(t_1,\ldots,t_k)
\mapsto\jacobi_{i_1}(t_1)\cdots\jacobi_{i_k}(t_k)$, 
and 
$(-\infty,0)^k \to \Neg_\sigma$,
$(t_1,\ldots,t_k)
\mapsto\jacobi_{i_1}(t_1)\cdots\jacobi_{i_k}(t_k)$, 
are diffeomorphisms. 
We have 
\[
\overline\Pos_\eta=\bigsqcup_{\sigma\in S_{n+1}}\Pos_\sigma, 
\quad
\overline\Neg_\eta=\bigsqcup_{\sigma\in S_{n+1}}\Neg_\sigma, 
\quad \overline\Pos_\eta\cap\overline\Neg_\eta=\{I\} = \Pos_{e} = \Neg_{e}. 
\]
This is closely related to the Bruhat stratifications:
\[
\Spin_{n+1}=\bigsqcup_{\sigma\in S_{n+1}}\Bru_\sigma, \qquad
\Bru_\sigma 
=\bigsqcup_{q\in\Quat_{n+1}}\Bru_{q\acute\sigma}.
\]
Recall that $z\in\Bru_\sigma$ if and only if there exist upper triangular 
matrices $U_0,U_1$ 
such that $\Pi(z)=U_0\Pi(\acute\sigma)U_1$. 
We have 
$\Bru_{q\acute\sigma}=\cU_{q\acute\sigma}\cap\Bru_\sigma$.
Also, given a reduced word $\sigma=a_{i_1}\cdots a_{i_k}$, 
$k=\inv(\sigma)$, the map 
$(0,\pi)^k \to \Bru_{q\acute\sigma}$,
$(\theta_1,\ldots,\theta_k) \mapsto
q\alpha_{i_1}(\theta_1)\cdots\alpha_{i_k}(\theta_k)$, 
is a diffeomorphism 
(Corollary 1.2 of \cite{Goulart-Saldanha0}). 
For all $\sigma\in S_{n+1}$ and $q\in\Quat_{n+1}$,
the set $q\bQ[\Pos_\sigma]$
is a contractible connected component of the submanifold
$\cU_q\cap\Bru_{q\acute\sigma}$.
Similarly, $q\bQ[\Neg_\sigma]$
is a contractible connected component of 
$\cU_q\cap\Bru_{q\grave\sigma}$.
We have $q\grave\sigma=\tilde q\acute\sigma
\in\widetilde\B^+_{n+1}$, $\tilde q=q\hat\sigma^{-1}\in\Quat_{n+1}$. 

For $L_0,L_1\in\Lo^1_{n+1}$, we write $L_0\ll L_1$ if and only if 
$L_0^{-1}L_1\in\Pos_\eta$ 
(equivalently, $L_1^{-1}L_0\in\Neg_\eta$) and 
$L_0\leq L_1$ if and only if $L_0^{-1}L_1\in\overline\Pos_\eta$ 
(equiv., $L_1^{-1}L_0\in\overline\Neg_\eta$). 
These are partial orders in $\Lo_{n+1}^1$ 
(Lemma 5.2 of \cite{Goulart-Saldanha0}). 
We have $L_0\ll L_1$ if and only if 
there is a convex curve $\Gamma:[0,1]\to\Lo^1_{n+1}$ 
satisfying $\Gamma(0)=L_0$ and $\Gamma(1)=L_1$ 
(Lemma 5.3 of \cite{Goulart-Saldanha0}). 
Convex curves $\Gamma:J\to\Lo^1_{n+1}$ are such that, for $t_0<t<t_1$ in $J$, we have 
$\Gamma(t)\in(\Gamma(t_0)\Pos_\eta)\cap(\Gamma(t_1)\Neg_\eta)$ 
(Lemma 5.7 of \cite{Goulart-Saldanha0}). 

\emph{Projective transformations} are 1-1
correspondences between (locally) convex curves that 
preserve itineraries and singular sets. 
We consider two types of them: 
\begin{enumerate}
\item\label{item:projUp}
{Given an upper triangular matrix $U$ 
with positive diagonal entries, we assign to each  
locally convex curve $\Gamma:[t_0,t_1]\to\Spin_{n+1}$ its 
projective transform $\Gamma^U:[t_0,t_1]\to\Spin_{n+1}$ 
given by $\Gamma^U(t)=
\bQ(U^{-1}\Gamma(t))$;}
\item\label{item:projDiag}
{Given $\lambda>0$, consider the diagonal matrix 
$E_\lambda=\diag(1,\lambda,\ldots,\lambda^n)$. 
We assign to each convex curve $\Gamma:[t_0,t_1]\to\Lo^1_{n+1}$ 
its projective transform $\Gamma^\lambda:[t_0,t_1]\to\Lo^1_{n+1}$  given by 
$\Gamma^\lambda(t)=E_\lambda^{-1}\Gamma(t)E_\lambda$.}
\end{enumerate}

Projective transformations come from the 
smooth actions of Lie groups: 
\begin{align*}
\Spin_{n+1}\times\Up^+_{n+1}\to\Spin_{n+1}, &\quad
(z,U)\mapsto z^U=\bQ(U^{-1}z), \\ 
\Lo_{n+1}^{1}\times(0,+\infty)\to\Lo^1_{n+1}, &\quad
(L,\lambda)\mapsto L^\lambda=E_\lambda^{-1}LE_\lambda, 
\end{align*}
We abuse the distinction between $z\in\Spin_{n+1}$ and 
$\Pi(z)\in\SO_{n+1}$ in the first formula, so that 
$\bQ(U^{-1}z)$ is the lift to the spin group of the $Q$-part in the 
$QR$ factorization of the invertible matrix $U^{-1}\Pi(z)$. 
Both these actions preserve signed Bruhat cells 
$\Bru_{q\acute\sigma}$ 
(we consider $\bL[\cU_1\cap\Bru_{q\acute\sigma}]$ 
as the corresponding signed Bruhat cell in $\Lo^1_{n+1}$); 
the subgroup $\Up^1_{n+1}\subset\Up^+_{n+1}$ 
of matrices with unit diagonal entries 
acts transitively on each signed Bruhat cell. 
See Section 6 in \cite{Goulart-Saldanha0}. 

In projective transformations of type \ref{item:projUp}, 
the lift to $\Spin_{n+1}$ is made in such a way that,
for each $t$, $\Gamma(t)$ and $\Gamma^U(t)$ 
are in the same signed Bruhat cell $\Bru_{q\acute\sigma}$. 
Also notice that if $\Gamma(t_\ast)\in\Bru_{q\acute\sigma}$, 
then, for each $z\in\Bru_{q\acute\sigma}$, there is a 
projective transformation of type \ref{item:projUp} such that 
$\Gamma^U(t_\ast)=z$. 
Moreover, the matrix $U$ can always be taken in the 
subgroup $\Up^1_{n+1}\subset\Up^+_{n+1}$ 
of upper triangular matrices with unit diagonal entries; 
in type \ref{item:projDiag}, notice that, for all $t$, we have 
$\lim_{\lambda\to+\infty}\Gamma^\lambda(t)=I$. 

The maps $\chop, \adv: \Spin_{n+1} \to \acute\eta \Quat_{n+1} \subset
\widetilde \B_{n+1}^{+}$
are defined by 
\begin{equation}
\label{equation:chopadvancei}
\adv(z)=q_a \acute\eta, \quad
\chop(z) = q_c \grave\eta, \quad
z\in\Bru_{z_0}\subset\Bru_{\sigma_0}, \quad 
z_0=q_a \acute\sigma_0=q_c \grave\sigma_0,
\end{equation}
where $z_0\in\widetilde\B^+_{n+1}$, 
$\sigma_0 \in S_{n+1}$ and $q_a,q_c\in\Quat_{n+1}$.  
For  $\rho_0=\eta\sigma_0$, we have
$\adv(z)=\chop(z)\hat\rho_0$.
Given a locally convex curve $\Gamma:J\to\Spin_{n+1}$, 
for each $t\in J$, there is $\epsilon>0$ such that 
$\Gamma[(t-\epsilon,t)]\in\Bru_{\chop(\Gamma(t))}$ and 
$\Gamma[(t,t+\epsilon)]\in\Bru_{\adv(\Gamma(t))}$ 
(Theorem 3 of \cite{Goulart-Saldanha0});  
notice that these are open signed Bruhat cells.

\subsection{Hilbert and Banach manifolds of curves}
\label{subsect:Hilbert}

There are many possible choices for the exact definition 
and topology of our spaces of locally convex curves. 
Notice that there are sections with a similar purpose in 
\cite{Saldanha3, Saldanha-Shapiro, Saldanha-Zuhlke1}. 
We also plan to discuss this subject 
in greater detail and generality in \cite{gsie}.


Given $r\in\NN=\{0,1,2,\ldots\}$ 
and a finite dimensional real vector space $V$,
we are interested in the Banach spaces $C^r(V) = C^r([0,1];V)$
and the Hilbert space $H^r(V) = H^r([0,1];V)$.
The space $H^r(V)$ consists of functions
$f:[0,1]\to V$ of Sobolev class $H^r=W^{r,2}$. 
In more detail, for $r\geq1$, we have $f\in H^r(V)$ if 
$f:[0,1]\to V$ is of class $C^{(r-1)}$, 
its $(r-1)$-th derivative $f^{(r-1)}$ is absolutely continuous 
and its $r$-th derivative $f^{(r)}$ (defined a.e.) 
is a function of class $L^2$. 
We follow the convention that $H^0(V)=L^2([0,1],V)$. 


Consider a compact Lie group $G$ contained
in a finite dimensional associative algebra $A$
(such as $G = \SO_{n+1} \subset A = \RR^{(n+1)\times(n+1)}$
or $G = \Spin_{n+1} \subset A = \Cliff_{n+1}^{0}$).
Define $C^r(A)$ and $H^r(A)$ as above.
The spaces
$C^r(G) = C^r([0,1];G) \subset C^r(A)$ and
$H^r(G) = H^r([0,1];G) \subset H^r(A)$ 
of functions whose images are contained in $G$ 
are Banach and Hilbert submanifolds, respectively.
Notice that
the maps $\Gamma \mapsto \Gamma(t_0)$, $t_0 \in [0,1]$,
are smooth surjective submersions onto $G$.

Let $T\subset\fg$ be a vector subspace 
which generates $\fg$ (as a Lie algebra).
An absolutely continuous 
curve $\Gamma:[0,1]\to G$ 
is \emph{$T$-holonomic} 
if $(\Gamma(t))^{-1}\Gamma'(t)\in T$ whenever
$\Gamma'(t)$ is defined
(which is almost always).
For $r \ge 1$,
the subsets $C^r(G;T) \subset C^r(G)$
and $H^r(G;T) \subset H^r(G)$ 
of $T$-holonomic curves are smooth 
Banach and Hilbert submanifolds, respectively.
In our example, $G = \Spin_{n+1}$ and
$T$ is spanned by $\fa_1,\ldots,\fa_n$. 
Let $\cJ \subset T$ be the open cone
of linear combinations of  $\fa_1,\ldots,\fa_n$ 
with positive coefficients. 
A holonomic curve is \emph{locally convex}
if $(\Gamma(t))^{-1}\Gamma'(t)\in \cJ$ whenever defined.
This defines open subsets
$\cL_n^{[C^r]}(\cdot\,;\cdot) \subset C^r(G;T)$ for $r \ge 1$
and
$\cL_n^{[H^r]}(\cdot\,;\cdot) \subset H^r(G;T)$ for $r \ge 2$
of locally convex curves.
These are Banach and Hilbert manifolds, respectively.
The case $H^1$ requires a more delicate discussion and is postponed.

If we fix the initial point, we have submanifolds
$\cL_n^{[C^r]}(z_0;\cdot) \subset \cL_n^{[C^r]}(\cdot\,;\cdot)$ and
$\cL_n^{[H^r]}(z_0;\cdot) \subset \cL_n^{[H^r]}(\cdot\,;\cdot)$.
Indeed, multiplication by $z \in \Spin_{n+1}$ shows that $\Gamma(0)$
as a function of $\Gamma \in \cL_n(\cdot\,;\cdot)$ is a submersion.
In these submanifolds the map $\mu$, $\mu(\Gamma) = \Gamma(1)$,
is called {\em monodromy}.
Lemma \ref{lemma:submersion} below shows that likewise
$\cL_n(z_0;z_1) \subset \cL_n(z_0;\cdot)$
is a submanifold.


Recall from Equation \eqref{equation:locallyconvex}
that a locally convex curve is characterized
by its initial point $z_0 = \Gamma(0)$
and by the positive functions $\kappa_j$, $j \in \nmesmo$.
If $\Gamma$ is of class $H^r$, $r > 1$,
the functions $\kappa_j$ are in $H^{r-1}([0,1];\RR)$.
We could use this system of coordinates to produce an alternative
description of the Hilbert manifold structure
of the space $\cL_n^{[H^r]}(\cdot\,;\cdot)$.
We follow a similar method to define $\cL_n^{[H^1]}(\cdot\,;\cdot)$.
The functions $\kappa_j$ would then be in $H^0 = L^2$:
the difficulty is that the set of positive functions
is not open in $L^2([0,1];\RR)$.
We circumvent this difficulty by defining functions
$\xi_j\in H^0$ by
\begin{equation}
\label{equation:xis}
\xi_j=\kappa_j-\dfrac{1}{\kappa_j},\qquad 
\kappa_j=\dfrac{\xi_j+\sqrt{\xi_j^{2}+4}}{2}.
\end{equation}
By definition, a locally convex curve $\Gamma$
(assumed to be absolutely continuous)
is in $\cL_n^{[H_1]}(\cdot\,;\cdot)$
if the corresponding functions $\xi_j$
(through Equations \eqref{equation:locallyconvex}
and \eqref{equation:xis})
are in $H^0 = L^2([0,1])$.
Conversely, given functions $\xi_j\in H^0$, 
Equation \eqref{equation:xis} above yields positive functions $\kappa_j$:
Equation \eqref{equation:locallyconvex} is then an ODE,
defining $\Gamma$ (see \cite{gsie} for details). 
Thus, the initial point $z_0$ and the functions $\xi_j$ 
give a smooth Hilbert manifold structure to the space $\cL_n^{[H^1]}(\cdot\,;\cdot)$. 
As in the case $r\geq 2$, 
$\cL_n^{[H^1]}(z_0;\cdot) \subset \cL_n^{[H^1]}(\cdot\,;\cdot)$ 
is clearly a submanifold for each $z_0\in\Spin_{n+1}$. 

\begin{lemma}
\label{lemma:submersion}
Consider a topology $C^r$, $r > 1$, $H^1$ or $H^r$, $r > 2$, and 
the monodromy map $\mu: \cL_n(z_0;\cdot) \to \Spin_{n+1}$,
$\mu(\Gamma) = \Gamma(1)$, for a fixed $z_0\in\Spin_{n+1}$.
The map $\mu$ is a surjective submersion.
In particular, $\cL_n(z_0;z_1) \subset \cL_n(z_0;\cdot)$
is a nonempty submanifold for each $z_1\in\Spin_{n+1}$.
\end{lemma}

The case $n = 2$ is discussed in \cite{Saldanha3}
and that proof can be adapted by using some basic facts
about total positivity.

\begin{proof}
We use spaces of convex curves
in the triangular group $\Lo_{n+1}^1$.
If the image of a locally convex arc $\Gamma$
is contained in $\cU_{z_1}$, $z_1 \in \Spin_{n+1}$,
then $t \mapsto \bL(z_1^{-1} \Gamma(t))$ is a convex curve.

Convex curves if $\Lo_{n+1}^1$ follow explicit formulae.
In particular, it follows from Equation (10)
in \cite{Goulart-Saldanha0} that,
given a convex curve in $\Lo_{n+1}^1$,
there exist perturbations keeping
the first half of the arc fixed
and moving the end point in any prescribed direction.

Given $\Gamma \in \cL_n(z_0;\cdot)$, set $z_1 = \Gamma(1)$.
Set $t_{\frac12} \in (0,1)$ such that the image of the arc
$\Gamma_{[t_{\frac12},1]}$ is contained in $\cU_{z_1}$.
Identify $\cU_{z_1}$ with $\Lo_{n+1}^1$, as above.
The perturbation in $\Lo_{n+1}^1$ can be brought back to $\Gamma$,
proving that $\mu$ is a submersion.
Surjectivity follows by adding loops,
as in \cite{Saldanha-Shapiro}.
\end{proof}

The inclusion 
$\cL_n^{[H^{r+1}]}(z_0;\cdot)\hookrightarrow\cL_n^{[H^r]}(z_0;\cdot)$ 
is continuous with dense image for all $r\geq1$. 
The same happens for 
$\cL_n^{[C^{r+1}]}(z_0;\cdot)\hookrightarrow\cL_n^{[C^r]}(z_0;\cdot)$ 
and
$\cL_n^{[C^r]}(z_0;\cdot)\hookrightarrow\cL_n^{[H^r]}(z_0;\cdot)$.
The following facts imply that these are also homotopy equivalences. 


\begin{fact}[Theorem 2 of \cite{Burghelea-Saldanha-Tomei1}]
\label{fact:BST}
Let $\mathbf{B}_1$ and $\mathbf{B}_2$ be infinite dimensional separable Banach
spaces. Suppose $i:\mathbf{B}_1\to \mathbf{B}_2$ is a bounded, injective linear
map with dense image and $M_2\subset \mathbf{B}_2$ is a smooth closed Banach
submanifold of finite codimension. Then, $M_1=i^{-1}[M_2]$ is a smooth closed
Banach submanifold of $\mathbf{B}_1$ and $i:(\mathbf{B}_1,M_1)\to
(\mathbf{B}_2,M_2)$ is a homotopy equivalence of pairs.
\end{fact}

\begin{fact}[from Theorem 0.1 of \cite{Burghelea-Henderson} and
Corollary 3 of \cite{Henderson}]
\label{fact:BH}
Let $M_1$ and $M_2$ be topological (respectively, smooth)
manifolds modeled on infinite dimensional separable Banach (resp. Hilbert)
spaces.  Any homotopy equivalence $i:M_1\to M_2$ is homotopic to a
homeomorphism (resp., diffeomorphism).
\end{fact}

\begin{rem}
\label{rem:nofibration}
A natural question at this point would be whether 
$\mu$ qualifies as some sort of fibration. 
The reader of course knows that the spaces $\cL_n(z)$ exhibit 
different homotopy types as $z$ ranges over $\Spin_{n+1}$
\cite{Little, Khesin-Shapiro2, Saldanha3, Saldanha-Shapiro,
Shapiro-Shapiro, Shapiro}. 
In fact, $\mu$ is not even a Serre fibration, 
since it lacks the homotopy lifting property 
for polyhedra (see \cite{Khesin-Shapiro2, Saldanha1}).
\end{rem}

\begin{lemma}
\label{lemma:spaces}
For all $z_0,z_1\in\Spin_{n+1}$, we have that:
\begin{enumerate}
\item\label{item:lemma:spaces:submanifold}
{for all $r,r'\in\NN^\ast$, 
$r\neq 2$, $r'\neq 1$, 
the subspaces $\cL^{[H^r]}_n(z_0;z_1)$ 
and $\cL^{[C^{r'}]}_n(z_0;z_1)$ are closed embedded smooth 
submanifolds of codimension $m=n(n+1)/2$ 
of $\cL^{[H^r]}_n(z_0;\cdot)$ and 
$\cL^{[C^{r'}]}_n(z_0;\cdot)$, respectively;}
\item\label{item:lemma:spaces:homotopyequivalence}
{for all $r,\tilde r\in\NN$, $r\geq 1$, the natural inclusion maps  
\begin{gather*}
i_{r,\tilde r}:
(\cL^{[H^{r+\tilde r}]}_n(z_0;\cdot),\cL^{[H^{r+\tilde r}]}_n(z_0;z_1))
\hookrightarrow
(\cL^{[H^r]}_n(z_0;\cdot),\cL^{[H^r]}_n(z_0;z_1)), \\  
j_{r,\tilde r}:
(\cL^{[C^{r+\tilde r}]}_n(z_0;\cdot),\cL^{[C^{r+\tilde r}]}_n(z_0;z_1))
\hookrightarrow(\cL^{[H^r]}_n(z_0;\cdot),\cL^{[H^r]}_n(z_0;z_1)), \\
\ell_{r,\tilde r}:
(\cL^{[C^{r+\tilde r}]}_n(z_0;\cdot),\cL^{[C^{r+\tilde r}]}_n(z_0;z_1))
\hookrightarrow(\cL^{[C^r]}_n(z_0;\cdot),\cL^{[C^r]}_n(z_0;z_1))
\end{gather*} 
are homotopy equivalences of pairs.}
\item\label{item:lemma:spaces:homeomorphism}
{each of the natural inclusions $i_{r,\tilde r}$,
$j_{r,\tilde r}$, $\ell_{r,\tilde r}$ 
of item \ref{item:lemma:spaces:homotopyequivalence} 
is homotopic to a homeomorphism between the respective pairs 
(a diffeomorphism for $i_{r,\tilde r}$).   
}
%
\end{enumerate}
\end{lemma}

\begin{proof}
Item \ref{item:lemma:spaces:submanifold} follows 
directly from Lemma \ref{lemma:submersion}. 
For $r,n\in\NN^\ast$, consider the Banach spaces 
$\mathbf{H}^{r,n} = (H^{r-1}([0,1]);\RR))^n$ and  
$\mathbf{C}^{r,n} = (C^{r-1}([0,1]);\RR))^n$. 
Regard $\cL_n^{[H^r]}(z_0;z_1)\subset\cL_n^{[H^r]}(z_0;\cdot)$ 
as submanifolds of $\mathbf{H}^{r,n}$ and 
$\cL_n^{[C^r]}(z_0;z_1)\subset\cL_n^{[C^r]}(z_0;\cdot)$ 
as submanifolds of $\mathbf{C}^{r,n}$, 
given by the functions $\xi_1,\ldots,\xi_n$ of Equation 
\eqref{equation:xis}. 
Item 
\ref{item:lemma:spaces:homotopyequivalence} now follows from 
Fact \ref{fact:BST}, 
by choosing $\mathbf{B}_1$ and $\mathbf{B}_2$ 
amongst the Banach spaces above. 
Item \ref{item:lemma:spaces:homeomorphism} 
follows from Item \ref{item:lemma:spaces:homotopyequivalence} and Fact \ref{fact:BH}.
\end{proof}

In particular, we see that, given $z_0,z_1\in\Spin_{n+1}$, 
all spaces $\cL_n^\ast(z_0;z_1)$ are homeomorphic. 
In some situations, this warrants us the right to 
drop the superscripts altogether 
and to adopt a definition of $\cL_n(z_0;z_1)$ 
that is well-suited to the purpose at hand.
Throughout this paper, 
the spaces of locally convex curves of class $H^r$ 
take precedence over their $C^r$ counterparts  
for being Hilbert manifolds. 
We are particularly interested in
$\cL_n^{[H^1]}(z_0;z_1)$ and in
$\cL_n^{[H^r]}(z_0;z_1)$ for large $r$.

\subsection{Convex curves}
\label{subsect:convex}

A smooth parametric curve $\gamma:J\to\Ss^{n}$ 
defined on a compact interval $J\subset\RR$ is said 
to be \emph{strictly convex} if for each nonzero linear functional 
$\omega\in(\RR^{n+1})^{\ast}\smallsetminus\{0\}$ 
the function $\omega\gamma:J\to\RR$,
$(\omega\gamma)(t) = \omega(\gamma(t))$,
has at most $n$ zeroes counted with multiplicities 
(zeroes at endpoints taken into account).  
It is said to be \emph{convex} if its restriction to any proper compact 
subinterval of $J$ is strictly convex.

In other words, a convex curve is one that 
(possibly neglecting one endpoint at a time) 
intersects each $n$-dimensional vector subspace 
$V\subset\RR^{n+1}$ at most $n$ times 
with multiplicities taken into account. 
Thus, for instance,
a transversal intersection counts as $1$; a generic tangency counts as $2$; 
a generic osculation counts as $3$.
Other terms used for the same or closely related 
concepts are \emph{non-oscillatory curves} \cite{Novikov-Yakovenko, Shapiro} 
and \emph{disconjugate curves}
\cite{Khesin-Shapiro2, Shapiro-Shapiro, Shapiro-Shapiro3}.

In Appendix A in \cite{Goulart-Saldanha}
we show that a smooth nondegenerate curve 
$\gamma:[0,1]\to\Ss^n$ with initial frame 
$\Frenet_{\gamma}(0)=1$ is convex 
if and only if its itinerary is the empty word, \textit{i.e.}, 
that the notion of convexity introduced in 
Section \ref{sect:sing} and given in terms of the
singular set of $\Frenet_{\gamma}$ coincides with this
geometric definition.
These results are essentially present in \cite{Shapiro}.


Clearly, convexity implies nondegeneracy. 
Conversely, as we shall see in Lemma \ref{lemma:convex},
(smooth) nondegeneracy implies local convexity:
this is why the terms \emph{nondegenerate} and 
\emph{locally convex} are used interchangeably. 


We quote below the main result of Appendix A of \cite{Goulart-Saldanha}.
For $J$ a compact interval,
we say that a locally convex curve $\Gamma:J\to\Spin_{n+1}$ 
is \emph{short} if there exists $z \in \Spin_{n+1}$ such that
$\Gamma[J]\subset\cU_z$. 
Recall that $\cU_z \subset \Spin_{n+1}$ is the domain
of a triangular system of coordinates 
(see Subsection \ref{subsect:Bruhat} or Section 4 of \cite{Goulart-Saldanha0}).

\begin{lemma}
\label{lemma:convex}
Let $\gamma:J\to\Ss^n$ be a smooth nondegenerate curve 
defined on a compact interval $J\subset\RR$. 
The following conditions are equivalent:
\begin{enumerate}
\item\label{item:strictlyconvex}{$\gamma$ is strictly convex;}
\item\label{item:short}{$\Frenet_{\gamma}$ is short;}
\item\label{item:BruhatAT}{$\forall t_{0},t_{+}\in J 
\left((t_{0}<t_{+})\rightarrow 
(\Frenet_{\gamma}(t_{+})\in \Frenet_{\gamma}(t_{0}) \Bru_{\acute\eta})\right)$;}
\item\label{item:BruhatA}{$\forall t_{0},t_{-}\in J 
\left((t_{-}<t_{0})\rightarrow 
(\Frenet_{\gamma}(t_{-})\in \Frenet_{\gamma}(t_{0}) \Bru_{\grave\eta})\right)$.}
\end{enumerate}
\end{lemma}

\begin{proof}
See Appendix A in \cite{Goulart-Saldanha};
closely related sufficient conditions for convexity may be 
found in \cite{Novikov-Yakovenko, Shapiro}. 
\end{proof}

\begin{example}
\label{example:convex}
Given $z\in\Bru_{\acute\eta}$, consider the curve 
$\Gamma=\Gamma_z:[0,1]\to\Spin_{n+1}$ 
passing through $\Gamma_z(\frac12)=z$ given by 
Lemma 6.1 of \cite{Goulart-Saldanha0}. 
It follows immediately from Lemma \ref{lemma:convex} 
that $\gamma_z=\Gamma_z e_1$ is a convex curve
(though not strictly convex). 
For an alternative proof, recall that $\Gamma_z$ 
is obtained from $\Gamma_{\acute\eta}(t)=\exp(\pi t\fh)$ 
through a projective transformation and see Lemma 2.2 of 
\cite{Saldanha-Shapiro} for a direct proof of the 
convexity of $\Gamma_{\acute\eta}$.
For $n=2$, 
$\gamma_{\acute\eta}(t)=
\frac{1}{2}(1+\cos(2\pi t),\sqrt{2}\sin(2\pi t),1-\cos(2\pi t))$ 
is the circle of diameter $e_{1}e_{3}$ in $\Ss^2$. 
Notice that $\gamma_{\acute\eta}$ is closed 
if and only if $n$ is even
(as usual, a curve
$\gamma:[0,1]\to\Ss^n$ is closed if
$\Frenet_{\gamma}(0)=\Frenet_{\gamma}(1)$).
\end{example}

\section{Proof of Theorem \ref{theo:Hausdorff}}
\label{sect:sing}

Before proving Theorem \ref{theo:Hausdorff}, 
we state and prove two related results. 
The first lemma is essentially equivalent to 
Conjecture 2.6 in \cite{Shapiro-Shapiro3}. 
We thank B. Shapiro and M. Shapiro for insightful 
conversations on this subject.

\begin{lemma}
\label{lemma:conjecture}
Let $\Gamma \in \cL_n(z_0;z_1)$ be a locally convex curve.
Let $t_\ast \in \sing(\Gamma) \subset (0,1)$, $z_\ast = \Gamma(t_\ast) \in \Sing_{n+1}$.
Then there exists an open set $\cU_\ast \subset \Spin_{n+1}$, $z_\ast \in \cU_\ast$
with the following properties.
There exists $\epsilon > 0$ such that $[t_\ast - \epsilon, t_\ast + \epsilon] \subset (0,1)$,
$[t_\ast - \epsilon, t_\ast + \epsilon]  \cap \sing(\Gamma) = \{t_\ast\}$
and $\Gamma[[t_\ast-\epsilon,t_\ast+\epsilon]] \subset \cU_\ast$.
There exist distinct open connected components 
$\cU_\ast^{-}$ and $\cU_\ast^{+}$
of $\cU_\ast \smallsetminus \Sing_{n+1}$ such that 
$\Gamma[[t_\ast - \epsilon,t_\ast)] \subset \cU_\ast^{-}$ and
$\Gamma[(t_\ast,t_\ast+\epsilon]] \subset \cU_\ast^{+}$.
\end{lemma}

\begin{proof}
Assume that $z_{\ast} \in \Bru_\rho$, $\rho \ne \eta$. 
After applying a projective transformation, we may assume
that $z_{\ast} \in q_{\ast} \bQ[\Pos_\rho] \subset \cU_{q_\ast}$, 
$q_\ast\in\Quat_{n+1}$.
We shall take $\cU_\ast = \cU_{q_\ast}$
and $\cU_\ast^{+} = q_{\ast} \bQ[\Pos_\eta]$,
which is a connected component of 
$\cU_\ast \smallsetminus \Sing_{n+1}$. 
The number $\epsilon > 0$ can easily be chosen so as to satisfy the conditions in the statement.
It follows from Lemma 5.7 
of \cite{Goulart-Saldanha0} 
that $\Gamma[(t_\ast,t_\ast+\epsilon]] \subset \cU_\ast^{+}$
and that $\Gamma[[t_\ast - \epsilon,t_\ast)]$ is disjoint from $\cU_\ast^{+}$:
let $\cU_\ast^{-}$ be the connected component of $\cU_\ast \smallsetminus \Sing_{n+1}$
containing $\Gamma[[t_\ast - \epsilon,t_\ast)]$.
\end{proof}


\begin{lemma}
\label{lemma:novanishingletter}
Let $z_0,z_1\in\Spin_{n+1}$. 
Let $K$ be a compact set and 
$H: K \to \cL_n(z_0;z_1)$ be a continuous function.
Let
\[ K_1 = \bigsqcup_{s \in K} (\{s\} \times \sing(H(s))) =
\{(s,t) \in K \times (0,1) \;|\; H(s)(t) \in \Sing_{n+1} \}. \]
Then $K_1$ is a compact set and satisfies the following condition:
\begin{gather*}
\forall (s_0,t_0) \in K_1, \; \forall \epsilon > 0, \;
\exists \delta > 0, \; \forall s \in K, \\
|s-s_0| < \delta \quad\to\quad
(\exists t \in (0,1), \; (s,t) \in K_1, \; |t-t_0| < \epsilon). 
\end{gather*}
\end{lemma}

\begin{proof}
Write $\tilde H(s,t) = H(s)(t) \in \Spin_{n+1}$
so that $\tilde H: K \times [0,1] \to \Spin_{n+1}$ is continuous.
Notice that $K_2 \subseteq K \times [0,1]$ defined by
\[ K_2 = \tilde H^{-1}[\Sing_{n+1}]\subseteq K_1 \cup (K \times \{0,1\}), 
\qquad
\Sing_{n+1} = \sqcup_{\sigma \ne \eta} \Bru_{\sigma}, \]
is closed and therefore compact.
Furthermore, the sets $A_0 = \tilde H^{-1}[\Bru_{\adv(z_0)}]$ and
$A_1 = \tilde H^{-1}[\Bru_{\chop(z_1)}]$ are open and disjoint from $K_2$.
From Theorem 3 of \cite{Goulart-Saldanha0},
for each $s \in K$ there exists $\epsilon_s > 0$ such that
$\{s\} \times (0,\epsilon_s) \subset A_0$
and
$\{s\} \times (1-\epsilon_s,1) \subset A_1$.
By compactness of $K$ there exists $\epsilon_{\ast} > 0$ such that
$K \times (0,\epsilon_{\ast}) \subset A_0$ and
$K \times (1-\epsilon_{\ast},1) \subset A_1$,
implying the compactness of $K_1 = K_2 \smallsetminus (K \times \{0,1\})$.

The remaining claim follows from Lemma \ref{lemma:conjecture}. 
\end{proof}

\begin{proof}[Proof of Theorem \ref{theo:Hausdorff}]
It follows from the condition in 
Lemma \ref{lemma:novanishingletter}, 
that the composite map $\sing \circ H$ is continuous. 
Since $\cL_n(z_0;z_1)$ is metrizable 
and $K$ is arbitrary, this implies the continuity  
of the map $\sing:\cL_n(z_0;z_1)\to\mathcal{H}([0,1])$. 
\end{proof}




Recall from the introduction that $\Gamma\in\cL_{n,\conv}(z)$ 
if and only if $\Gamma\in\cL_n(1;z)$ and 
$\sing(\Gamma)=\emptyset$.  
The following result is well known 
\cite{Anisov, Shapiro-Shapiro, Shapiro} and is presented here
for completeness and as an example of an application.

\begin{lemma}
\label{lemma:convex2}
If $z \in \widetilde\B^+_{n+1}$ then the subset
$\cL_{n,\conv}(z) \subset \cL_n(z)$
is either empty or a contractible connected component.
It is nonempty if and only if
$\chop(z)=\acute\eta$.
\end{lemma}

\begin{proof}
From Theorem \ref{theo:Hausdorff}
and the fact that $\emptyset \in \cH([0,1])$ is an isolated point
it follows that $\cL_{n,\conv}(z)$ is a union of connected components.
By item \ref{item:lemma:spaces:homotopyequivalence} 
of Lemma \ref{lemma:spaces}, 
it suffices to show that $\cL^{[H^1]}_{n,\conv}(z)$ is contractible. 

Consider first the case $z \in \Bru_{\acute\eta}$.
By applying a projective transformation we may assume
$z = \acute\eta = \exp(\frac{\pi}{2}\fh)$.
Take $\Gamma_0 \in \cL_{n,\conv}(\acute\eta)$,
$\Gamma_0(t) = \exp(\frac{\pi}{2}t\fh)$. 
Notice that Equation (9) of 
\cite{Goulart-Saldanha0} implies that 
$\Gamma_0(t)\in\Bru_{\acute\eta}$ for $0<t\leq1$, 
since $\Gamma_0(t)=U_0(t)\Pi(\acute\eta)U_1(t)$, 
where 
\begin{gather*}
U_1(t)=\exp\left(-\log\left(\cos\left(\pi t/2 \right)[\fh_L,\fh_L^\top]\right)\right)
\exp\left(-\tan\left(\pi t/2 \right)\fh_L^\top\right), \\
U_0(t)=\Pi(\acute\eta)
\exp\left(\tan\left(\pi t/2 \right)\fh_L\right)
\Pi(\acute\eta)^\top \in\Up^+_{n+1} 
\end{gather*}
(recall that the comutator 
$[\fh_L,\fh_L^\top]=\sum_{k=0}^n(2k-n)e_{k+1}e_{k+1}^\top$ 
is a diagonal matrix).
For $s \in (0,1]$, let $U_s \in \Up_{n+1}^{1}$ be such that
$\acute\eta^{U_s}=\Gamma_0(s)$.
For $\Gamma_1 \in \cL_{n,\conv}^{[H^1]}(\acute\eta)$ 
and $s \in (0,1]$
define $\Gamma_s  \in \cL_{n,\conv}^{[H^1]}(\acute\eta)$ by:
\[ \Gamma_s(t) = \begin{cases}
(\Gamma_1(\frac{t}{s}))^{U_s}, & t \in [0,s], \\
\Gamma_0(t), & t \in [s,1]. \end{cases} \]
The map $[0,1] \to  \cL_{n,\conv}^{[H^1]}(\acute\eta)$,
$s \mapsto \Gamma_s$ is continuous (even at $s = 0$).

The general case follows from 
Proposition 6.4 of \cite{Saldanha-Shapiro};
see also Remark \ref{rem:explicitcontraction}.
\end{proof}

\begin{rem}
\label{rem:convex}
It follows from Lemma \ref{lemma:convex} 
that $\sing(\Gamma)=\emptyset$ implies 
that $\Gamma$ is convex.
The reciprocal is not true. Indeed, 
for any $\sigma\in S_{n+1}$, $\sigma\ne\eta$, take 
$\Gamma:[0,1]\to\Spin_{n+1}$, 
$\Gamma(t)=\acute\sigma\exp(v(t-\frac12)\fh)$. 
For small $v>0$, $\Gamma$ is convex (short) but 
$\frac12\in\sing(\Gamma)$ (see also Example \ref{example:inout} below). 
\end{rem}


\section{Accessibility in triangular coordinates}
\label{sect:acctriangle}

For $L_{\bfx} \in \Pos_{\eta} \subset \Lo_{n+1}^{1}$, 
we shall be interested in the \emph{interval}
\[ [I,L_{\bfx}) = 
\overline{\Pos_{\eta}} \cap (L_{\bfx} \Neg_{\eta})
= \{ L \in \Lo_{n+1}^{1} \;|\; I \le L \ll L_{\bfx} \}
= \bigsqcup_{\sigma\in S_{n+1}} \Ac_{\sigma}(L_{\bfx}) \]
where the strata $\Ac_\sigma(L_{\bfx}) \subset \Pos_\sigma$ 
are given by 
\[ \Ac_{\sigma}(L_{\bfx}) = [I,L_{\bfx}) \cap \Pos_\sigma =
\{ L \in \Pos_\sigma \;|\; L \ll L_{\bfx} \}. \]
The sets $\Ac_{\sigma}(L_{\bfx})$ will be called \emph{accessibility sets},
suggesting that for $L \in \Pos_\sigma$,
$L \in \Ac_\sigma(L_{\bfx})$ if and only if there exists
a convex curve $\Gamma: [0,1] \to \Lo_{n+1}^1$
with $\Gamma(0) = L$ and $\Gamma(1) = L_{\bfx}$.

\begin{example}
\label{example:Ac}
Take $n = 2$ and, for $x,y,z\in\RR$,  write 
\begin{gather*}
L_{\bfx} = L(x,y,z) =
\begin{pmatrix} 
1 & 0 & 0 \\ 
x & 1 & 0 \\ 
z & y & 1 
\end{pmatrix} 
= \jacobi_1(c_1) \jacobi_2(c_2) \jacobi_1(c_3)
= \jacobi_2(\tilde c_1) \jacobi_1(\tilde c_2) \jacobi_2(\tilde c_3), \\
c_1 = x-\frac{z}{y}, \quad c_2 = y,  \quad c_3 = \frac{z}{y}, \qquad
\tilde c_1 = y-\frac{z}{x},\quad \tilde c_2 = x,\quad \tilde c_3 = \frac{z}{x};
\end{gather*}
as in Section \ref{sect:gso}, $\jacobi_j(t)=\exp(t\fl_j)$. 
We describe the strata 
$\Ac_{\sigma}=\Ac_{\sigma}(L_{\bfx})$.
The first stratum is a point:
$\Ac_e = \{I\}$.
Next we have line segments:
\[ \Ac_{a} = \{ \jacobi_1(t_1) \;|\; t_1 \in (0, c_1) \}, \quad
\Ac_{b} = \{ \jacobi_2(\tilde t_1) \;|\; \tilde t_1 \in (0, \tilde c_1) \}. \]
The next strata are surfaces:
\begin{align*}
{\Ac_{ab}} &=
\left\{ \jacobi_1(t_1) \jacobi_2(t_2) \;|\;
t_1 \in (0,c_1), t_2 \in (0,g_2(t_1)) \right\},
\quad g_2(t_1) = \frac{c_2c_3}{c_1+c_3-t_1}, \\
{\Ac_{ba}} &=
\left\{ \jacobi_2(\tilde t_1) \jacobi_1(\tilde t_2) \;|\;
\tilde t_1 \in (0,\tilde c_1), \tilde t_2 \in (0,\tilde g_2(\tilde t_1))
\right\},
\quad \tilde g_2(\tilde t_1) =
\frac{\tilde c_2\tilde c_3}{\tilde c_1+\tilde c_3-\tilde t_1}.
\end{align*}
Translating this parametrization back to $(x,y,z)$ coordinates
shows that $\Ac_{ab}$ is contained in the plane $z = 0$
and $\Ac_{ba}$ is contained in the hyperbolic paraboloid $z = xy$.
Finally, the open stratum $\Ac_{aba}$ can be described as
\begin{gather*}
\Ac_{aba} =
\left\{ \jacobi_1(t_1) \jacobi_2(t_2) \jacobi_1(t_3) \;|\;
t_1 \in \left(0,c_1 \right),
t_2 \in \left(0,g_2(t_1) \right),
t_3 \in \left(0,g_3(t_1,t_2) \right)\right\} \\
\phantom{\Ac_\eta} =
\left\{ \jacobi_2(\tilde t_1) \jacobi_1(\tilde t_2) \jacobi_2(\tilde t_3) \;|\;
\tilde t_1 \in \left(0,\tilde c_1 \right),
\tilde t_2 \in \left(0,\tilde g_2(\tilde t_1) \right),
\tilde t_3 \in \left(0,\tilde g_3(\tilde t_1,\tilde t_2) \right)\right\}, \\
g_3(t_1,t_2) = \frac{c_2(c_1-t_1)}{c_2-t_2}, \qquad
\tilde g_3(\tilde t_1,\tilde t_2) =
\frac{\tilde c_2(\tilde c_1-\tilde t_1)}{\tilde c_2-\tilde t_2}.
\end{gather*}
\end{example}

A \emph{quasiproduct} is a finite sequence $(X_j)_{1 \le j \le k}$
of open sets $X_j \subset (0,+\infty)^j$ such that there exist
a constant $c_1 \in (0,+\infty)$
and continuous functions $g_j: X_{j-1} \to (0,+\infty)$
for $2 \le j \le k$ such that
$X_1 = (0,c_1)$ and
\[ X_j = \{ (t_1, \ldots, t_{j-1}, t_j) \in X_{j-1} \times (0,+\infty) \;|\;
t_{j} < g_j(t_1, \ldots, t_{j-1}) \}, \quad
2 \le j \le k. \]
Notice that $X_k$ is homeomorphic to $\RR^k$.

\begin{lemma}
\label{lemma:triangularquasiproduct}
If $L_{\bfx} \in \Pos_\eta$, each stratum $\Ac_\sigma(L_{\bfx})$
is an open, bounded and contractible subset of $\Pos_\sigma$.
Moreover, if $\sigma = \sigma_k = a_{i_1} \cdots a_{i_k}$ is a reduced word
and $\sigma_j =  a_{i_1} \cdots a_{i_j}$, $j \le k$, then
\[ \Ac_{\sigma_j}(L_{\bfx}) = 
\{ \jacobi_{i_1}(t_1) \cdots
\jacobi_{i_j}(t_{j})
\;|\;
(t_1,\ldots,t_{j}) \in X_j \} \]
where the sequence $(X_j)_{1 \le j \le k}$ is a quasiproduct.
Furthermore, the functions \mbox{$g_j: X_{j-1} \to (0,+\infty)$}
are rational and bounded.
\end{lemma}


Example \ref{example:Ac} above illustrates this claim for $n = 2$.

\begin{proof}
Notice that $I \le L \ll L_{\bfx}$ implies
that $L \in \overline{\Pos_\eta}$ and that
there exists $\tilde L \in \Pos_\eta$ with $L \tilde L = L_{\bfx}$.
Computing $(L_{\bfx})_{ij}$ in this product yields
$0 \le (L)_{ij} \le (L_{\bfx})_{ij}$:
it follows that the interval $[I,L_{\bfx})$ is bounded.

The proof is by induction on $k = \inv(\sigma)$; the case $k = 1$ is easy.
Write
\[ X_j = \{ (t_1, \ldots, t_{j}) \in (0,+\infty)^{j} \;|\;
\jacobi_{i_1}(t_1) \cdots  \jacobi_{i_{j}}(t_{j})
\ll L_{\bfx} \}. \]
We assume by induction that $(X_j)_{1\le j\le k-1}$ is a quasiproduct;
we need to construct the function $g_k: X_{k-1} \to (0,+\infty)$
that obtains $X_k$.

Let $\eta = a_{j_1} \cdots a_{j_m}$ be a reduced word with $j_1 = i_k$. 
Given $(t_1, \ldots , t_{k-1}) \in X_{k-1}$, let
\[ L_{\sigma_{k-1}} = \jacobi_{i_1}(t_1) \cdots \jacobi_{i_{k-1}}(t_{k-1})
\in \Ac_{\sigma_{k-1}}(L_{\bfx}) \subset \Pos_{\sigma_{k-1}} \]
and write
\[ L_{\sigma_{k-1}}^{-1} L_{\bfx} =
\jacobi_{j_1}(\tau_1) \cdots \jacobi_{j_m}(\tau_m) \in \Pos_\eta \]
so that $\tau_1 > 0$ is a function of $(t_1, \ldots, t_{k-1})$:
define $g_k(t_1, \ldots, t_{k-1}) = \tau_1$.
That $g_k$ is a rational function follows from the fact that,   
for all $i\in\llbracket n-1\rrbracket$, $s_1,s_2,s_3\in\RR$,  
\[\jacobi_i(s_1)\jacobi_{i+1}(s_2)\jacobi_i(s_3)
=\jacobi_{i+1}\left(\frac{s_2s_3}{s_1+s_3}\right)
\jacobi_i(s_1+s_3)
\jacobi_{i+1}\left(\frac{s_1s_2}{s_1+s_3}\right).\]
As in Lemma 5.8 
of \cite{Goulart-Saldanha0}, 
$\jacobi_{i_k}(t) \ll L_{\sigma_{k-1}}^{-1} L_{\bfx}$
if and only if $t < g_k(t_1, \ldots, t_{k-1})$.
\end{proof}



\section{Accessibility in the spin group}
\label{sect:acc}

For $z_{\bfx} \in \bQ[\Pos_\eta] \subseteq \cU_1 \cap \Bru_{\acute\eta}$
and $\sigma \in S_{n+1}$, we define
\[ \Ac_\sigma(z_{\bfx}) = \bQ[\Ac_\sigma(\bL(z_{\bfx}))]
\subset \bQ[\Pos_\sigma] \subset \Bru_{\acute\sigma}. \]
For each $z \in \Ac_\sigma(z_{\bfx})$, there exists a locally convex curve
$\Gamma: [0,1] \to \cU_1$ 
with $\Gamma(0) = z$ and $\Gamma(1) = z_{\bfx}$.
Indeed, just take a convex curve
$\Gamma_L: [0,1] \to \Lo_{n+1}^{1}$
with $\Gamma_L(0) = \bL(z)$ and $\Gamma_L(1) = \bL(z_{\bfx})$
and define $\Gamma = \bQ \circ \Gamma_L$.
Similarly, for $z \in \bQ[\Pos_\sigma] \smallsetminus \Ac_\sigma(z_{\bfx})$, 
no such curve exists.

For $z_{\bfx} \in \Bru_{\acute\eta}$, choose $U \in \Up_{n+1}^{+}$
such that $z_{\bfx} = z_0^U$, $z_0  \in \bQ[\Pos_\eta]$.
For $\sigma \in S_{n+1}$, define
$\Ac_{\sigma}(z_{\bfx}) = (\Ac_{\sigma}(z_0))^U$;
this turns out to be well-defined and the properties above still hold.
We want to define $\Ac_{\sigma}(z_{\bfx})$ for any 
$z_{\bfx} \in \chop^{-1}[\{\acute\eta\}]$.
This will require a certain detour.
We shall first present a topological construction (using curves),
then an algebraic one (using coordinates)
and then finally prove their equivalence.

For $q\in \Quat_{n+1}$, set
\begin{align*}
\Bru^{0}_{q\acute\eta} &= \adv^{-1}[\{q\acute\eta\}] 
= \bigsqcup_{\sigma \in S_{n+1}} \Bru_{q\acute\sigma}, \\
\Bru^{1}_{q\grave\eta} &= \chop^{-1}[\{q\grave\eta\}]
= \bigsqcup_{\sigma \in S_{n+1}}
\Bru_{q\grave\sigma}.
\end{align*}
A locally convex curve $\Gamma: [0,1] \to \Spin_{n+1}$
satisfying $\Gamma(t) \in \Bru_{\acute\eta}$ for $t \in (0,1)$
will necessarily satisfy
$\Gamma(0) \in \Bruadv$ and $\Gamma(1) \in \Bruchop$.
Notice that $\Bru_{\acute\eta} \subseteq \Bruadv \cap \Bruchop$.
Given $\sigma \in S_{n+1}$, we have
\[ \Bru_\sigma \cap \Bruadv = \Bru_{\acute\sigma}, \qquad
\Bru_\sigma \cap \Bruchop = \Bru_{\hat\eta(\hat\sigma)^{-1}\acute\sigma} \]
and therefore
\[ \Bruadv \cap \Bruchop =
\bigsqcup_{\sigma \in S_{n+1}, \hat\sigma = \hat\eta} \Bru_{\acute\sigma}.\]
It follows from Remark 3.8 
of \cite{Goulart-Saldanha0} that  
$\Bruadv \cap \Bruchop = \Bru_{\acute\eta}$ precisely for $n \le 3$.
In order to extend locally convex curves in $\Bru_{\acute\eta}$
to the boundary and not mix up entry points with exit points
we define a new larger space:
\[ \Brujojo =
( (\adv^{-1}[\{\acute\eta\}] \times \{0\}) \sqcup
(\chop^{-1}[\{\acute\eta\}] \times \{1\}) )/\sim \]
where $(z,0)\sim(z,1)$ for $z\in\Bru_{\acute\eta}$ (and only there).
We abuse notation by writing
\[ \Bruadv \subset \Brujojo, \qquad \Bruchop \subset \Brujojo; \]
in \emph{this} context, $\Bruadv \cap \Bruchop = \Bru_{\acute\eta}$.
A locally convex curve $\Gamma: [0,1] \to \Brujojo$ corresponds to 
a locally convex curve $\Gamma_1: [0,1] \to \Spin_{n+1}$
satisfying $\Gamma_1(t) \in \Bru_{\acute\eta}$ for $t \in (0,1)$
with $\Gamma(0) = (\Gamma_1(0),0)$, $\Gamma(1) = (\Gamma_1(1),1)$.

For $z_0 \in \Bruadv$ and $z_1 \in \Bruchop$,
let $\cLjojo(z_0;z_1) \subset \cL_n^{[H^1]}(z_0;z_1)$
be the set of locally convex curves $\Gamma: [0,1] \to \Spin_{n+1}$
such that $\Gamma(0) = z_0$, $\Gamma(1) = z_1$
and $\Gamma(t) \in \Bru_{\acute\eta}$ for all $t \in (0,1)$.
For $z_0, z_1 \in \Brujojo$,
write $z_0 \ll z_1$ if and only if
$z_0 \in \Bruadv$, $z_1 \in \Bruchop$,
and $\cLjojo(z_0;z_1) \ne \emptyset$
(compare with Lemma 5.3 
of \cite{Goulart-Saldanha0}).

\begin{lemma}
\label{lemma:cLjojo}
Consider $z_0 \in \Bruadv$ and $z_1 \in \Bruchop$.
The set $\cLjojo(z_0;z_1)$ is either empty or contractible.
If $z_0 \ll z_1$ then $z_0^{-1}z_1 \in \Bruchop$
and $\cLjojo(z_0;z_1) = \cL_{n,\conv}(z_0;z_1)$.
\end{lemma}


\begin{example}
\label{example:inout}
Recall from Lemma \ref{lemma:convex2} that
$\cL_{n,\conv}(z_0;z_1) \subset \cL_n(z_0;z_1)$
is a contractible connected component if $z_0^{-1}z_1 \in \Bruchop$
and is empty otherwise.
It is entirely possible to have
$z_0 \in \Bruadv$,  $z_1 \in \Bruchop$,
$z_0^{-1}z_1 \in \Bruchop$ and $z_0 \not\ll z_1$
so that $\cLjojo(z_0;z_1) = \emptyset$.
In this case, there are convex curves in $\cL(z_0;z_1)$
but they never belong to $\cLjojo(z_0;z_1)$. 
A simple case is $n = 6$, 
$z_0 = \exp(\frac{3\pi}{4}\fh)$ and $z_1 = \exp(\frac{\pi}{4}\fh)$.
Recall from Example 3.7 
of \cite{Goulart-Saldanha0} 
that, for $n = 6$, we have $\hat\eta = \exp(\pi\fh) = 1$.
We have $z_0^{-1}z_1 = \exp(\frac{\pi}{2}\fh) = \acute\eta$
and the curve $\Gamma(t) = z_0 \exp(\frac{\pi}{2} t\fh)$ is convex.
\end{example}
 
\begin{proof}[Proof of Lemma \ref{lemma:cLjojo}]
By definition, $\cLjojo(z_0;z_1) = \sing^{-1}[\{\emptyset\}]$.
Since $\emptyset$ is an isolated point in $\cH([0,1])$,
the set $\cLjojo(z_0;z_1)$ is a union of
connected components of $\cL_n(z_0;z_1)$, 
by Theorem \ref{theo:Hausdorff}.

We know from \cite{Saldanha-Shapiro} that if $\Gamma$ is not convex
then $\Gamma$ is in the same connected component as
$\Gamma$ with added loops,
which clearly does not have empty singular set.
Thus the only connected component of $\cL_n(z_0;z_1)$
which may be contained in $\cLjojo(z_0;z_1)$
is $\cL_{n,\conv}(z_0;z_1)$.
\end{proof}

Given $z_{\bfx} \in \Bruchop$ and $\sigma \in S_{n+1}$,
consider $\Bru_{\acute\sigma} \subset \Bruadv$:
let
\[ \Ac_{\sigma}(z_{\bfx}) =
\{ z \in \Bru_{\acute\sigma} \subset \Bruadv \;|\; z \ll z_{\bfx} \}. \]

\begin{lemma}
\label{lemma:acstep}
Consider $z_{\bfx} \in \Bruchop$.
Consider $\sigma_{k-1} \vartriangleleft
\sigma_k = \sigma_{k-1} a_{i_k} \in S_{n+1}$,
$\inv(\sigma_k) = k$.
Consider $z_{k-1} \in \Bru_{\acute\sigma_{k-1}}$
and $z_k = z_{k-1} \alpha_{i_k}(\theta_k) \in \Bru_{\acute\sigma_k}$,
$\theta_k \in (0,\pi)$.
If $z_k \ll z_{\bfx}$ then
$z_{k-1}\alpha_{i_k}(\theta) \ll z_{\bfx}$
for all $\theta \in [0,\theta_k]$.
\end{lemma}

As in Section \ref{sect:gso}, $\vartriangleleft$ 
denotes the covering relation for the Bruhat order in $S_{n+1}$. 

\begin{proof}
From Corollary 6.4 of \cite{Goulart-Saldanha0},
if $z_k \in \bQ[\Pos_{\sigma_k}]$ then 
\[z_{k-1}\alpha_{i_k}(\theta) \in
\bQ[\Pos_{\sigma_{k-1}} \sqcup \Pos_{\sigma_{k}}].\]
In this case, take $L_{k} = \bL(z_{k})$ 
and $L_{\theta} = \bL(z_{k-1} \alpha_{i_k}(\theta) )$.
Consider a locally convex curve $\Gamma \in \cLjojo(z_k;z_{\bfx})$.
As in the proof of Theorem 3 
of \cite{Goulart-Saldanha0}, 
take $\Gamma_L(t) = \bL(\Gamma(t))$ so that
$\Gamma_L(0) = L_k$. 
By Lemma 5.7 of \cite{Goulart-Saldanha0},
there exists $\epsilon > 0$ such that
$\Gamma_L$ is well-defined in $[0,2\epsilon]$ and 
$L_{\epsilon} = \Gamma_L(\epsilon) \in \Pos_\eta$.
We have $L_{\theta} \le L_k \ll L_\epsilon$
and therefore $L_{\theta} \ll L_\epsilon$
(Lemma 5.2 and
Equation (15) 
of \cite{Goulart-Saldanha0}).
By Lemma 5.3 of \cite{Goulart-Saldanha0},
there exists a locally convex curve
$\Gamma_{\epsilon}: [0,\epsilon] \to \Lo_{n+1}^1$,
$\Gamma_{\epsilon}(0) = L_{\theta}$ and
$\Gamma_{\epsilon}(\epsilon) = L_{\epsilon}$.
Define 
\[ \Gamma_1(t) = \begin{cases}
\bQ(\Gamma_{\epsilon}(t)), & t \in [0,\epsilon], \\
\Gamma(t), & t \in [\epsilon,1]. \end{cases} \]
Notice that, for $t \in (0,\epsilon)$, we have
$\Gamma_{\epsilon}(t) \in \Pos_{\eta}$ and therefore
$\Gamma_1(t) \in \Bru_{\acute\eta}$.
The curve $\Gamma_1: [0,1] \to \Spin_{n+1}$
is locally convex and satisfies
$\Gamma_1(0) = z_{k-1}\alpha_{i_k}(\theta)$, $\Gamma(1) = z_{\bfx}$
and $\Gamma(t) \in \Bru_{\acute\eta}$ for all $t \in (0,1)$.
By definition, $z_{k-1} \alpha_{i_k}(\theta) \ll  z_{\bfx}$.

In general, there is an upper triangular matrix $U \in \Up_{n+1}^{1}$
such that the correponding projective transformation
takes $z_k$ to $z_k^U = \bQ(U^{-1} z_k) \in \bQ[\Pos_{\sigma_k}]$,
reducing the situation to the previous case.
\end{proof}

We now present an algebraic definition.
Consider $z_{\bfx} \in \Bruchop$.
Consider $\rho_0 \in S_{n+1}$ such that
$z_{\bfx} \in \Bru_{\hat\eta(\acute\rho_0)^{-1}}$,
$y_0 = z_{\bfx}^{-1} \hat\eta \in \Bru_{\acute\rho_0}$.
We first define sets
$\Ac_{(i_1,\ldots,i_k)}(z_{\bfx}) \subseteq \Bru_{\acute\sigma}$
where $\sigma = a_{i_1}\cdots a_{i_k}$ is a reduced word.
When $z_{\bfx}$ is fixed (and thus so are $y_0$ and $\rho_0$)
we write for simplicity 
$\Ac_{(i_1,\ldots,i_k)} = \Ac_{(i_1,\ldots,i_k)}(z_{\bfx})$.

For each $j\in\llbracket k \rrbracket$, 
set $\sigma_j = a_{i_1} \cdots a_{i_j}$, 
so that 
$\sigma_{j-1} \vartriangleleft \sigma_j = \sigma_{j-1} a_{i_j}$;
also, define recursively 
\[\rho_j=
\begin{cases}
\rho_{j-1}a_{i_j}, \quad \text{if} 
\quad \rho_{j-1}\vartriangleleft\rho_{j-1} a_{i_j}, \\
\rho_{j-1}, \quad \text{otherwise},
\end{cases}\]
so that either $\rho_{j-1} = \rho_j$
or $\rho_{j-1} \vartriangleleft \rho_j = \rho_{j-1} a_{i_j}$. 
For those $j$ such that 
$\rho_{j-1} =\rho_j$, 
define auxiliary functions 
$\Theta_{i_j}:\Bru_{\acute\rho_j}\to(0,\pi)$ as follows: 
$\Theta_{i_j}(z)=\theta$ if and only 
$z\alpha_{i_j}(-\theta)\in\Bru_{\rho_{j}a_{i_j}}$. 
It is a consequence of the proof of Theorem 1 of \cite{Goulart-Saldanha0} 
that these functions are well-defined and smooth 
(
see Remark 6.6 of \cite{Goulart-Saldanha0}). 

Set 
$\Ac_{(\,)} = \Bru_{1} = \{1\}$.
We assume $\Ac_{(i_1,\ldots,i_{j-1})}$ defined
and proceed to construct $\Ac_{(i_1,\ldots,i_j)}$:
\begin{gather*}
\Ac_{(i_1,\ldots,i_j)} =
\{ z_{j-1} \alpha_{i_j}(\theta_j) \;|\;
z_{j-1} \in  \Ac_{(i_1,\ldots,i_{j-1})}, \;
\theta_j \in (0,\vartheta_{i_j}(z_{j-1})) \}; 
\\
\vartheta_{i_j}: \Ac_{(i_1,\ldots,i_{j-1})} \to (0,\pi], \qquad
\vartheta_{i_j}(z_{j-1}) = 
\begin{cases}
\pi, & \rho_{j-1} \vartriangleleft \rho_j, \\
\pi-\Theta_{i_j}(y_0 z_{j-1}), & \rho_{j-1} = \rho_j.
\end{cases}
\end{gather*}

\begin{lemma}
\label{lemma:quasiproduct}
The sets $\Ac_{(i_1, \ldots, i_j)}$, $1 \le j \le k$,
defined above satisfy
\[ \Ac_{(i_1,\ldots,i_j)} =
\{ \alpha_{i_1}(\theta_1) \cdots \alpha_{i_j}(\theta_j) \;|\;
(\theta_1, \ldots, \theta_j) \in X_j \}
\subseteq
\Bru_{\acute\sigma_j} \cap
(y_0^{-1} \Bru_{\acute\rho_j}) \]
where $(X_j)_{1 \le j \le k}$ is a quasiproduct;
$\Ac_{(i_1,\ldots,i_j)}$ is diffeomorphic to $\RR^j$.
\end{lemma}

Notice that the inclusion in the statement is necessary to make sense
of the definition of $\vartheta_{i_j}$.
The reader should compare this result with
Lemma \ref{lemma:triangularquasiproduct}.

\begin{proof}
The proof is by induction on $k$; the case $k = 0$ is trivial.
Take
$z_k = z_{k-1} \alpha_{i_k}(\theta_k) \in \Ac_{(i_1, \ldots, i_k)}$,
$z_{k-1} \in  \Ac_{(i_1, \ldots, i_{k-1})}$,
$\theta_k \in (0,\vartheta_{i_k}(z_{k-1}))$.
We assume by induction hypothesis that
$\Ac_{(i_1, \ldots, i_{k-1})} \subseteq \Bru_{\acute\sigma_{k-1}}$.
We therefore have
$z_k \in \Bru_{\acute\sigma_{k-1}} \Bru_{\acute a_{i_k}}
= \Bru_{\acute\sigma_k}$
(
by Corollary 6.2 of \cite{Goulart-Saldanha0}).
We also assume by induction hypothesis that
$y_{k-1} = y_0 z_{k-1} \in \Bru_{\acute\rho_{k-1}}$.
If $\rho_{k-1} \vartriangleleft \rho_k$,
Corollary 6.2 of \cite{Goulart-Saldanha0} 
implies that 
$y_k = y_0 z_{k} = y_{k-1} \alpha_{i_k}(\theta_k)
\in \Bru_{\acute\rho_{k-1}} \Bru_{\acute a_{i_k}} = 
\Bru_{\acute\rho_k}$.
If $\rho_{k-1} = \rho_k$,
take $\tilde\theta = \Theta_{i_k}(y_{k-1})$
and $\tilde y \in \Bru_{\rho_k a_{i_k}}$
such that
$y_{k-1} = \tilde y \alpha_{i_k}(\tilde\theta)$.
By our recursive definition,
$\tilde\theta + \theta_k < \pi$;
by Theorem 1 of \cite{Goulart-Saldanha0},
we have $y_{k} =
\tilde y \alpha_{i_k}(\tilde\theta + \theta_k) \in \Bru_{\acute\rho_k}$.
\end{proof}



We now prove that the two definitions are equivalent.

\begin{lemma}
\label{lemma:twodefs}
Consider $z_{\bfx} \in \Bruchop$ and
$\sigma_k = a_{i_1}\cdots a_{i_k}$ a reduced word in $S_{n+1}$.
Then $\Ac_{\sigma_k}(z_{\bfx}) = \Ac_{(i_1,\ldots,i_k)}(z_{\bfx})$.
\end{lemma}

\begin{proof}
The proof is by induction on $k$; the case $k = 0$ is trivial.
Assume therefore
$\Ac_{\sigma_{k-1}}(z_{\bfx}) = \Ac_{(i_1,\ldots,i_{k-1})}(z_{\bfx})$
for $\sigma_{k-1} = a_{i_1}\cdots a_{i_{k-1}}$.

Consider $z_k = z_{k-1} \alpha_{i_k}(\theta_k)$,
$z_{k-1} \in \Bru_{\acute\sigma_{k-1}}$, $z_k \in \Bru_{\acute\sigma_k}$,
$\theta_k \in (0,\pi)$.  
It follows from Lemma \ref{lemma:acstep}
that $z_k \in \Ac_{\sigma_k}$ implies
$z_{k-1} \in \Ac_{\sigma_{k-1}}$
and therefore
\[ \Ac_{\sigma_k} \subseteq \Ac_{\sigma_{k-1}} \Bru_{\acute a_{i_k}},
\qquad
\Ac_{(i_1,\ldots,i_k)} \subseteq \Ac_{\sigma_{k-1}} \Bru_{\acute a_{i_k}};
\]
we have to prove that these two sets are equal.

Given $z_{k-1} \in \Ac_{\sigma_{k-1}}$,
let $J_{z_{k-1}} \subseteq (0,\pi)$ be the set such that,
for all $\theta_k \in (0,\pi)$,
$\theta_k \in J_{z_{k-1}}$ 
if and only if $z_{k-1} \alpha_{i_k}(\theta_k) \in \Ac_{\sigma_k}$. 
It follows from Lemma \ref{lemma:acstep}
that $J_{z_{k-1}}$ is either empty
or an initial interval.

We claim that $J_{z_{k-1}}$ is not empty.
By applying a projective transformation,
we may assume 
$z_{k-1} = \bQ(L_{k-1})$, $L_{k-1} \in \Pos_{\sigma_{k-1}}$.
Take $\Gamma \in \cLjojo(z_{k-1};z_{\bfx})$.
Define $\Gamma_L = \bL \circ \Gamma$,
with maximal connected domain containing $t = 0$.
Consider $t_\bullet > 0$ in this domain
and $L_{\bullet} = \Gamma_L(t_\bullet)$,
$L_{\bullet} \in \Pos_{\eta}$, $L_{k-1} \ll L_{\bullet}$.
Take $t_k > 0$ such that
$L_{k-1}\lambda_{i_k}(t_k) \ll L_{\bullet}$;
define $\theta_k > 0$ by
$\bQ(L_{k-1}\lambda_{i_k}(t_k)) = z_{k-1} \alpha_{i_k}(\theta_k)$.
Take convex $\Gamma_{L,1}:[0,t_\bullet]\to \Lo_{n+1}^1$ 
such that $\Gamma_{L,1}(0) = L_{k-1}\lambda_{i_k}(t_k)$
and $\Gamma_{L,1}(t_\bullet) = L_{\bullet}$.
Finally, take $\Gamma_1: [0,1] \to \Spin_{n+1}$,
\[ \Gamma_1(t) = \begin{cases}
\bQ(\Gamma_{L,1}(t)), & t \in [0,t_\bullet], \\
\Gamma(t), & t \in [t_\bullet,1]. \end{cases} \]
We have   
$\Gamma_1\in\cLjojo(z_{k-1} \alpha_{i_k}(\theta_k);z_\bfx)$ 
and therefore $\theta_k \in J_{z_{k-1}}$, as claimed.

We claim that $J_{z_{k-1}}$ is open.
Assume by contradiction $\theta_k^\star = \max(J_{z_{k-1}})$,
$z_k^\star = z_{k-1} \alpha_{i_k}(\theta_k^\star)\in\Ac_{\sigma_k}$.
By applying a projective transformation,
we may assume that $z_k^\star \in \bQ[\Pos_{\sigma_{k}}]$.
As in the previous paragraph, take a locally convex curve 
$\Gamma\in\cLjojo(z_k^\star;z_\bfx)$, 
use $\bL$ to take its initial segment to $\Lo_{n+1}^{1}$
and slightly perturb it to obtain
$\theta_k \in J_{z_{k-1}}$, $\theta_k > \theta_k^\star$.
The argument is so similar that we feel that a repetition is pointless.

At this point we know that there exists a function
$\tilde\vartheta_{i_k}: \Ac_{\sigma_{k-1}} \to (0,\pi]$
such that $J_{z_{k-1}} = (0,\tilde\vartheta_{i_k}(z_{k-1}))$.
We are left with proving that
$\vartheta_{i_k} = \tilde\vartheta_{i_k}$.

We first prove that
$\tilde\vartheta_{i_k}(z_{k-1}) \le \vartheta_{i_k}(z_{k-1})$
for all $z_{k-1}$.
If $\rho_{k-1} \vartriangleleft \rho_k$ then
$\vartheta_{i_k}(z_{k-1}) = \pi$ and we are done.
If $\rho_{k-1} = \rho_k$,
take $\theta_k^{\bullet} = \vartheta_{i_k}(z_{k-1})$,
$z_k^{\bullet} = z_{k-1}\alpha_{i_k}(\theta_k^{\bullet})$
and $y_k^{\bullet} = y_0 z_k^\bullet$.
Recall that in this case there exists $\rho_{\bullet} \in S_{n+1}$,
$\rho_{\bullet} \vartriangleleft \rho_{k-1} = \rho_{k} = \rho_{\bullet} a_{i_k}$.
By definition of $\vartheta_{i_k}$,
$y_k^{\bullet} \in \Bru_{\acute\rho_{\bullet} \hat a_{i_k}}$
so that $\adv(y_k^\bullet) = {q^\bullet \acute\eta}$
for $q^\bullet \in \Quat_{n+1}$, $q^\bullet \ne 1$.
By Theorem 3 of \cite{Goulart-Saldanha0}, 
any locally convex curve starting at $y_k^{\bullet}$
immediately enters $\Bru_{q^\bullet \acute\eta}$.
Thus, $\cLjojo(y_k^\bullet;\hat\eta)=\emptyset$ 
and therefore 
$\cLjojo (z_k^\bullet;z_\bfx)=\emptyset$.
It follows that $z_k^{\bullet} \notin \Ac_{\sigma_k}(z_{\bfx})$
and therefore $\theta_k^{\bullet} \ge \tilde\vartheta_{i_k}(z_{k-1})$,
proving our claim.

We finally prove that
$\tilde\vartheta_{i_k}(z_{k-1}) \ge \vartheta_{i_k}(z_{k-1})$.
Consider $\theta_k <  \vartheta_{i_k}(z_{k-1})$,
$z_k = z_{k-1} \alpha_{i_k}(\theta_k)$
and $y_k = y_{k-1} \alpha_{i_k}(\theta_k) = y_0 z_k \in \Bru_{\acute\rho_k}$.
Notice that
$z_{k-1} \alpha_{i_k}(\theta) \in \Bruadv$ and
$y_{k-1} \alpha_{i_k}(\theta) \in \Bruadv$
for all $\theta \in [0,\theta_k]$.
By compactness and 
Theorem 3 of \cite{Goulart-Saldanha0},
there exists $c > 0$ such that, for all $\theta \in [0,\theta_k]$
and for all $t \in (0,c]$, we have
both $z_{k-1} \alpha_{i_k}(\theta) \exp(t\fh) \in \Bru_{\acute\eta}$ and
$y_{k-1} \alpha_{i_k}(\theta) \exp(t\fh) \in \Bru_{\acute\eta}$.
Apply Lemma 6.1 of \cite{Goulart-Saldanha0} 
to obtain a continuous family 
$H: [0,\theta_k] \times [\frac12,1] \to \Spin_{n+1}$
of locally convex curves $H(\theta): [\frac12,1] \to \Spin_{n+1}$
going from $y_{k-1} \alpha_{i_k}(\theta) \exp(c\fh)$ to $\hat\eta$.
Extend this to 
$H: [0,\theta_k] \times [0,1] \to \Spin_{n+1}$
by defining $H(\theta)(t) = 
y_{k-1} \alpha_{i_k}(\theta) \exp(2ct\fh)$ 
for $t \in [0,\frac12]$. 
This extension is still continuous. 
For each $\theta \in [0,\theta_k]$, 
the arc $H(\theta): [0,1] \to \Spin_{n+1}$ is in 
$\cLjojo(y_{k-1} \alpha_{i_k}(\theta);\hat\eta)$, 
since we have $H(\theta)(t) \in \Bru_{\acute\eta}$
for all $t \in (0,1)$. 
Multiply by $y_0^{-1}$ to obtain a family $y_0^{-1}H$
of locally convex curves $\Gamma_{\theta} = y_0^{-1} H(\theta): [0,1] \to \Spin_{n+1}$
going from $z_{k-1} \alpha_{i_k}(\theta)$ to $z_{\bfx}$.
We prove that, for all $\theta$, we have
$\Gamma_{\theta} \in \cLjojo(z_{k-1} \alpha_{i_k}(\theta); z_{\bfx})$,
i.e., that $\Gamma_{\theta}(t) \in \Bru_{\acute\eta}$
for all $t \in (0,1)$.
We know that $\Gamma_0$ is convex 
and that $z_{k-1} \in \Ac_{\sigma_{k-1}}(z_{\bfx})$
and therefore, from Lemma \ref{lemma:cLjojo},
that $\Gamma_0 \in \cLjojo(z_{k-1}; z_{\bfx})$.
We know by construction that 
$\Gamma_{\theta}(t) \in \Bru_{\acute\eta}$
for all $t \in (0,\frac12)$.
Apply again Lemma 6.1 
of \cite{Goulart-Saldanha0} to construct 
a continuous family of 
convex arcs $\tilde\Gamma_{\theta}:[0,\frac12] \to \Spin_{n+1}$
from $\tilde\Gamma_\theta(0) = 1$ to
$\tilde\Gamma_\theta(\frac12) = \Gamma_\theta(\frac12)$.
Extend $\tilde\Gamma_\theta$ to $[0,1]$
by $\tilde\Gamma_\theta(t) = \Gamma_\theta(t)$ 
for $t \in [\frac12,1]$. 
The corresponding family of extended 
locally convex curves is again continuous. 
We have $\sing(\tilde\Gamma_0) = \emptyset$.
Also, from Theorem \ref{theo:Hausdorff},
$\sing(\tilde\Gamma_\theta)$ is a continuous function of $\theta$.
Since $\emptyset \in \cH([0,1])$ is an isolated point,
we have $\sing(\tilde\Gamma_\theta) = \emptyset$ for all $\theta$,
as desired.
This implies that $z_k=z_{k-1}\alpha_{i_k}(\theta_k) \ll z_{\bfx}$
and therefore $\theta_k < \tilde\vartheta_{i_k}(z_{k-1})$.
Since this holds for any $\theta_k <  \vartheta_{i_k}(z_{k-1})$
we have $\tilde\vartheta_{i_k}(z_{k-1}) \ge \vartheta_{i_k}(z_{k-1})$,
completing our proof.
\end{proof}

\begin{rem}
\label{rem:explicitcontraction}
We saw in Lemmas \ref{lemma:convex2} and 
\ref{lemma:cLjojo} that,
given $z_0 \in \Bruadv$ and $z_1 \in \Bruchop$, 
the set $\cLjojo(z_0;z_1)$ is either empty
or equal to $\cL_{n,\conv}(z_0;z_1)$ and contractible.
In Lemma \ref{lemma:convex2} we saw an explicit contraction
if $z_0^{-1}z_1 \in \Bru_{\acute\eta}$
but otherwise used Proposition 6.4 of \cite{Saldanha-Shapiro}.
We now present a more explicit contraction in general.
For any $\Gamma \in \cLjojo(z_0;z_1)$, we have
$(\Gamma(0))^{-1}\Gamma(\frac12)  \in \Bru_{\acute\eta}$ and
$(\Gamma(\frac12))^{-1}\Gamma(1)  \in \Bru_{\acute\eta}$.
Apply the contraction in the proof of Lemma \ref{lemma:convex2}
to each arc, leaving $\Gamma(\frac12)$ fixed.
This takes us to a set of curves parametrized
by $\Gamma(\frac12) \in z_0 \Ac_{\eta}(z_0^{-1}z_1)$.
We now know that $\Ac_{\eta}(z_0^{-1}z_1)$ is diffeomorphic to $\RR^m$,
$m=n(n+1)/2$ 
(with a rather explicit diffeomorphism).
\end{rem}

\section{Proof of Theorem \ref{theo:stratification}}
\label{sect:paths}

 
Given $\Gamma \in \cL_n(z_0;z_\bullet)$, we write  
$\sing(\Gamma) = \{t_1 < \cdots < t_\ell\}$. 
Let its \emph{path} be
\[ 
\pathiti(\Gamma) = (z_1, \ldots, z_\ell), \quad
z_j = \Gamma(t_j) \in \Bru_{\eta\sigma_j}. \]
Let the \emph{length} of the corresponding itinerary 
$\iti(\Gamma)=w = (\sigma_1, \ldots, \sigma_\ell) \in \Word_n$
be $\ell=\ell(w) = \operatorname{card}(\sing(\Gamma)) \in \NN$.
Recall from Lemma \ref{lemma:cLjojo} that  
$\Gamma \in \cL_n(\hat\eta)$ is convex if and only if 
its itinerary $\iti(\Gamma)$ is the empty word of length $0$.

Given the path $(z_1, \ldots, z_\ell)$ of some $\Gamma \in \cL_n$,
it is trivial to determine
the corresponding itinerary $w = (\sigma_1, \ldots, \sigma_\ell)$.
Conversely, given an itinerary 
$w = (\sigma_1, \ldots, \sigma_\ell) \in \Word_n$,
define $B(w,j) \in \widetilde \B_{n+1}^{+}$ 
for $j \in \ZZ$, $0 \le j \le \ell+1$,  
and $B(w,j+\frac12) \in \widetilde \B_{n+1}^{+}$ 
for $j \in \ZZ$, $0 \le j \le \ell$, by
\begin{equation}
\label{eq:Bwj}
\begin{gathered}
B(w,0) = 1, \quad B\left(w,\frac12\right) = \acute\eta, \quad
B(w,j) = B\left(w,j-\frac12\right) \acute\sigma_j, \\
B\left(w,j+\frac12\right) = B\left(w,j-\frac12\right) \hat\sigma_j, \quad
B(w,\ell+1) = B\left(w,\ell+\frac12\right) \acute\eta.
\end{gathered}
\end{equation}
In particular, we have $B(w,\ell+1) =
\acute\eta\hat w\acute\eta\in\Quat_{n+1}$, 
where we define the hat of a word by 
$\hat w=\hat\sigma_1\cdots\hat\sigma_\ell\in\Quat_{n+1}$.
We adopt here the conventions $t_0 = 0$, $t_{\ell+1} = 1$,
$z_0 = 1$, $z_{\ell+1} = B(w,\ell+1)$, $\sigma_0 = \sigma_{\ell+1} = \eta$.
It follows from Theorem 3 
of \cite{Goulart-Saldanha0} 
that if $\Gamma \in \cL_n[w]$ and
$\sing(\Gamma) = \{t_1 < \cdots < t_\ell \}$ then
\[
\Gamma \in \cL_n(\acute\eta\hat w \acute\eta), \qquad
\Gamma(t_j) \in \Bru_{B(w,j)}, \qquad
\forall t \in (t_j,t_{j+1}), \; \Gamma(t) \in \Bru_{B\left(w,j+\frac12\right)}. \]
Thus, if $\Gamma \in \cL_n[w]$ then
\( \pathiti(\Gamma) \in \Bru_{B(w,1)} \times \cdots \times \Bru_{B(w,\ell)} \).

Given $w=(\sigma_1,\cdots,\sigma_\ell) \in \Word_n$ 
and $j \in \ZZ$, $0 \le j \le \ell$, 
define $q_j \in \Quat_{n+1}$ by
\begin{gather*}
B(w,j) = q_j \longacute(\eta\sigma_j) \in q_j \Bruadv, \qquad
B\left(w,j+\frac12\right) = q_j \acute\eta, \\
B(w,j+1) = q_j \hat\eta \longgrave(\eta\sigma_{j+1})
\in q_j \Bruchop. 
\end{gather*}
A sequence $(z_1, \ldots, z_\ell) \in \Bru_{B(w,1)} \times \cdots \times \Bru_{B(w,\ell)}$
is an \emph{accessible path} for $w$ if
\[ \forall j \in \llbracket \ell \rrbracket \;
\left(q_j^{-1} z_j \in \Ac_{\eta\sigma_j}(q_j^{-1} z_{j+1})\right). \]
Let $\Pathiti(w) \subseteq \Bru_{B(w,1)} \times \cdots \times \Bru_{B(w,\ell)}$
be the set of accessible paths for $w$.

\begin{lemma}
\label{lemma:accessible}
Consider $w \in \Word_n$.
For any $ \Gamma \in \cL_n[w]$, 
$\pathiti(\Gamma)$ is accessible, i.e.,
belongs to $\Pathiti(w)$.
\end{lemma}

\begin{proof}
Consider $t_j < t_{j+1}$ and the arc $q_j^{-1} \Gamma|_{[t_j,t_{j+1}]}$.
Except for the modified domain, this arc belongs to
$\cLjojo(q_j^{-1}\Gamma(t_j); q_j^{-1}\Gamma(t_{j+1}))$
and therefore
$q_j^{-1}\Gamma(t_j) \in \Ac_{\eta\sigma_j}(q_j^{-1}\Gamma(t_{j+1}))$,
as desired.
\end{proof}

\begin{lemma}
\label{lemma:path}
Consider $w \in \Word_n$.
For any accessible path $(z_1,\ldots,z_\ell) \in \Pathiti(w)$, the set
$\{ \Gamma \in \cL_n^{[H^1]}[w] \;|\; \pathiti(\Gamma) = (z_1, \ldots, z_\ell) \}$
is nonempty and contractible.
\end{lemma}

\begin{proof}
The collection of sets $\{t_1 < \cdots < t_\ell \} \in \cH([0,1])$ is a contractible subset.
At this point, the values of $q_j \in \Quat_{n+1}$,
of $t_j < t_{j+1}$,
of $z_j \in q_j \Bru_{\longacute(\eta\sigma_j)}$
and of $z_{j+1} \in q_j \Bruchop$ 
with $q_j^{-1}z_j \in \Ac_{\eta\sigma_j}(q_j^{-1} z_{j+1})$
are all given. 
The set of locally convex arcs
$\Gamma: [t_j,t_{j+1}] \to \Spin_{n+1}$
with $\Gamma(t_j) = z_j$, $\Gamma(t_{j+1}) = z_{j+1}$
and $\Gamma(t) \in q_j \Bru_{\acute\eta}$ for all $t \in (t_j,t_{j+1})$
is homeomorphic to $\cLjojo(q_j^{-1}z_j; q_j^{-1}z_{j+1})$;
by Lemma \ref{lemma:cLjojo}, this set is contractible
(with an explicit contraction given by Remark \ref{rem:explicitcontraction}).
Concatenate the above arcs to construct $\Gamma$;
this yields the desired result.
\end{proof}

\begin{lemma}
\label{lemma:Path}
Consider $w \in \Word_n$. 
The set $\Pathiti(w) \subseteq \Bru_{B(w,1)} \times \cdots \times \Bru_{B(w,\ell)}$
is diffeomorphic to $\RR^d$,
$d = \inv(\eta\sigma_1) + \cdots + \inv(\eta\sigma_\ell)$.
In particular, $\Pathiti(w)$ is contractible (and nonempty).
\end{lemma}

\begin{proof}
Start constructing the set from the $\ell$-th coordinate $\Bru_{B(w,\ell)}$
and proceed backwards.
Use Lemma \ref{lemma:quasiproduct}
for the inductive step.
The set $\Pathiti(w)$ is parametrized by a quasiproduct.
\end{proof}


\begin{lemma}
\label{lemma:cLnw}
For any $w \in \Word_n$, the subset
\( \cL_n^{[H^1]}[w] \subset \cL_n^{[H^1]}(\acute\eta\hat w\acute\eta) \)
is contractible. 
\end{lemma}

\begin{proof}
We omit the superscript $[H^1]$ throughout the proof.  
Let $\Pathiti_1(w) \subset \cL_n[w]$ be the set of 
locally convex curves $\Gamma$
such that the arcs $\Gamma|_{[t_{i-1},t_i]}$ are 
assigned base points to the contractible sets 
$\cL_{n,\conv}(\Gamma(t_{i-1});\Gamma(t_i))$ 
(up to a reparameterization). 
Here we assume that $\sing(\Gamma) = \{t_1 < \cdots < t_\ell \}$;
we may use the construction in Remark \ref{rem:explicitcontraction}
to select a basepoint.

Lemma \ref{lemma:path} yields a deformation retract
from $\cL_n[w]$ to $\Pathiti_1(w)$,
a homotopy
$H_0: [0,1] \times \cL_n[w] \to \cL_n[w]$
which starts with an arbitrary curve $\Gamma_0 \in \cL_n[w]$
and deforms it through 
$\Gamma_s = H_0(s,\Gamma_0)$ for $s \in [0,1]$.
The homotopy satisfies
$\sing(\Gamma_s) = \sing(\Gamma_0) = \{ t_1 < \cdots < t_\ell \}$ and
$\pathiti(\Gamma_s) = \pathiti(\Gamma_0)$ for all $s \in [0,1]$.
We have $\Gamma_1 \in \Pathiti_1(w)$, i.e.,
the arcs $\Gamma_1|_{[t_{i-1},t_i]}$ are the base points of
the contractible sets
$\cL_{n,\conv}(\Gamma_0(t_{i-1});\Gamma_0(t_i))$.
Also, if $\Gamma_0 \in \Pathiti_1(w)$
then $\Gamma_s = \Gamma_0$ for all $s \in [0,1]$.

Let $\Pathiti_2(w) \subset \Pathiti_1(w)$
be the set of paths $\Gamma \in \Pathiti_1(w)$
such that $\sing(\Gamma) = \{ \frac{1}{\ell+1} < \cdots < \frac{\ell}{\ell+1} \}$.
There is an easy deformation retract
$H_1: [1,2] \times \Pathiti_1(w) \to \Pathiti_1(w)$
from $\Pathiti_1(w)$ to $\Pathiti_2(w)$:
affinely reparameterize each interval $[t_{i-1},t_i]$.

Lemma \ref{lemma:accessible} shows that
$\Pathiti_2(w)$ is homeomorphic to $\Pathiti(w)$:
the homeomorphism takes $\Gamma$ to $\pathiti(\Gamma)$.
Lemma \ref{lemma:Path} shows us how to construct a homotopy
$\tilde H_2: [2,3] \times \Pathiti(w) \to \Pathiti(w)$
with $\tilde H_2(2,\mathbf{z}) = \mathbf{z}$ and $\tilde H_2(3,\mathbf{z}) = \mathbf{z}_0$
where $\mathbf{z}_0 \in \Pathiti(w)$ is a base point.
Compose with the homeomorphism above to define
a deformation retract $H_2$ from $\Pathiti_2(w)$ to a point.
Concatenate $H_0, H_1, H_2$ to construct the desired contraction.
\end{proof}

\begin{rem}
\label{rem:exactsequence}
The proof of Lemma \ref{lemma:cLnw} above
obtains a rather explicit contraction.
A slightly shorter proof is possible using
the metrizable topological manifold structure provided in 
the proof of Theorem \ref{theo:stratification} below: 
use Theorem 15 of \cite{Palais} 
and the long exact sequence of homotopy groups for the fibration 
$\pathiti:\cL_n[w]\to\Pathiti(w)$, via  
Lemmas \ref{lemma:accessible}, \ref{lemma:path} and \ref{lemma:Path}. 
The longer proof above 
is more self-contained.
\end{rem}

Up to this point in this section, 
all arguments relied solely on the fact that 
$\cL_n^{[H^1]}(z_0;z_1)$ is a metrizable manifold including 
piecewise $C^1$ curves. 
Herein, by piecewise $C^1$ we mean that there exists a finite family 
of compact intervals $[0,t_1], [t_1,t_2], \ldots, [t_k,1]$
covering $[0,1]$
such that $\Gamma$ is of class $C^1$ in each interval $[t_i,t_{i+1}]$.
In certain situations though, 
we prefer to work in a space of curves whose derivatives 
of certain orders are well-defined. 
The Hilbert manifolds $\cL_n^{[H^r]}(z_0;z_1)$ and $\cL_n^{[H^r]}$, 
for $r>2$, were introduced in Subsection \ref{subsect:Hilbert} 
to fulfill this role. 
Therein, we prove that the inclusions 
$\cL_n^{[H^r]}(z_0;z_1)\hookrightarrow\cL_n^{[H^1]}(z_0;z_1)$ 
are homotopy equivalences homotopic to diffeomorphisms. 
We now verify that there exist similar stratifications for $r>2$:
\[\cL_n^{[H^r]}=\bigsqcup_{w\in\Word_n}\cL_n^{[H^r]}[w],\qquad
\cL_n^{[H^r]}[w]=\cL_n^{[H^1]}[w]\cap\cL_n^{[H^r]}.\] 
It turns out that the adjacency relations 
between strata are different in the two cases  
(compare Theorem \ref{theo:poset} and 
Equation \eqref{equation:acbHk}).

\begin{lemma}
\label{lemma:cLnwHk}
Consider $w \in \Word_n$, $r\in\NN$, $r>2$. The set
\( \cL_n^{[H^r]}[w] \subset \cL_n^{[H^r]}(\acute\eta\hat w\acute\eta) \)
is nonempty and contractible.
\end{lemma}

Fact \ref{fact:BST} is used in Section \ref{subsect:Hilbert} 
to prove that the inclusions 
$\cL_n^{[H^r]}(z_0;z_1)\hookrightarrow\cL_n^{[H^1]}(z_0;z_1)$ 
are weak homotopy equivalences. 
It is tempting to want to use the same fact to prove that 
$\cL_n^{[H^r]}[w]\subset\cL_n^{[H^1]}[w]$ is also a weak homotopy 
equivalence. 
This proof is not valid at this point, however, since we do not know 
that these sets are manifolds. 
We shall prove below that they are indeed topological manifolds, 
but this is not sufficient to apply Fact \ref{fact:BST}.

The idea is to imitate the proof of Lemma \ref{lemma:cLnw}. 
Notice that some of the building blocks, i.e., 
Lemmas \ref{lemma:accessible} and \ref{lemma:Path}, 
apply just the same to the case $r>2$. 
We need to state and prove results leading to 
an alternate version of Lemma \ref{lemma:path}. 
In order to do so, we introduce the concepts of $r$-\emph{jet} 
and \emph{enhanced path}. 
Given $r\in\NN^\ast$, $J\subset\RR$ an interval, 
$t\in J$ and $\Gamma:J\to\Spin_{n+1}$ 
locally convex of class $C^r$, 
we define the $r$-\emph{jet} of $\Gamma$ at $t$, 
$\jet^r(\Gamma;t)\in \Spin_{n+1}\times\RR^{nr}$, by 
\begin{equation}
\label{equation:jetHk} 
\jet^r(\Gamma;t)=(\Gamma(t), 
\kappa_1(t), \ldots, \kappa_1^{(r-1)}(t),\ldots,
\kappa_n(t),\ldots,\kappa_n^{(r-1)}(t)),
\end{equation}
where the real functions 
$\kappa_i(t)=\kappa_i(\Gamma;t)$ are 
described in Equation \eqref{equation:locallyconvex}
in the Introduction.
Notice that, given $\Gamma_0:[t_0,t_1]\to\Spin_{n+1}$ and 
$\Gamma_1:[t_1,t_2]\to\Spin_{n+1}$, 
$\Gamma_0\in\cL_n^{[H^r]}(z_0;z_1)$, 
$\Gamma_1\in\cL_n^{[H^r]}(z_1;z_2)$ 
(up to reparameterizations),  
the concatenation 
$\Gamma_0\ast\Gamma_1:[t_0,t_2]\to\Spin_{n+1}$ 
belongs to $\cL_n^{[H^r]}(z_0;z_2)$ 
(up to a reparameterization) if and only if 
$\jet^{r-1}(\Gamma_0;t_1)=\jet^{r-1}(\Gamma_1;t_1)$.
For $\sing(\Gamma)=\{t_1<\cdots<t_\ell\}\subset (0,1)$, 
define the \emph{enhanced path} of $\Gamma$ as 
\[\pathiti^{[H^r]}(\Gamma)=
(\jet^{r-1}(t_1),\ldots, \jet^{r-1}(t_\ell)).\] 
Notice that $\pathiti(\Gamma)$ is obtained from  
$\pathiti^{[H^r]}(\Gamma)$ by 
coordinate-wise application of the cartesian product projection 
$\Pi:\Spin_{n+1}\times\RR^{n(r-1)}\to\Spin_{n+1}$. 

Given $(z_1,\ldots, z_\ell)\in\Pathiti(w)$, we consider 
in the proof of Lemma \ref{lemma:path} the contractible 
set of curves 
$\cLjojo(q_j^{-1}z_j; q_j^{-1}z_{j+1})
\approx\cL_{n, \conv}^{[H^1]}(z_j;z_{j+1})$  
(the bijection is obtained by multiplication by $q_j\in\Quat_{n+1}$). 
Given jets $\bj_j\in\Spin_{n+1}\times\RR^{n(r-1)}$ with $\Pi(\bj_j)=z_j$, 
we are now interested in the subsets 
\[\cL_n^{[H^r]}(\bj_j;\bj_{j+1})\subset  
\cL_n^{[H^r]}(z_j;\bj_{j+1}), 
\cL_n^{[H^r]}(\bj_j;z_{j+1})
\subset\cL_n^{[H^r]}(z_j;z_{j+1}), \]
where, for instance, $\Gamma:[t_j,t_{j+1}]\to\Spin_{n+1}$, 
$\Gamma\in\cL_n^{[H^r]}(z_j;z_{j+1})$ 
(up to a reparameterization) 
belongs to $\cL_n^{[H^r]}(z_j;\bj_{j+1})$ if and only if 
$\jet^{r-1}(\Gamma;t_{j+1})=\bj_{j+1}$. 
In each case we consider the corresponding subset of convex curves: 
thus, for instance, $\cL_{n,\conv}^{[H^r]}(z_j;\bj_{j+1})$ 
is the subset of convex curves in $\cL_{n}^{[H^r]}(z_j;\bj_{j+1})$. 

\begin{lemma} 
\label{lemma:convexHk}
The subsets below are contractible connected components: 
\begin{gather*}
\cL_{n,\conv}^{[H^r]}(\bj_j;\bj_{j+1})
\subset\cL_{n}^{[H^r]}(\bj_j;\bj_{j+1}), \qquad
\cL_{n,\conv}^{[H^r]}(z_j;\bj_{j+1})
\subset\cL_{n}^{[H^r]}(z_j;\bj_{j+1}), \\ 
\cL_{n,\conv}^{[H^r]}(\bj_j;z_{j+1})
\subset\cL_n^{[H^r]}(\bj_j;z_{j+1}), \qquad  
\cL_{n, \conv}^{[H^r]}(z_j;z_{j+1})
\subset\cL_n^{[H^r]}(z_j;z_{j+1}). 
\end{gather*} 
\end{lemma}

\begin{proof}
In this proof we use the following notation:
$\mathbf{H}^{r,n} = (H^{r-1}([t_j,t_{j+1}];\RR))^n$. 
We also consider $\cL_n^{[H^r]}(z_j;z_{j+1})$ 
as a submanifold of $\mathbf{H}^{r,n}$, 
given by the functions $\xi_1,\ldots,\xi_n$ of Equation 
\eqref{equation:xis}. 
For $\mathbf{B}_1=\mathbf{H}^{r,n}$ 
and 
$\mathbf{B}_2=\mathbf{H}^{1,n}$, 
Fact \ref{fact:BST} implies that the inclusion 
\[ \cL_{n}^{[H^r]}(z_j;z_{j+1})
\subset\cL_n^{[H^1]}(z_j;z_{j+1})\]
is a homotopy equivalence 
between Hilbert manifolds.  
Since $\cL_{n, \conv}^{[H^1]}(z_j;z_{j+1})$ is a contractible 
connected component, so is $\cL_{n, \conv}^{[H^r]}(z_j;z_{j+1})$, 
as we already pointed out in 
the proof of Lemma \ref{lemma:convex2}. 

Now, we deal with the inclusion 
$\cL_{n,\conv}^{[H^r]}(\bj_j;z_{j+1})
\subset\cL_{n}^{[H^r]}(\bj_j;z_{j+1})$.  
We want to use Fact \ref{fact:BST}. 
We have however not a linear subspace, 
but an affine subspace. 
This requires a minor adaptation.  
Take
\[ \mathbf{B}_1=
\left\{
(\xi_1,\ldots, \xi_n)\in\mathbf{H}^{r,n} 
\,\left\vert\, 
\forall i, \in\nmesmo\,
\left(\xi_i(t_j)=\xi_i'(t_j)=\cdots=\xi_i^{(r-2)}(t_j)=0\right.\right)
\right\}  \]
and $\mathbf{B}_2=\mathbf{H}^{1,n}$.
Take $\tilde\kappa\in C^{\infty}([t_j,t_{j+1}], \RR^n)$ with
$\bj_j=(z_j,\tilde\kappa_1(t_j),\ldots,\tilde\kappa_n^{(r-2)}(t_j))$ 
and the corresponding $\tilde\xi=(\tilde\xi_1,\ldots,\tilde\xi_n)$. 
Consider the translated submanifold 
$M_2=\cL_n^{[H^1]}(z_j;z_{j+1}) - \tilde\xi \subset\mathbf{B}_2$. 
Apply Fact \ref{fact:BST} in order to obtain the desired conclusion. 
The other cases are similar.
\end{proof}

\begin{proof}[Proof of Lemma \ref{lemma:cLnwHk}]
Let $\Pathiti^{[H^r]}(w)=\Pathiti(w)\times\RR^{\ell nr}$ 
be the contractible set of accessible enhanced paths, 
defined in the obvious manner (here, $\ell=\ell(w)$). 
Lemma \ref{lemma:convexHk} shows that the set 
of $H^r$ locally convex curves with 
a prescribed enhanced path is contractible. 
Thus, the map from $\cL_n^{[H^r]}[w]$ to $\Pathiti^{[H^r]}(w)$ 
taking a curve $\Gamma$ to its enhanced path is a fibration 
with a fiber homemorphic to the separable Hilbert space and 
base space homeomorphic to an Euclidean space, 
proving our claim.
\end{proof}


We now present a smooth example of tubular neighborhood, used in the proof of 
Theorem \ref{theo:stratification}. 
We use the notation of Remark \ref{rem:collared}. 



\begin{rem}
\label{rem:pathcoordinates}
For all $z_0\in\widetilde\B^+_{n+1}$, the open set $\cU_{z_0}$ is a 
smooth tubular neighborhood in $\Spin_{n+1}$ 
of the signed Bruhat cell $\Bru_{z_0}$ (with $B = \RR^k$). 
We denote its smooth projection map by 
$\Pi_{z_0}: \cU_{z_0} \to \Bru_{z_0}\subseteq\cU_{z_0}$. 
Write $z_0=q\acute\sigma$ for $q\in\Quat_{n+1}$ and 
$\sigma\in S_{n+1}$. 
If $\sigma\neq\eta$, we have that $\Bru_{z_0}$ is a 
signed Bruhat cell of $\Spin_{n+1}$ with positive codimension 
$k=\inv(\eta)-\inv(\sigma)$. 
In this case, there is a smooth submersion  
$f_{z_0} = (f_{z_0,1}, \ldots, f_{z_0,k}): \cU_{z_0} \to \RR^k$ 
satisfying the following conditions: 
\begin{enumerate}
\item\label{item:submanifold}
{$\Bru_{z_0}=f_{z_0}^{-1}(0)=\{z\in\cU_{z_0}\,\vert\,f_{z_0}(z)=0\}$;}
\item\label{item:tubular}
{$(\Pi_{z_0},f_{z_0}):\cU_{z_0}\to\Bru_{z_0}\times\,\RR^k$ 
is a smooth diffeomorphism;} 
\item\label{item:transverse}{Given a locally convex curve 
$\Gamma:(-\epsilon,\epsilon)\to\cU_{z_0}$, 
if $\Gamma$ is differentiable in $t$, then $(f_{z_0,k}\circ\Gamma)'(t)>0$. 
}
\end{enumerate}

The pair of maps $(\Pi_{z_0},f_{z_0})$ is explicitly constructed 
in Theorem 2 of \cite{Goulart-Saldanha0} 
using a triangular system of coordinates 
(see also Remark 6.7 therein).
\end{rem}

For the proof of Theorem \ref{theo:stratification} below, 
we also need the following technical result.
Notice that the proof uses the concept of positivity
(see Subsection \ref{subsect:Bruhat}).

\begin{lemma}
\label{lemma:singleletter}
Consider $\sigma \in S_{n+1}$, $\sigma \ne \eta$
and $z_0 = q \acute\sigma \in \tilde \B_{n+1}^{+}$, $q \in \Quat_{n+1}$.
If $\Gamma: [-\epsilon,\epsilon] \to \cU_{z_0}$ is locally convex
and there exists $t_1 \in [-\epsilon,\epsilon]$ with 
$\Gamma(t_1) \in \Bru_{z_0}$ then $\sing(\Gamma) = \{t_1\}$.
\end{lemma}

\begin{proof}
Consider a projective transformation $\phi$ for which
$\phi(z_0) \in q \bQ[\Pos_\sigma] \subset \cU_q$.
By continuity, there exists an open set $A \subset \cU_{z_0}$, $z_0 \in A$,
such that $\phi[A] \subset \cU_q$.
We may furthermore assume that $z \in \phi[A] \cap \Bru_\sigma$
implies $z \in q \bQ[\Pos_\sigma]$.

Consider $\Gamma$ as in the statement.
Apply triangular coordinates to $\cU_{z_0}$ 
to define the convex curve
$\Gamma_L: [-\epsilon,\epsilon] \to \Lo_{n+1}^{1}$,
$\Gamma(t) = z_0 \bQ(\Gamma_L(t))$.
For $\lambda \in [1,+\infty)$, consider 
the projective transform  
\[ \Gamma_L^\lambda(t) = 
\diag(1,\lambda^{-1},\ldots,\lambda^{-n})
\Gamma_L(t) 
\diag(1,\lambda,\ldots,\lambda^{n}) \]
and the locally convex curve 
$\Gamma^\lambda(t) = z_0 \bQ(\Gamma_L^\lambda(t))$.
Notice that $\Gamma^\lambda: [-\epsilon,\epsilon] \to  \cU_{z_0}$
satisfies $\sing(\Gamma^\lambda) = \sing(\Gamma)$
and $\iti(\Gamma^\lambda) = \iti(\Gamma)$.
Given $t_0 \in [-\epsilon,\epsilon]$, we have
$\lim_{\lambda \to +\infty} \Gamma^\lambda(t_0) = z_0$;
by compactness, there exists $\lambda_0$ such that
$\Gamma^{\lambda_0}[[-\epsilon,\epsilon]] \subset A$.
The curve $\tilde\Gamma = \phi\circ\Gamma^{\lambda_0}$ therefore
admits triangular coordinates
$\tilde\Gamma_L: [-\epsilon,\epsilon] \to \Lo_{n+1}^{1}$,
$\tilde\Gamma(t) = q \bQ[\tilde\Gamma_L(t)]$.
We have $\tilde\Gamma(t_1) \in \phi[A] \cap \Bru_\sigma$
and therefore $\tilde\Gamma_L(t_1) \in \Pos_\sigma$.
From Lemma 5.7 of \cite{Goulart-Saldanha0},
$t > t_1$ implies $\tilde\Gamma_L(t) \in \Pos_\eta$,
i.e., $\sing(\tilde\Gamma) \cap (t_1,\epsilon] = \emptyset$.
Thus $t_1$ is the last element of $\sing(\Gamma)$. 
A similar argument using the sets $\Neg_{\ast}$ instead of 
$\Pos_{\ast}$ proves that $t_1$ is also the first element of 
$\sing(\Gamma)$.
\end{proof}



\begin{proof}[Proof of Theorem \ref{theo:stratification}]
The nonemptiness and contractibility of $\cL_n^{[H^r]}[w]$ 
is already established by previous lemmas in this section. 
It remains to be shown that, 
for $r\neq 2$,  
$\cL_n^{[H^r]}[w]$ is a 
globally collared topological submanifold of $\cL_n^{[H^r]}(q_w)$, 
$q_w=\acute\eta\hat w \acute\eta\in\Quat_{n+1}$, 
with codimension $\dim(w)$; 
also that, if $r>2$, then $\cL_n^{[H^r]}[w]$ is in fact an 
embedded submanifold of differentiability class $C^{r-1}$. 

For $w = \sigma_1\cdots\sigma_\ell = (\sigma_1, \ldots, \sigma_\ell)$
and $2j \in \ZZ \cap [0,2\ell+2]$,
set $B(w,j) \in \widetilde \B_{n+1}^{+}$ 
as in Equation \eqref{eq:Bwj} above;
in particular, $B(w,\ell+1) = q_w = \acute\eta\hat w\acute\eta \in \Quat_{n+1}$.
We first define an open subset 
$\cA_w^\sharp \subset \cL_n^{[H^r]}(q_w) \times (0,1)$.
A pair $(\Gamma,\tilde\epsilon)$ belongs to $\cA_w^\sharp$ if there
exist $0 = \tilde t_0 < \tilde t_1 < \cdots
< \tilde t_\ell < \tilde t_{\ell+1} = 1$ such that:
\begin{enumerate}[label=(\roman*)]
\item{For each $i$, $\tilde t_{i+1} - \tilde t_i > 8\tilde\epsilon$.}
\item\label{item:Bwi}{Each arc $\Gamma|_{[\tilde t_{i}-2\tilde\epsilon,
\tilde t_{i}+2\tilde\epsilon]}$ is convex,
with image in $\cU_{B(w,i)}$.
In particular, for $\tilde z_i = \Gamma(\tilde t_i)$, we have
$\tilde z_i \in \cU_{B(w,i)}$. }
\item{Each arc $\Gamma|_{[\tilde t_{i}+\tilde\epsilon,
\tilde t_{i+1}-\tilde\epsilon]}$ is convex,
with image in $\cU_{B(w,i+\frac12)} = \Bru_{B(w,i+\frac12)}$.}
\item{For $f_{i,k_i}: \cU_{B(w,i)} \to \RR$ and $k_i$
as in Remark \ref{rem:pathcoordinates}, 
we have $f_{i,k_i}(\tilde z_i) = 0$.}
\item{Let $\Pi_{B(w,i)}: \cU_{B(w,i)} \to \Bru_{B(w,i)} \subset \cU_{B(w,i)}$
be the smooth projection of Remark \ref{rem:pathcoordinates}. 
Set $\check z_i = \Pi_{B(w,i)}(\tilde z_i)$.
There exist convex arcs in $\cU_{B(w,i)}$
from $\Gamma(\tilde t_{i}-\frac{\tilde\epsilon}{2})$ to $\check z_i$
and from $\check z_i$ to $\Gamma(\tilde t_{i}+\frac{\tilde\epsilon}{2})$.}
\end{enumerate}
For $\Gamma \in  \cL_n^{[H^r]}(q_w)$, set
\( J_\Gamma = \{ \tilde\epsilon \in (0,1) \;|\; 
(\Gamma,\tilde\epsilon) \in \cA_w^\sharp \} \);
clearly, $J_\Gamma$ is either an open interval or empty;
for $\Gamma \in \cL_n^{[H^r]}[w]$,
$J_\Gamma$ is an interval of the form $J_\Gamma = (0,\epsilon)$
for some $\epsilon > 0$.
Set 
\[ \cA_w = \{ \Gamma \in \cL_n^{[H^r]}(q_w) \;|\; J_\Gamma \ne \emptyset \}
\subseteq \cL_n^{[H^r]}(q_w), \]
an open subset.
For $\Gamma \in \cA_w$, the times $\tilde t_i$ are well-defined
and, from Condition \ref{item:transverse} of 
Remark \ref{rem:pathcoordinates}, 
the functions $\Gamma \mapsto \tilde t_i$ are 
of class $C^{r-1}$. 
For $\Gamma \in \cA_w$, select
$\epsilon = \epsilon_\Gamma \in J_\Gamma$;
from the $C^{r-1}$ regularity of $\tilde t_i$ and 
several uses of Lemma 5.5  
of \cite{Goulart-Saldanha0},
the function $\Gamma \mapsto \epsilon$
can be taken to be of class $C^{r-1}$.
We have $\tilde z_i = \Gamma(\tilde t_i) \in \cU_{B(w,i)}$; 
we define $\tilde z_i^{-} = \Gamma(\tilde t_{i}-\frac{\epsilon}{2})$,
$\tilde z_i^{+} = \Gamma(\tilde t_{i}+\frac{\epsilon}{2})$,
and $\check z_i = \Pi_{B(w,i)}(\tilde z_i)$.  
Also, the map 
$\Gamma \mapsto (\tilde z_i, \tilde z_i^{-}, \tilde z_i^{+}, \check z_i)$
is of class $C^{r-1}$.
For $\Gamma \in \cA_w$, $\epsilon = \epsilon_\Gamma$,
$\tilde t_i$, $\tilde z_i$ and $\check z_i$
as above we therefore have the following properties:
\begin{enumerate}[label=(\alph*)]
\item{For each $i$, $\tilde t_{i+1} - \tilde t_i \ge 8\epsilon$.}
\item{Each arc $\Gamma|_{[\tilde t_{i}-\epsilon,
\tilde t_{i}+\epsilon]}$ is convex,
with image in $\cU_{B(w,i)}$; also, $\tilde z_i \in \cU_{B(w,i)}$. }
\item{Each arc $\Gamma|_{[\tilde t_{i}+\epsilon,
\tilde t_{i+1}-\epsilon]}$ is convex,
with image in $\cU_{B(w,i+\frac12)}$.}
\item{For $f_{i,k_i}: \cU_{B(w,i)} \to \RR$ 
as in Remark \ref{rem:pathcoordinates},
we have $f_{i,k_i}(\tilde z_i) = 0$.}
\item{There exist convex arcs in $\cU_{B(w,i)}$
from $\tilde z_i^{-}$ to $\check z_i$
and from $\check z_i$ to $\tilde z_{i}^{+}$.}
\end{enumerate}

Set $d=\dim(w)=d_1+\cdots+d_\ell$, where $d_i = k_i - 1$. 
Define $F: \cA_w \to \RR^d$, 
$F(\Gamma)=(F_1(\Gamma), \ldots, F_\ell(\Gamma))$, 
where $F_i: \cA_w \to \RR^{d_i}$, 
$F_i(\Gamma) = (f_{i,1}(\tilde z_i), \ldots, f_{i,d_i}(\tilde z_i))$.
The coordinate functions $f_{i,j}$ above are 
the first $d_i$ coordinate functions of the smooth submersion 
$f_{B(w,i)}=(f_{i,1},\ldots,f_{i,k_i}):\cU_{B(w,i)}\to\RR^{k_i}$  
of Remark \ref{rem:pathcoordinates}. 

We claim that, for $\Gamma \in \cA_w$,
$F(\Gamma) = 0$ if and only if $\Gamma \in \cL_n^{[H^r]}[w]$.

Indeed, if $\Gamma \in \cA_w$ and $F(\Gamma) = 0$ we have
$\tilde z_i = \Gamma(\tilde t_i) \in \Bru_{\eta\sigma_i}$.
We already know that 
$\{\tilde t_1 < \cdots < \tilde t_\ell \} \subseteq
\sing(\Gamma) \subset 
\bigcup_i (\tilde t_i-\epsilon,\tilde t_i+\epsilon)$.
By Lemma \ref{lemma:singleletter} we have
$\sing(\Gamma) = \{ \tilde t_1 < \cdots <  \tilde t_{\ell} \}$
and therefore $\iti(\Gamma) = w$.

For $r>2$, the Regular Value Theorem applied to the 
submersion $F$ shows that $\cL_n^{[H^r]}[w]$ is an 
embedded submanifold of $\cL_n^{[H^r]}(q_w)$ 
of class $C^{r-1}$ and codimension $d=\dim(w)$, as claimed. 
The set $\cA_w$ is a promising candidate for a tubular neighborhood:
all we would have to do is to construct a projection.
We prefer, however, to use the normal bundle.
Indeed, there is a well-defined tubular neighborhood, 
i.e., a $C^{r-1}$ embedding 
of the normal bundle 
$N\cL_n^{[H^r]}[w]\hookrightarrow\cA_w$ 
that extends the inclusion of $\cL_n^{[H^r]}[w]$ 
(regarded as the zero section of its normal bundle) 
in $\cA_w$. 
This tubular neighborhood is foliated by normal sections. 

We now deal with the case $r=1$. 
We construct a projection
$\Pi: \cA_w \to \cL_n^{[H^1]}[w] \subset \cA_w$.
Given $\Gamma \in \cA_w$, the curve $\check\Gamma = \Pi(\Gamma)$,
$\check\Gamma: [0,1] \to \Spin_{n+1}$,
will coincide with $\Gamma$ except in the intervals
$[\tilde t_i -\frac{\epsilon}{2}, \tilde t_i + \frac{\epsilon}{2}]$
and will satisfy $\check\Gamma(\tilde t_i) = \check z_i$.
The restrictions
$\check\Gamma|_{[\tilde t_i - \frac{\epsilon}{2}, \tilde t_i]}$ and
$\check\Gamma|_{[\tilde t_i, \tilde t_i+\frac{\epsilon}{2}]}$ 
will be convex arcs contained in $\cU_{B(w,i)}$
joining $\tilde z_i^{-}$ to $\check z_i$
and $\check z_i$ to $\tilde z_i^{+}$, respectively.
These two convex arcs are obtained from the convex arcs
$\Gamma|_{[\tilde t_i - \frac{\epsilon}{2}, \tilde t_i]}$ and
$\Gamma|_{[\tilde t_i, \tilde t_i+\frac{\epsilon}{2}]}$ 
by projective transformations as follows. 
We have 
$\tilde z_i\,,\check z_i\in\tilde z_i^-\Bru_{\acute\eta}$; 
take $U\in\Up^1_{n+1}$ such that 
$((\tilde z_i^-)^{-1}\tilde z_i)^U= 
(\tilde z_i^-)^{-1}\check z_i$ and set 
$\check\Gamma|_{[\tilde t_i - \frac{\epsilon}{2}, \tilde t_i]}
=\tilde z_i^-((\tilde z_i^-)^{-1}\Gamma)^U|_{[\tilde t_i - \frac{\epsilon}{2}, \tilde t_i]}$. 
The convex arc 
$\check\Gamma|_{[\tilde t_i, \tilde t_i+\frac{\epsilon}{2}]}$ 
is obtained likewise. 

The function $\Gamma \to \epsilon$ can be constructed
so as to satisfy $\epsilon_{\check\Gamma} = \epsilon_{\Gamma}$.
If $\Gamma \in \cL_n[w]$ we have $\check z_i = \tilde z_i$
and therefore $\check\Gamma = \Gamma$.

We now have a continuous map
$(\Pi,F): \cA_w \to \cL_n^{[H^1]}[w] \times \RR^d$;
let $\cB_w \subseteq \cL_n^{[H^1]}[w] \times \RR^d$ be its image.
We construct the inverse map $\Phi: \cB_w \to \cA_w$;
in the process we see that the set
$\cB_w$ is an open neighborhood of $\cL_n^{[H^1]}[w] \times \{0\}$.
Indeed, given $\check\Gamma \in \cL_n^{[H^1]}[w]$
construct $\epsilon$, $\tilde t_i$,
$\check z_i = \check\Gamma(\tilde t_i)$
and $\tilde z_i^{\pm} = \check\Gamma(\tilde t_i \pm \frac{\epsilon}{2})$
as above.
Given $\bfx = (\bfx_1,\ldots,\bfx_\ell) \in \RR^{d_1+\cdots+d_\ell}$,
there exist unique $\tilde z_i \in \cU_{B(w,i)}$ 
with $\Pi_{B(w,i)}(\tilde z_i) = \check z_i$
and $f_{B(w,i)}(\tilde z_i) = \bfx_i$.
If there exist convex arcs contained in $\cU_{B(w,i)}$
from $\tilde z_i^{-}$ to $\tilde z_i$ and
from $\tilde z_i$ to $\tilde z_i^{+}$ then
$(\check\Gamma,\bfx) \in \cB_w$ and the curve
$\tilde\Gamma = \Phi(\check\Gamma,\bfx)$ is constructed as before.
More precisely,
$\tilde\Gamma$ coincides with $\check\Gamma$ except
in the intervals
$[\tilde t_i-\frac{\epsilon}{2},\tilde t_i+\frac{\epsilon}{2}]$.
The convex arcs 
$\tilde\Gamma|_{[\tilde t_i - \frac{\epsilon}{2}, \tilde t_i]}$ and
$\tilde\Gamma|_{[\tilde t_i, \tilde t_i+\frac{\epsilon}{2}]}$ 
are obtained from the arcs
$\check\Gamma|_{[\tilde t_i - \frac{\epsilon}{2}, \tilde t_i]}$ and
$\check\Gamma|_{[\tilde t_i, \tilde t_i+\frac{\epsilon}{2}]}$ 
by projective transformations.

In the proof of Theorem 2 
of \cite{Goulart-Saldanha0}, it was shown that 
there exists a natural diffeomorphism between 
$\cU_{B(w,i)}$ and the cartesian product 
$\Up_{\eta\sigma_i} \times \Lo_{\sigma_i^{-1}}$
of certain affine spaces of triangular matrices 
defined in Equation (4)  
of \cite{Goulart-Saldanha0}. 
Endow the affine spaces $\Up_{\eta\sigma_i}$ and $\Lo_{\sigma_i^{-1}}$ with the natural Euclidean metrics  
coming from the sets of free coordinates 
(i.e., entries not obligatorily equal to $0$ or $1$).  
Use the above diffeomorphism to endow  $\cU_{B(w,i)}$
with a flat Euclidean metric.
Similarly, endow the cartesian product
$\cU^3_{B(w,i)} =  \cU_{B(w,i)} \times \cU_{B(w,i)} \times \cU_{B(w,i)}$
with a flat Euclidean metric.
Let $\cW_i \subset \cU^3_{B(w,i)}$
be the open set of triples $(z_i^{-},z_i,z_i^{+})$ such that
there exist convex arcs contained in $\cU_{B(w,i)}$ 
from $z_i^{-}$ to $z_i$ and from $z_i$ to $z_i^{+}$.
Let $\delta_i: \cW_i \to (0,+\infty)$ be the continuous function
taking a triple  $(z_i^{-},z_i,z_i^{+}) \in \cW_i$
to one half of the distance
(in the flat Euclidean metric constructed above)
from the complement $\cU^3_{B(w,i)} \smallsetminus \cW_i$, i.e.,
$\delta_i(z_i^{-},z_i,z_i^{+}) =
\frac12 d((z_i^{-},z_i,z_i^{+}), \cU^3_{B(w,i)} \smallsetminus \cW_i)$.
Given $\check\Gamma \in \cL_n^{[H^1]}[w]$,
define $\delta(\check\Gamma) =
\min_i \delta_i(\tilde z_i^{-},\check z_i,\tilde z_i^{+})$
where, as above, $\tilde z_i^{\pm} = \Gamma(\tilde t_i\pm \frac{\epsilon}{2})$.
Notice that $\delta: \cL_n^{[H^1]}[w] \to (0,+\infty)$ is continuous
and that if $|\bfx| \le \delta(\Gamma)$ then
$(\Gamma,\bfx) \in \cB_w$ (by construction).

Let $\BB^d \subset \DD^d \subset \RR^d$ be the open and closed balls of radius $1$, respectively.
Define $\hat\Phi: \cL_n^{[H^1]}[w] \times \DD^d \to \cA_w$ by
$\hat\Phi(\Gamma,\bfx) = \Phi(\Gamma,\delta(\Gamma) \bfx)$.
Let $\hat\cA_w = \hat\Phi\left[\cL_n^{[H^1]}[w] \times \BB^d\right] \subset \cA_w$
and define $\hat F: \hat\cA_w \to \BB^d$
so that $(\Pi,\hat F) = \hat\Phi^{-1}: 
\hat\cA_w \to \cL_n^{[H^1]}[w] \times \BB^d$.
This completes the construction of
the tubular neighborhood of $\cL_n^{[H^1]}[w]$.
\end{proof}

\section{Transversal sections}
\label{sect:transversal}

The proof of Theorem \ref{theo:stratification} 
implicitly gives us (topologically) transversal sections.
We now construct \emph{explicit} transversal sections  
to $\cL_n^{[H^1]}[w]$ in $\cL_n^{[H^1]}$. 
We omit the superscript $[H^1]$ throughout this section. 
The construction roughly corresponds to going back
to Theorem \ref{theo:stratification}, 
then to Theorem 2
and Remark 6.7 
of \cite{Goulart-Saldanha0},
then to Lemma 5.5 
and Remark 5.6 
of \cite{Goulart-Saldanha0}, 
and following the steps.
A key difference is that strictly following 
the proof of Theorem \ref{theo:stratification} 
given in Section \ref{sect:paths} 
yields curves which fail to be smooth 
precisely at the times $\tilde t_i$ 
(defined in the mentioned proof);
the curves produced by our construction in this section
are smooth (indeed algebraic) in a neighborhood of $\tilde t_i$ 
(though they are not globally smooth). 
A standard procedure can be applied to 
smoothen out the tranversal section 
$\phi:\DD^d\to\cL_n^{[H^1]}$ 
to $\cL_n^{[H^1]}[w]$ just constructed. 
The result is, for each $r\geq 3$, a smooth map 
$\tilde\phi:\DD^d\to\cL_n^{[H^r]}$ such that 
$\tilde\Phi:\DD^d\times[0,1]\to\Spin_{n+1}$, 
$\tilde\Phi(\bfx,t)=\tilde\phi(\bfx)(t)$, 
coincides with $\Phi(\bfx,t)=\phi(\bfx)(t)$ (and hence is algebraic) 
in $\DD^d\times(\cup_i (\tilde t_i-\delta,\tilde t_i+\delta))$ 
for some $\delta>0$. 
Of course, $\tilde\phi$ is tranversal to $\cL_n^{[H^r]}[w]$. 
We first present the construction as an algorithm,
then provide examples. 


Consider $\sigma \in S_{n+1}$, $\sigma \ne e$, $\rho = \eta\sigma$ and
$d = \dim(\sigma) = \inv(\sigma) - 1$.
Consider $z_0 = q \acute\eta \acute\sigma \in \widetilde \B_{n+1}^{+}$, 
$q \in \Quat_{n+1}$,
so that $\chop(z_0) = q \acute\eta$
and $\adv(z_0) = q \acute\eta \hat\sigma$.
Let $Q_0 = \Pi(z_0) \in \B_{n+1}^{+}\subset\SO_{n+1}$.
We first construct an explicit transversal section
$\psi: \RR^{d+1} \to \SO_{n+1}$
to the Bruhat cell $\Bru_{Q_0}=\Pi[\Bru_{z_0}]\subset \SO_{n+1}$
passing through $Q_0 = \psi(0)$
(compare with  
Remarks 5.6 and 6.7 
of \cite{Goulart-Saldanha0}).
First we define a matrix
$\tilde M \in (\RR[x_1,\ldots,x_{d+1}])^{(n+1)\times (n+1)}$ 
with polynomial entries in the variables $x_l$, $1 \le l \le d+1$.
For $i \in \nmaisum$,
set $(\tilde M)_{i,i^\rho} = (Q_0)_{i,i^\rho} = \pm 1$.
There are $d+1$ zero entries in $Q_0$ which are simultaneously
below a nonzero entry and to the left of a nonzero entry:
these are the pairs $(i,j)$ for which
$j < i^\rho$ and $j^{\rho^{-1}} < i$.
Assign to each such position $(i,j)$ an integer $l$ from $1$ to $d+1$ 
in the same order you would read or write them on a page
(top to bottom and left to right).
For each such position $(i,j)$,
set $(\tilde M)_{i,j} = (Q_0)_{i,i^\rho} x_l$.
The other entries of $\tilde M$ are set to $0$:
this defines the desired matrix
$\tilde M \in (\RR[x_1,\ldots,x_{d+1}])^{(n+1)\times (n+1)}$
or, equivalently, a smooth map 
$\psi_L: \RR^{d+1} \to \GL^{+}_{n+1}$
where $\psi_L(\bfx)$ is obtained
by evaluating $\tilde M$ at $\bfx \in \RR^{d+1}$.
As an example, 
the matrices below correspond 
respectively to $n = 2$, $\sigma_0 =  [321]=aba$ ($d = 2$), 
and $n = 3$, $\sigma_1 =  [3142]=acb$ ($d = 2$):
\[ \tilde M_0 = \begin{pmatrix}
1 & 0 & 0 \\
x_1 & 1 & 0 \\
x_2 & x_3 & 1 \end{pmatrix}; \qquad
\tilde M_1 = \begin{pmatrix}
0 & -1 & 0 & 0 \\
0 & -x_1 & 0 &- 1 \\
-1 & 0 & 0 & 0 \\
x_2 & x_3 & 1 & 0 \end{pmatrix} \]
(we take $q=1$ in both examples).
Notice that the map $\psi_L$ is a smooth diffeomorphism 
from $\RR^{d+1}$ to
$Q_0 \Lo_{\sigma^{-1}} \subset \GL^{+}_{n+1}$ 
(see Equation (4) of 
\cite{Goulart-Saldanha0} for the definition 
of $\Lo_{\sigma^{-1}}$). 
Recall that we denote by $\bQ:\GL^+_{n+1}\to\SO_{n+1}$  
the map that takes $M$ to the orthogonal part $\bQ(M)$ 
in the $QR$ decomposition 
$M=\bQ(M)R$, $R\in\Up^+_{n+1}$.
The smooth algebraic map $\psi_A = \bQ \circ \psi_L: \RR^{d+1} \to \SO_{n+1}$
is the desired transversal section to the Bruhat cell $\Bru_{Q_0}$.
In order to define $\psi: \RR^{d+1} \to \Spin_{n+1}$,
$\psi_A = \Pi \circ \psi$,
lift the map $\psi_A$ starting at $\psi(0) = z_0$.

Consider $\RR^d \subset \RR^{d+1}$ defined by $x_{d+1} = 0$.
Let $\fn = \sum_i \fl_i$ be the lower triangular nilpotent matrix
whose only nonzero entries are $\fn_{j+1,j} = 1$
(see Equation \eqref{equation:nfhLfh}).
For each $\bfx \in \RR^{d}$, define a curve
$\phi_L(\bfx;\cdot): \RR \to Q_0 \Lo_{n+1}^{1} \subset \GL^{+}_{n+1}$
by the IVP
\[ \frac{\partial}{\partial t}\phi_L(\bfx;t) = \phi_L(\bfx;t) \fn, \quad
\phi_L(\bfx;0) = \psi_L(\bfx), \]
so that $\phi_L(\bfx;t) = \psi_L(\bfx) \exp(t\fn)$.
Since entries of $\phi_L(\bfx;t)$ are polynomials in $\bfx$ and $t$,
we may equivalently consider the matrix
$M \in (\RR[\bfx;t])^{(n+1)\times (n+1)}$,
$M(\bfx,t) = \phi_L(\bfx;t)$,
whose entries are polynomials in $\bfx$ and $t$,
of degree at most $n$ in the variable $t$ and satisfying
\[ (M)_{i,j+1} = \frac{\partial}{\partial t} (M)_{i,j}. \]
As an example, the two matrices below again correspond to
$n = 2$,  $\sigma_0 = [321]=aba$  and $n = 3$,  
$\sigma_1 =  [3142]=acb$ (and $q=1$ in both cases):
\begin{equation}
\label{equation:M0M1}
M_0 = \begin{pmatrix}
1 & 0 & 0 \\
t+x_1 & 1 & 0 \\
\frac{t^2}{2}+x_2 & t & 1
\end{pmatrix}; \qquad
M_1 = \begin{pmatrix}
-t & -1 & 0 & 0 \\
-\frac{t^3}{6} - x_1 t & -\frac{t^2}{2} - x_1 & - t & 1 \\
-1 & 0 & 0 & 0 \\
\frac{t^2}{2} + x_2 & t & 1 & 0
\end{pmatrix}.
\end{equation}
Notice that, given $\bfx \in \RR^d$,
the map $Q_0^{-1} \phi_L(\bfx;\cdot): \RR \to \Lo_{n+1}^{1}$
is a smooth (indeed algebraic) convex curve.
Let $\Gamma_{\bfx}: \RR \to \Spin_{n+1}$
be the locally convex curve defined by
$\Gamma_{\bfx}(t) = \bQ(\phi_L(\bfx,t))$, $\Gamma_{\bfx}(0) = \psi(\bfx)$.
Clearly, $\Gamma_0(0) = z_0$,
$\Gamma_0(t) \in \Bru_{\chop(z_0)}$ for $t < 0$ and
$\Gamma_0(t) \in \Bru_{\adv(z_0)}$ for $t > 0$.

We now construct the desired transversal surface
$\phi: \DD^d \to \cL_n$.
Choose $z_0$ above such that $\chop(z_0) = \acute\eta$
and $\adv(z_0) = \acute\eta\hat\sigma$;
let $q_1 = \acute\eta\hat\sigma\acute\eta \in \Quat_{n+1}$.
For sufficiently small $r \in (0,\frac{\pi}{4})$,
there exists a convex arc contained in $\Bru_{\chop(z_0)}$
going from $\exp(r \fh)$ to $\Gamma_0(-r)$. 
For instance, the reader may use Lemma 6.1 
of \cite{Goulart-Saldanha0} 
and 
a projective transformation to obtain such a convex arc.
Similarly, for sufficiently small $r \in \left(0,\frac{\pi}{4}\right)$,
there exists a convex arc contained in $\Bru_{\adv(z_0)}$
going from $\Gamma_0(r)$ to $q_1 \exp(-r \fh)$.
Fix such a small $r \in \left(0,\frac{\pi}{4}\right)$.
By continuity, there exists a small $\tilde s > 0$ such that,
if $|\bfx| \le \tilde s$ then there exists a convex arc
contained in $\Bru_{\chop(z_0)}$
going from $\exp(r \fh)$ to $\Gamma_{\bfx}(-r)$.
Similarly, for sufficiently small $\tilde s > 0$, 
if $|\bfx| \le \tilde s$ then there exists a convex arc
contained in $\Bru_{\adv(z_0)}$
going from $\Gamma_{\bfx}(r)$ to $q_1 \exp(-r \fh)$.
Fix such a small $\tilde s > 0$.
We thus define, for each $\bfx \in \DD^d$, convex arcs
$\tilde \phi(\bfx)|_{\left[\frac18,\frac38\right]}$
going from $\exp(r \fh)$ to $\Gamma_{\tilde s\bfx}(-r)$ and
$\tilde \phi(\bfx)|_{\left[\frac58,\frac78\right]}$
going from $\Gamma_{\tilde s\bfx}(r)$ to $q_1 \exp(-r \fh)$.
For $t \in \left[0,\frac18\right]$, set $\tilde\phi(\bfx)(t) = \exp(8rt\fh)$;
for $t \in \left[\frac78,1\right]$, set $\tilde\phi(\bfx)(t) = q_1 \exp(8r(t-1)\fh)$;
for $t \in \left[\frac38,\frac58\right]$,
set $\tilde \phi(\bfx)(t) = 
\Gamma_{\tilde s\bfx}\left(8r\left(t-\frac12\right)\right)$.
Consider now $s \in (0,\tilde s]$ sufficiently small so that,
for all $\bfx \in \DD^d$ with $|\bfx| \le \frac{s}{\tilde s}$,
we have $\tilde\phi(\bfx) \in \hat \cA_\sigma$
(where $\hat \cA_\sigma$ is the open neighborhood 
of $\cL_n[\sigma]$ constructed in the proof of 
Theorem \ref{theo:stratification}).
Define $\phi: \DD^d \to \hat\cA_\sigma \subset \cL_n$ by
$\phi(\bfx) = \tilde\phi(\frac{s}{\tilde s} \bfx)$.

\begin{lemma}
\label{lemma:transsection}
Consider $\sigma \in S_{n+1}$, $\sigma \ne e$, $\dim(\sigma) = d$
and construct the map $\phi: \DD^d \to \cL_n$ as above.
This map is topologically transversal to $\cL_n[\sigma]$,
with a unique intersection at $\bfx = 0 \in \DD^d$.
\end{lemma}

\begin{proof}
Uniqueness of intersection follows from 
Theorem 2 
of \cite{Goulart-Saldanha0}.
Topological transversality follows from taking the composition
$\hat F \circ\phi$, where 
$\hat F: \hat \cA_\sigma \to \DD^d \subset \RR^d$
is constructed in the proof of Theorem \ref{theo:stratification} 
in Section \ref{sect:paths}.
The map $\hat F \circ \phi: \DD^d \to \DD^d$
is a positive multiple of the identity.
\end{proof}

Notice that the maps $\hat F: \hat \cA_\sigma \to \RR^d$ and
$\phi: \DD^d \to \hat\cA_\sigma \subset \cL_n$ consistently provide us
with a transversal orientation to $\cL_n[\sigma]$.

This completes the construction of a transversal section
to $\cL_n[\sigma_1]$ at the path $(z_1)\in\Pathiti((\sigma_1))$,
$z_1 = q \acute\sigma_1$, $q \in \Quat_{n+1}$.
By applying affine transformations in the interval
and projective transformations in the group $\Spin_{n+1}$,
this defines a map $\phi_1$ taking each $\bfx \in \DD^{d_1}$
($d_1 = \dim(\sigma_1)$) to a convex arc 
$\Gamma_{\bfx}: [t_1 - \epsilon, t_1 + \epsilon] \to \Spin_{n+1}$
with $\sing(\Gamma_{\bfx}) \ne \emptyset$,
$\sing(\Gamma_{\bfx}) \subset
(t_1 - \frac{\epsilon}{2}, t_1 + \frac{\epsilon}{2})$
and satisfying $\iti(\Gamma_{\bfx}) = (\sigma_1)$ if and only if $\bfx = 0$.
We may furthermore assume that 
$\Gamma_{\bfx}(t_1 \pm \epsilon) = z_1 \exp(\pm\epsilon\fh)$
for all $\bfx \in \DD^{d_1}$
and that
$\Gamma_{\bfx}(t) = z_1 \exp((t-t_1)\fh)$ for $\bfx = 0 \in \DD^{d_1}$.

More generally, for any $w = \sigma_1\cdots\sigma_\ell = 
(\sigma_1, \ldots, \sigma_\ell) \in \Word_n$,
for any path $(z_1, \ldots, z_\ell) \in \Pathiti(w)$
and for any set $\{t_1 < \cdots < t_\ell \} \subset (0,1)$, 
we show how to construct
a smooth map $\phi: \DD^d \to \cL_n$, $d = \dim(w)$,
transversal to $\cL_n[w]$ at $\phi(0) \in \cL_n[w]$,
$\pathiti(\phi(0)) = (z_1, \ldots, z_\ell)$,
$\sing(\phi(0)) = \{t_1 < \cdots < t_\ell \}$.
Make the convention $t_0 = 0$, $z_0 = 1$, $t_{\ell+1} = 1$
and $z_{\ell+1} = 
\acute\eta\hat\sigma_1 \cdots \hat\sigma_\ell\acute\eta$.
For each $i\in\llbracket\ell+1\rrbracket$, 
define $q_i \in \Quat_{n+1}$ such that 
$q_i \acute\eta = \adv(z_{i-1}) = \chop(z_i)$.
First, choose $\epsilon > 0$ such that, 
for all $i\in\llbracket\ell+1\rrbracket$, 
$t_{i-1} + \epsilon < t_i - \epsilon$,
$z_{i-1} \exp(\epsilon\fh) \in \Bru_{q_i \acute\eta}$ and
$z_{i} \exp(-\epsilon\fh) \in \Bru_{q_i \acute\eta}$.
Define $L_{i,-}, L_{i,+} \in \Lo_{n+1}^1$ by
$z_{i-1} \exp(\epsilon\fh) = q_i \acute\eta \bQ(L_{i,-})$ and 
$z_{i} \exp(-\epsilon\fh) = q_i \acute\eta \bQ(L_{i,+})$:
by taking $\epsilon$ sufficiently small we may assume that
$L_{i,-} \ll L_{i,+}$.
Choose fixed convex arcs
$\Gamma_{i-\frac12}:
[t_{i-1}+\epsilon, t_i-\epsilon] \to \Bru_{q_i\acute\eta}$ satisfying 
$\Gamma_{i-\frac12}(t_{i-1}+\epsilon) = z_{i-1} \exp(\epsilon\fh)$,
$\Gamma_{i-\frac12}(t_{i}-\epsilon) = z_{i} \exp(-\epsilon\fh)$.
In each interval $[t_i - \epsilon, t_i + \epsilon]$,
define as above a map $\phi_i$ associating to each $\bfx_i \in \DD^{d_i}$
a convex arc
$\Gamma_{i,\bfx_i}: [t_i - \epsilon, t_i + \epsilon] \to \Spin_{n+1}$
with $\Gamma_{i,\bfx_i}(t_i-\epsilon) = z_i \exp(-\epsilon\fh)$,
$\Gamma_{i,\bfx_i}(t_i+\epsilon) = z_i \exp(\epsilon\fh)$.
Set $\Gamma_0: [0,\epsilon] \to \Spin_{n+1}$,
$\Gamma_0(t) = \exp(t \fh)$ and
$\Gamma_{\ell+1}: [1-\epsilon,1] \to \Spin_{n+1}$,
$\Gamma_{\ell+1}(t) = z_{\ell+1} \exp((t-1)\fh)$.
Finally, for $\bfx = (\bfx_1, \ldots, \bfx_\ell)$,
concatenate these arcs to define $\phi(\bfx) = \Gamma_{\bfx} \in \cL_n$:
the map $\phi$ is the desired transversal section.

These explicit transversal sections allow us to explore the 
vincinity of a given stratum $\cL_n[w]$. 
In the examples below, we 
follow the above algorithm: 
we consider a family of convex arcs 
$\phi(\bfx)=\Gamma_{\bfx}\in\cL_n$, 
$\phi(0)=\Gamma_0\in\cL_n[w]$,  
obtained from a matrix with polynomial entries 
$M=M(\bfx,t)=\phi_L(\bfx;t)$.

Let $m_j(\bfx,t)$ be the southwest 
$j\times j$ minor of $M$,   
so that $m_j$ is an explicit 
element of $\RR[x_1,\ldots,x_d,t]$. 
By construction, $m_j(\bfx,t)$ is a positive multiple of 
the southwest $j\times j$ minor of $\Pi(\Gamma_\bfx(t))$. 
From Theorem 4 of \cite{Goulart-Saldanha0}, 
given $t_\ast\in\RR$ and $\sigma\in S_{n+1}$, 
we have $\Gamma_\bfx(t_\ast)\in\Bru_{\eta\sigma}$ 
if and only if, for each $j\in\nmesmo$, 
$t=t_\ast$ is a zero of $m_j(t)$ of multiplicity $\mult_j(\sigma)$. 
Here, 
$\mult_j(\sigma)=(1^\sigma-1)+\cdots+(j^\sigma-j)$, 
as in Equation \eqref{equation:mult}. 
The permutation $\sigma\in S_{n+1}$ can be readily recovered from 
the list of its multiplicities 
$\mult(\sigma)=(\mult_1(\sigma),\ldots,\mult_n(\sigma))\in\NN^n$:  
we have $j^\sigma=\mult_j(\sigma)-\mult_{j-1}(\sigma)+j$ 
(with the convention $\mult_0=\mult_{n+1}=0$). 

Adjacent strata $\cL_n[w']$ 
of codimension $\dim(w')=0$ 
are such that $w'=(a_{i_1},\ldots,a_{i_{\ell'}})$ is 
a string of Coxeter generators. 
In this case, $\Gamma_\bfx\in\cL_n[w']$ if and only if 
the real roots of $m_1(t),\ldots,m_n(t)$ are all simple 
and distinct.  
Multiple or common real roots correspond to 
more profound strata. 
More explicitly, 
if, for some value of $\bfx$ and some $t_i\in\sing(\Gamma_\bfx)$, 
there exists a subset $\{j_1,\ldots,j_k\}\subseteq\nmesmo$ such 
that $m_{j_1}(t_i)=\cdots=m_{j_k}(t_i)=0$, then the 
corresponding letter $\sigma_i$ in the itinerary 
$\iti(\Gamma_\bfx)=(\sigma_1,\ldots,\sigma_\ell)$ 
has reduced words involving all the generators 
$a_{j_1},\ldots,a_{j_k}$. 
The set of $\bfx=(x_1,\ldots,x_d)$ for which a given profound letter 
occurs is a subset of the zero locus of discriminants and resultants 
of the polynomials $m_j$. 
Let 
\begin{gather*}
\label{equation:discres}
d_{j}(\bfx) =\operatorname{discrim}_t(m_j(\bfx,t))\in\RR[\bfx], 
\quad j\in\nmesmo; \\
r_{i,j}(\bfx) =
\operatorname{res}_t(m_{i}(\bfx,t),m_{j}(\bfx,t))\in\RR[\bfx], 
\quad i, j\in\nmesmo, \quad i< j.
\end{gather*}
Thus, for instance, if a letter $[ab]$ occurs in the itinerary 
of $\Gamma_\bfx$, then $d_1(\bfx)=r_{12}(\bfx)=0$; 
we shall see other examples below. 

\begin{example}
\label{example:transversalsectionaba}

In our first example, $n = 2$, $w=(\sigma)$, 
$\sigma = [321] = aba$ 
(see matrix $M_0$ in Equation \eqref{equation:M0M1}), 
we have
\[ m_1 = \frac{t^2}{2} + x_2, \;
m_2 = \frac{t^2}{2} + x_1t - x_2, \;
d_1=-2x_2, \, d_2=x_1^2+2x_2, \;
r_{1,2}=-\frac{d_1d_2}4.
\]
Thus, $m_1(t)$ has two simple real roots 
$t = \pm\sqrt{-2x_2}$ if $x_2 < 0$ and
$m_2(t)$ has two simple real roots $t = -x_1\pm\sqrt{x_1^2+2x_2}$ if
$x_2 > -\frac{x_1^2}{2}$.
Thus, if $x_2 > 0$ the itinerary of $\Gamma_\bfx$ is $bb$ and
if $x_2 <  -\frac{x_1^2}{2}$ the itinerary is $aa$.
If $x_1 < 0$ (resp. $x_1 > 0$) and $ -\frac{x_1^2}{2} < x_2 < 0$ 
the itinerary is $abab$ 
(resp. $baba$). 
These itineraries correspond to adjacent strata 
of codimension zero; 
more profound strata occur for $x_2=0$ or $x_1^2+2x_2=0$. 
If $x_1<0$ (resp. $x_1 > 0$) and $x_2=0$, 
the itinerary is $[ab]b$ (resp. $b[ab]$). 
If $x_1<0$ (resp. $x_1 > 0$) and $x_1^2+2x_2=0$, 
the itinerary is $a[ba]$ (resp. $[ba]a$).
For instance, let $\bfx=(x_1,x_2)$ with 
$x_1 > 0$ and $x_1^2+2x_2=0$. 
Then, $m_1(t)$ has two simple roots at $t=\pm x_1$ 
and $d_2=0$, so that $m_2(t)$ has a double root at $t=-x_1$. 
Therefore, $\sing(\Gamma_\bfx)=\{-x_1,x_1\}$, 
$w'=\iti(\Gamma_\bfx)=(\sigma_1,\sigma_2)$ and, 
by Theorem 4 of \cite{Goulart-Saldanha0}, 
$\mult(\sigma_1)=(1,2)=\mult([ba])$ and
$\mult(\sigma_2)=(1,0)=\mult(a)$. 
Thus, $w'=[ba]a$. The other cases are similar. 
The reader should compare these results, 
summarized in Figure \ref{fig:transversalsectionaba}, 
with Example \ref{example:aba} and Figure \ref{fig:aba} 
in the introduction. 
Notice that, the more profound the stratum, 
the less generic are the curves therein. 
\end{example}

\begin{figure}[ht]
\centering
\begin{tikzpicture}[scale=0.8]
\begin{axis}[
xmin=-1, xmax=1, 
ymin=-1, ymax=1,
axis x line = center, 
axis y line = center,
xtick = \empty,
ytick = \empty,
xlabel = {$x_1$},
ylabel = {$x_2$},
legend style = {nodes=right},
legend pos = north east,
clip mode = individual, 
]
\addplot[red!80!black, line width=0.4mm, domain=-1:0.95] {0.01};
\node[left, line width=0.4mm, red!80!black] at (-1,0) {$x_2=0$};
\addplot[red!80!black, line width=0.4mm, samples=200, domain=-1:1] {-x*x/2};
\node[below, red!80!black] at (1,-0.5) {$x_1^2+2x_2=0$};
\node[above left, red] at (0,0) {$[aba]$};
\node[above, red!80!black] at (-0.6,0) {$[ab]b$};
\node[above, red!80!black] at (0.5,0) {$b[ab]$};
\node[below, red!80!black] at (-0.5,-0.125) {$a[ba]$};
\node[below, red!80!black] at (0.5,-0.125) {$[ba]a$};
\node[left,black] at (0,-0.5) {$aa$};
\node[left,black] at (0,0.5) {$bb$};
\node[left, black] at (-0.7,-0.125) {$abab$};
\node[right, black] at (0.7,-0.125) {$baba$};
\end{axis}
\end{tikzpicture}
\begin{tikzpicture}[scale=0.7]
\begin{axis}[
xmin=-1, xmax=1, 
ymin=-1/4, ymax=1,
axis x line = center, 
axis y line = center,
xtick = \empty,
ytick = \empty,
xlabel = {$t$},
ylabel = {$m$},
legend style = {nodes=right},
legend pos = north east,
clip mode = individual, 
]
\addplot[red, line width=0.4mm, samples=200, domain=-1:1] {x*x/2-1/18};
\addplot[blue, line width=0.4mm, domain=-1:1] {x*x/2+x/3+1/18};
\node[above right, red] at (1,4/9) {$m=m_1(t)$};
\node[above right , blue] at (1,8/9) {$m=m_2(t)$};
\draw[black] (-1/3,-0.03)--(-1/3,0.03) node[below] at (-1/3,-0.05) {$-1/3$};
\draw[black] (1/3,-0.03)--(1/3,0.03) node[below] at (1/3,-0.05) {$1/3$};
\node[above, purple] at (-1/3,0.03) {$[ba]$};
\node[above, red] at (1/3,0.03) {$a$};
\end{axis}
\end{tikzpicture}
\caption{
Left: transversal section $\phi:\DD^2\to\cL_2$ to $\cL_2[[aba]]$ 
(see Example \ref{example:transversalsectionaba}). 
Right: for $x_1=1/3$ and $x_2=-x_1^2/2=-1/18$ we have 
$\sing(\Gamma_\bfx)=\{\pm 1/3\}$ and 
$\iti(\Gamma_\bfx)=[ba]a$. 
Compare with 
Figure \ref{fig:aba} 
from Example \ref{example:aba}.}
\label{fig:transversalsectionaba}
\end{figure}



\begin{example}
\label{example:transversalsectionacb}

In our second example, $n = 3$, $w=(\sigma)$, 
$\sigma = [3142] = acb$ 
(see matrix $M_1$ in Equation \eqref{equation:M0M1}), 
we have
\begin{gather*}
m_1 = \frac{t^2}{2} + x_2, \;
m_2 = -t, \;
m_3 = \frac{t^2}{2} - x_1, \\  
d_1=-2r_{1,2}=-2x_2,\;
d_2=1, \;
d_3=-2r_{2,3}= 2x_1, \;
r_{1,3}=\frac{(x_1+x_2)^2}4.
\end{gather*}
Thus, $m_2(t)$ has a simple root at $t = 0$ for all values of $x_1,x_2$.
If $x_2 > 0$, $m_1(t)$ has no real roots;
if $x_2 < 0$, $m_1(t)$ has roots $t = \pm\sqrt{-2x_2}$.
Similarly, for $x_1 < 0$, $m_3(t)$ has no real roots
and for $x_1 > 0$, $m_3(t)$ has roots $t = \pm\sqrt{2x_1}$.
It is now easy to verify the itineraries of $\Gamma_\bfx$ 
in Figure \ref{fig:transversalsectionacb} using 
resultants and multiplicities.
For instance, for $\bfx=(x_1,x_2)$ with $x_1>0$ and $x_2=-x_1$, 
the simple roots of $m_1(t)$ and $m_3(t)$ coincide 
pairwise at $t=\pm\sqrt{2x_1}$. 
We therefore have 
$\sing(\Gamma_\bfx)=\{-\sqrt{2x_1},0,+\sqrt{2x_1}\}$ and 
$\iti(\Gamma_\bfx)=(\sigma_1,\sigma_2,\sigma_3)=[ac]b[ac]$, 
since, by Theorem 4 of \cite{Goulart-Saldanha0}, 
$\mult(\sigma_1)=\mult(\sigma_3)=(1,0,1)=\mult([ac])$ and 
$\mult(\sigma_2)=(0,1,0)=\mult(b)$. 
Notice that if $x_1<0$ and $x_2=-x_1$, 
then $\sing(\Gamma_\bfx)=\{0\}$ and $\iti(\Gamma_\bfx)=(b)$, 
even if we are in the zero locus of $r_{1,3}$. 
Also notice that the stratum $\cL_3[[ac]b[ac]]$ is 
as profound as $\cL_3[[acb]]$, since 
$\dim([ac]b[ac])=\dim([acb])=2$.


\end{example}

\begin{figure}[ht]
\centering
\begin{tikzpicture}[scale=0.7]
\begin{axis}[
xmin=-1, xmax=1, 
ymin=-1, ymax=1,
axis x line = center, 
axis y line = center,
xtick = \empty,
ytick = \empty,
xlabel = {$x_1$},
ylabel = {$x_2$},
legend style = {nodes=right},
legend pos = north east,
clip mode = individual,
]
\addplot[red, line width=0.4mm, domain=-1:0.95] {0.01};
\node[left, red] at (-1,0) {$x_2=0$};
\draw[blue, line width=0.4mm] (-0.008,-1)--(-0.008,0.95) node[below] at (0,-1) {$x_1=0$};
\addplot[magenta, line width=0.4mm, domain=0:1] {-x};
\addplot[dashed, black, domain=-1:0] {-x};
\node[below right, magenta] at (1,-1) {$x_2=-x_1$};
\node[above, black] at (-0.8,0.8) {$b$};
\node[above right, black] at (-0.5,0.5) {$b$};
\node[below left, black] at (-0.5,0.5) {$b$};
\node[above right, black] at (0.5,-0.5) {$cabac$};
\node[below left, black] at (0.5,-0.5) {$acbca$};
\node[black] at (0.5,0.5) {$cbc$};
\node[above, black] at (-0.7,-0.7) {$aba$};
\node[above right, magenta] at (1,-1) {$[ac]b[ac]$};
\node[left, blue] at (0,0.5) {$[cb]$};
\node[below left, blue] at (0,-0.5) {$a[cb]a$};
\node[below, red] at (-0.5,0) {$[ab]$};
\node[below, red] at (0.5,0) {$c[ab]c$};
\node[above right, magenta] at (0,0) {$[acb]$};
\end{axis}
\end{tikzpicture}
\begin{tikzpicture}[scale=0.7]
\begin{axis}[
xmin=-1, xmax=1, 
ymin=-1/4, ymax=1/2,
axis x line = center, 
axis y line = center,
xtick = \empty,
ytick = \empty,
xlabel = {$t$},
ylabel = {$m$},
legend style = {nodes=right},
legend pos = north east,
clip mode = individual, 
]
\addplot[red, line width=0.4mm, samples=200, domain=-1:1] {x*x/2-1/18};
\node[above, purple] at (1,4/9) {$m=m_1(t)=m_3(t)$};
\addplot[green, line width=0.4mm, samples=200, domain=-1:1] {-x};
\node[right, green] at (1/4,-1/4) {$m=m_2(t)$};
\addplot[blue, line width=0.4mm, domain=-1:1] {x*x/2-1/18+0.005};
\draw[black] (-1/3,-0.02)--(-1/3,0.02) node[below] at (-1/3,-0.02) {$-1/3$};
\draw[black] (1/3,-0.02)--(1/3,0.02) node[below] at (1/3,-0.02) {$1/3$};
\node[above, purple] at (-1/3,0.03) {$[ac]$};
\node[above right, green] at (0,0.03) {$b$};
\node[above, purple] at (1/3,0.03) {$[ac]$};
\end{axis}
\end{tikzpicture}
\caption{
Left: transversal section $\phi:\DD^2\to\cL_3$ to $\cL_3[[acb]]$ 
(see Example \ref{example:transversalsectionacb}). 
Right: for $x_1=-x_2=1/18$ we have 
$\sing(\Gamma_\bfx)=\{0,\pm1/3\}$ and 
$\iti(\Gamma_\bfx)=[ac]b[ac]$.}
\label{fig:transversalsectionacb}
\end{figure}

\begin{figure}[p]
\def\svgwidth{15cm}
\centerline{\includegraphics[width=15cm]{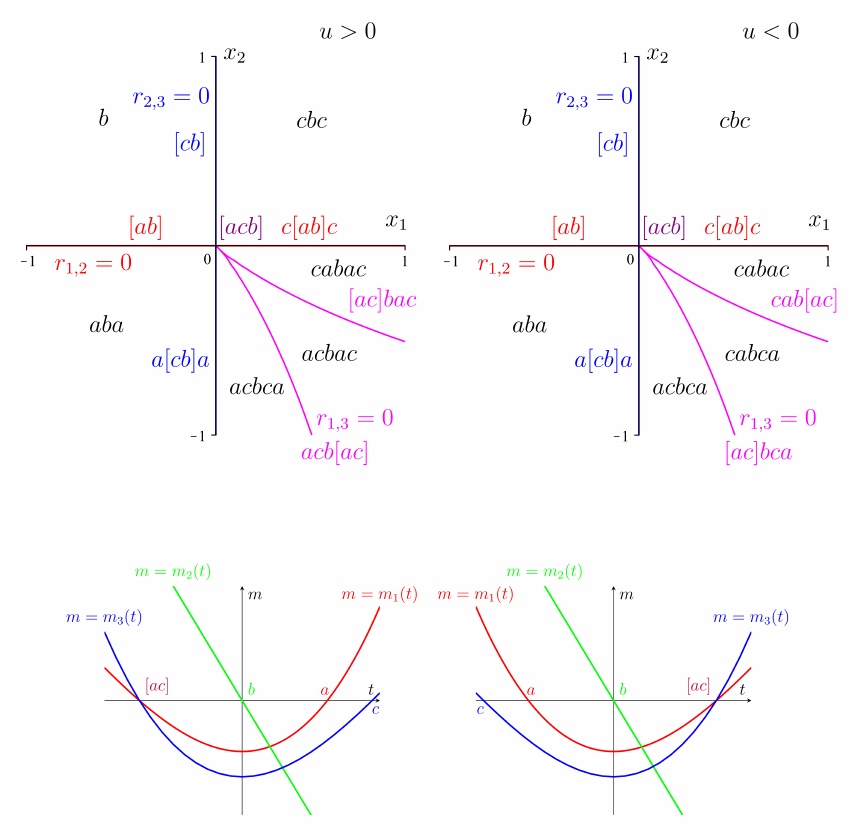}}
\caption{Above: transversal sections $\phi_u:\DD^2\to\cL_3$ to $\cL_3[[acb]]$ for $u=\pm2/5$ 
(see Example \ref{example:betaprime}).
Below: for $x_1=1/3$, and $x_2>-1/3$ 
such that $r_{1,3}(x_1,x_2,u)=0$, 
$\iti(\phi_u(x_1,x_2))=[ac]bac$ for $u=2/5$ while 
$\iti(\phi_u(x_1,x_2))=cab[ac]$ for $u=-2/5$. 
The slopes of the graphs of $m_1(t)$ and $m_3(t)$ in this example 
are suggestive of the fact that no perturbation of $\phi_u(0,0)$  
will produce the itinerary $cabca$ in the case $u>0$, 
and similarly, no perturbation of $\phi_u(0,0)$  
will produce the itinerary $acbac$ in the case $u<0$. 
We shall go back to this issue in Section \ref{sect:Hk}. 
The reader may want to compare this figure with 
Figure \ref{fig:betaprime} therein.}
\label{fig:planebetaprime}
\end{figure}

\begin{example}
\label{example:betaprime}
Consider the 1-parameter family of 
perturbations $\phi_u:\DD^2\to\cL_3$, $u\in(-\epsilon,\epsilon)$ 
(for some fixed $\epsilon\in(0,1)$), 
of the transversal section $\phi=\phi_0$ 
of Example \ref{example:transversalsectionacb} 
given by $\phi_u(\bfx)=\bQ_{\acute\eta\acute\sigma}\circ\Gamma_{\bfx;u}$, 
where $\Gamma_{\bfx;u}:[-1,1]\to \acute\eta\acute\sigma\Lo^1_4$ 
is the solution to the ODE   
\begin{gather*}
\Gamma'_{\bfx;u}(t)=\Gamma_{\bfx;u}(t)
\left(\beta_1(t)\fl_1+\beta_2(t)\fl_2+\beta_3(t)\fl_3\right), \\
\beta_1(t)=1+ut>0, \quad \beta_2(t)=1, \quad \beta_3(t)=1-ut>0,
\end{gather*}  
with the initial condition below:  
\[
\Gamma_{\bfx;u}(0)= \begin{pmatrix}
0 & -1 & 0 & 0 \\
0 & -x_1 & 0 &- 1 \\
-1 & 0 & 0 & 0 \\
x_2 & 0 & 1 & 0 \end{pmatrix}, \quad 
\Gamma_{\bfx;u}(t)=\begin{pmatrix}
g_{1,1}(t) & -1 & 0 & 0 \\
g_{2,1}(t) & g_{2,2}(t) & g_{2,3}(t) & -1 \\
-1 & 0 & 0 & 0 \\
g_{4,1}(t) & g_{4,2}(t) & 1 & 0 \end{pmatrix}.\] 
This can be explicitly integrated, yielding polynomial coefficients 
$g_{i,j}$. 
As before, consider the $j\times j$ southwest minors $m_j$
of $\Gamma_{\bfx;u}(t)$ as polynomials in 
the indeterminates $x_1,x_2,u,t$ 
and compute their discriminants and resultants:
\begin{gather*}
m_1=\frac{ut^3}3+\frac{t^2}2+x_2, \quad  
m_2=-t, \quad
m_3=-\frac{ut^3}3+\frac{t^2}2-x_1, \\
d_1=-\frac{x_2(6u^2x_2+1)}2, \quad
d_2=1, \quad 
d_3=-\frac{x_1(6u^2x_1-1)}2, \quad \\
r_{1,2}= x_2,\quad
r_{1,3}=\frac{u}{108}(4u^2(x_2-x_1)^3+9(x_1+x_2)^2), \quad  
r_{2,3}= x_1.
\end{gather*}
Figure \ref{fig:planebetaprime} 
shows the itineraries of curves 
in the section $\phi_u$ 
for fixed values of $u>0$ and $u<0$. 
The zero loci of the discriminants $d_j$ and resultants $r_{i,j}$  
contain the coordinate axes $x_1$ and $x_2$ as before, 
and a ordinary cusp $r_{1,3}$. 
The zero loci of $d_1$ and $d_3$ include lines 
far from the origin, which do not concern us. 
Notice the intersection of the zero loci of 
$d_1$, $d_3$ and $r_{i,j}$ at the origin.
The two diagrams differ combinatorially: 
for $u > 0$, the itinerary $acbac$ appears and $cabca$ does not;
for $u < 0$, it is the other way around; 
this will be discussed in Section \ref{sect:Hk}.
\end{example}

\begin{example}
\label{example:perturbedtransversalsectionacb}
Alternatively, consider another 1-parameter family of 
perturbations $\psi_u:\DD^2\to\cL_3$, $u\in(-\epsilon,\epsilon)$, 
of the transversal section $\phi=\psi_0$ 
of Example \ref{example:transversalsectionacb} 
given by taking 
\[ \tilde M_{1,u} = \begin{pmatrix}
0 & -1 & 0 & 0 \\
0 & -x_1 & 0 & -1 \\
-1 & -u & 0 & 0 \\
x_2 & x_3 & 1 & 0
\end{pmatrix}, 
\qquad
M_{1,u} = \begin{pmatrix}
-t & -1 & 0 & 0 \\
-\frac{t^3}{6} - x_1 t & -\frac{t^2}{2} - x_1 & -t & -1 \\
-ut - 1 & -u & 0 & 0 \\
\frac{t^2}{2} + x_2 & t & 1 & 0
\end{pmatrix} \]
instead of $\tilde M_1$ and $M_1$ of Equation \eqref{equation:M0M1}. 
As in Example \ref{example:betaprime}, 
the functions $m_j$ are explicit polynomials. 
Figure \ref{fig:perturbedtransversalsectionacb} 
shows the zero loci of the new resultants near the origin 
(there are complications far away 
which do not concern us). 
Notice the similarity between Figures \ref{fig:planebetaprime} and  \ref{fig:perturbedtransversalsectionacb}. 
\end{example}

\begin{figure}[t]
\def\svgwidth{15cm}
\centerline{\includegraphics[width =15cm]{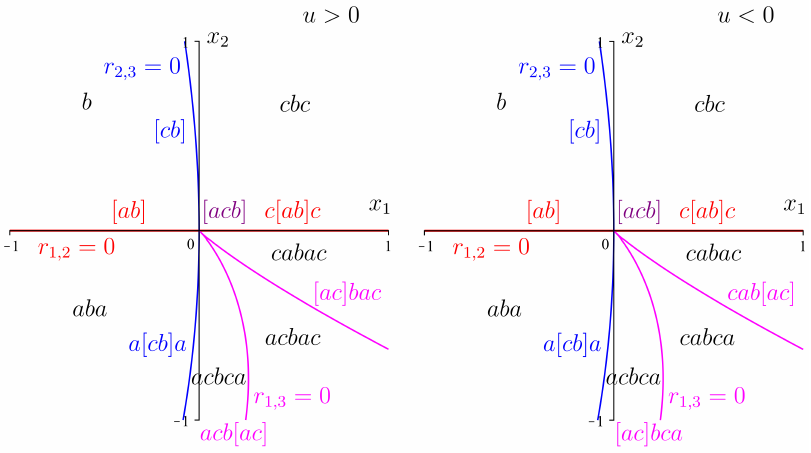}}
\caption{Transversal sections $\psi_u:\DD^2\to\cL_3$ to $\cL_3[[acb]]$ for $u=\pm2/5$ 
(see Example \ref{example:perturbedtransversalsectionacb}).}
\label{fig:perturbedtransversalsectionacb}
\end{figure}



\section{Proof of Theorem \ref{theo:poset}}
\label{sect:poset}

Let $\sigma \in S_{n+1}$, $\sigma \ne e$;
let $\sigma_1 = \eta\sigma$.
Let $z_1 =\acute\eta \acute\sigma= 
q \acute\sigma_1 \in \widetilde \B_{n+1}^{+}$, 
$q=\acute\eta\hat\sigma\acute\eta \in \Quat_{n+1}$.
Recall from Example 4.2 of \cite{Goulart-Saldanha0} 
that $\bL(\exp(\theta\fh)) = \exp(\tan(\theta) \fh_L)$
for $\theta \in (-\frac{\pi}{2}, \frac{\pi}{2})$. 
Let $\theta_0 \in (0,\frac{\pi}{2})$.
The smooth curve 
$\Gamma_{z_1,\fh}: [-\theta_0,+\theta_0] \to \Spin_{n+1}$,
$\Gamma_{z_1,\fh}(\theta) = z_1 \exp(\theta\fh)$,
is locally convex and has image contained in $\cU_{z_1}$.
It is therefore strictly convex and can be expressed in triangular coordinates
(see Subsections \ref{subsect:Bruhat} and \ref{subsect:convex}):
$\Gamma_{z_1,\fh}(\theta) = z_1\bQ(\Gamma_{L}(\theta))$, 
where $\Gamma_{L}(\theta): [-\theta_0,\theta_0] \to \Lo_{n+1}^{1}$,
$\Gamma_{L}(\theta) = \exp(\tan(\theta) \fh_L)$.
We have $\sing(\Gamma_{z_1,\fh})=\{0\}$ and 
$\iti(\Gamma_{z_1,\fh}) = (\sigma)$.


For $w \in \Word_n$, we 
define $w \dashv \sigma$
if there exists a convex curve $\Gamma_1$ 
with $\iti(\Gamma_1) = w$ in the set
\( \cL^{[H^1]}_{n,\conv}(z_1\exp(-\theta_0\fh); z_1\exp(\theta_0\fh)) \).
Lemma \ref{lemma:novanishingletter} implies that  
we have $(\,) \not\dashv \sigma$ 
(here, $(\,)\in\Word_n$ is the empty word). 
Our next result shows that the condition above does not depend
on the particular choice of $\theta_0 \in (0,\frac{\pi}{2})$.

\begin{lemma}
\label{lemma:tok0}
Consider $z_1 \in \widetilde \B_{n+1}^{+}$,
$\theta_1, \theta_2 \in (0,\frac{\pi}{2})$ 
and $w \in \Word_n$.
Then there exists 
$\Gamma_1 \in \cL_{n,\conv}^{[H^1]}(z_1\exp(-\theta_1\fh); z_1\exp(\theta_1\fh))$
with $\iti(\Gamma_1) = w$
if and only if there exists
$\Gamma_2 \in \cL_{n,\conv}^{[H^1]}(z_1\exp(-\theta_2\fh); z_1\exp(\theta_2\fh))$
with $\iti(\Gamma_2) = w$.
Furthermore, if there exists
$\Gamma_1 \in
\cL_{n,\conv}^{[H^1]}(z_1\exp(-\theta_1\fh); z_1\exp(\theta_1\fh))$
such that $\iti(\Gamma_1) = w$
then there exists a homotopy 
$H: [0,1] \to 
\cL_{n,\conv}^{[H^1]}(z_1\exp(-\theta_1\fh); z_1\exp(\theta_1\fh))$,
\[
H(0) = \Gamma_0=\Gamma_{z_1,\fh}, \quad H(1) = \Gamma_1, \quad
H(s)|_{[-\theta_1,-s\theta_1] \sqcup [s\theta_1,\theta_1]} =
\Gamma_0|_{[-\theta_1,-s\theta_1] \sqcup [s\theta_1,\theta_1]}
\]
such that $\iti(H(s)) = w$ for all $s \in (0,1]$.
\end{lemma}

\begin{proof}
We start with the first claim.
Assume without loss of generality that $\theta_1 < \theta_2$.
Given $\Gamma_1$ as above, $\Gamma_2$ can be constructed by attaching arcs:
set 
\[ \Gamma_2(\theta) = \begin{cases}
\Gamma_1(\theta), & \theta \in [-\theta_1,\theta_1], \\
z_1 \exp(\theta\fh), &
\theta \in [-\theta_2,-\theta_1] \sqcup [\theta_1,\theta_2].
\end{cases} \]
Conversely, given $\Gamma_2$ we apply a projective transformation
to obtain $\Gamma_1$.
More precisely, set
$\Gamma_{2;L}: [-\theta_2,\theta_2] \to \Lo_{n+1}^{1}$,
$\Gamma_{2;L}(\theta) = \bL(z_1^{-1} \Gamma_2(\theta))$.
Notice that a diagonal projective transformation
takes $\exp(\pm\tan(\theta_2) \fh_L)$
to $\exp(\pm\tan(\theta_1) \fh_L)$:
\[ \exp(\pm\tan(\theta_1) \fh_L) =
\diag\left(1,\lambda,\ldots,\lambda^n\right)
\exp(\pm\tan(\theta_2) \fh_L) 
\diag\left(1,\lambda^{-1},\ldots,\lambda^{-n}\right) \]
for $\lambda = \tan(\theta_1)/\tan(\theta_2)$;
apply this projective transformation and reparametrize the domain
to obtain $\Gamma_{1;L}$ and therefore 
$\Gamma_1 \in \cL_{n,\conv}^{[H^1]}(z_1\exp(-\theta_1\fh); z_1\exp(\theta_1\fh))$
with $\iti(\Gamma_1) = \iti(\Gamma_2)$.
For the second claim, given $\Gamma_1$,
apply a projective transformation as above to define $H(s)$
satisfying the conditions in the statement
(compare with the construction of the homotopy
in the proof of Lemma \ref{lemma:convex2}).
\end{proof}


\begin{lemma}
\label{lemma:tok1}
Consider $w \in \Word_n$, $\sigma \in S_{n+1}$, $\sigma \ne e$,
and $\tilde\Gamma \in \cL_n^{[H^1]}[(\sigma)]$.
There exists a sequence $(\Gamma_k)_{k \in \NN}$
of curves $\Gamma_k \in \cL_n^{[H^1]}[w]$
with 
$\lim_{k \to \infty} \Gamma_k = \tilde\Gamma$ 
in $\cL_n^{[H^1]}$ 
if and only if  
$w \dashv \sigma$.
\end{lemma}

Notice that one implication is already known for the special case 
$\tilde\Gamma=\Gamma_{z_1,\fh}$, $z_1=\acute\eta\acute\sigma$: 
for $w \dashv \sigma$, we constructed in Lemma \ref{lemma:tok0}
a path $H$ of curves of itinerary $w$ tending to $\Gamma_{z_1,\fh}$. 

\begin{proof}
Assume first that a sequence $(\Gamma_k)$ as above 
exists: we prove that $w \dashv \sigma$.
Let $z_1=\acute\eta\acute\sigma$.
Let 
$\Gamma(t)=\Gamma_{z_1;\fh}((t-\frac12)\pi)$ and 
$\Gamma_L(t)=\bL(z_1^{-1}\Gamma(t))$, $t\in(0,1)$; 
by Example 4.2 
of \cite{Goulart-Saldanha0}, 
$\Gamma_L(t)=\exp(-\cot(\pi t) \fh_L)$, $t\in(0,1)$.
By reparametrizing the domain 
and applying a projective transformation, 
we may assume 
$\sing(\tilde\Gamma) = \{\frac12\}$ and
$\tilde\Gamma(\frac12) = \Gamma(\frac12)=z_1$.
Take $\epsilon_0 > 0$ such that the condition
$|t-\frac12| \le \epsilon_0$ implies $\tilde\Gamma(t) \in \cU_{z_1}$.
For $t \in [\frac12 - \epsilon_0, \frac12 + \epsilon_0]$,
define $\tilde\Gamma_{L}(t) = \bL(z_1^{-1} \tilde\Gamma(t))$.
Notice that, by Lemma 5.3 
of \cite{Goulart-Saldanha0}, we have 
$\tilde\Gamma_{L}(\frac12 - \epsilon_0) \ll
I \ll \tilde\Gamma_{L}(\frac12 + \epsilon_0)$.
Take $\epsilon_1 \in (0,\frac{\epsilon_0}{2})$ such that
$\tilde\Gamma_{L}(\frac12 - \epsilon_0) \ll 
\Gamma_L(\frac12-\epsilon_1)$ and 
$\Gamma_L(\frac12+\epsilon_1)\ll \tilde\Gamma_{L}(\frac12 + \epsilon_0)$.
Set $L_{1;-} = \Gamma_L(\frac12-\epsilon_1)$ and
$L_{1;+} = \Gamma_L(\frac12+\epsilon_1)$.
Take $\epsilon_2 \in (0,\frac{\epsilon_1}{2})$ such that
$L_{1;-} \ll \tilde\Gamma_{L}(\frac12 - \epsilon_2) \ll I$ and
$I \ll \tilde\Gamma_{L}(\frac12 + \epsilon_2) \ll L_{1;+}$.

Take open neighborhoods $A_{0;-}$, $A_{2;-}$, $A_{2;+}$ and
$A_{0;+} \subset \Lo_{n+1}^{1}$ of
$\tilde\Gamma_{L}(\frac12 - \epsilon_0)$,
$\tilde\Gamma_{L}(\frac12 - \epsilon_2)$,
$\tilde\Gamma_{L}(\frac12 + \epsilon_2)$ and
$\tilde\Gamma_{L}(\frac12 + \epsilon_0)$, respectively, such that,
for all $L_{i;\pm} \in A_{i;\pm}$, $i \in \{0,2\}$,
we have
$L_{0;-} \ll L_{1;-} \ll L_{2;-} \ll I \ll
L_{2;+} \ll L_{1;+} \ll L_{0;+}$.
Let $B_{i;\pm} = z_1 \bQ[A_{i;\pm}] \subset \cU_{z_1}$,
$i \in \{0,2\}$;
notice that
$\tilde\Gamma(\frac12 \pm \epsilon_i) \in B_{i;\pm}$, $i \in \{0,2\}$.

For sufficiently large $k$, we have
$\Gamma_k(\frac12 \pm \epsilon_i) \in B_{i;\pm}$, $i \in \{0,2\}$.
By Theorem \ref{theo:Hausdorff},
for sufficiently large $k$, we also have
$\sing(\Gamma_k) \subset (\frac12 - \epsilon_2, \frac12 + \epsilon_2)$.
For such large $k$, define a locally convex curve $\tilde\Gamma_k$
which coincides with $\Gamma_k$ except in the intervals
$[\frac12 - \epsilon_0, \frac12 - \epsilon_2]$ and
$[\frac12 + \epsilon_2, \frac12 + \epsilon_0]$.
In these arcs, $\tilde\Gamma_k$ is defined so that
$\tilde\Gamma_k(\frac12 - \epsilon_1) =
z_1 \bQ(L_{1;-}) = \Gamma(\frac12-\epsilon_1)$ and 
$\tilde\Gamma_k(\frac12 + \epsilon_1) =
z_1 \bQ(L_{1;+}) = 
\Gamma(\frac12+\epsilon_1)$:
the above conditions guarantee that this is possible.
The restriction of any such curve $\tilde\Gamma_k$ 
to the interval $[\frac12-\epsilon_1,\frac12+\epsilon_1]$ yields,
by definition, $w \dashv \sigma$.

Now, assume $w \dashv \sigma$ and take 
$\tilde\Gamma\in\cL_n^{[H^1]}[(\sigma)]$ with 
$\sing(\tilde\Gamma)=\{\frac12\}$. 
As before, take $\Gamma(t)=\Gamma_{z_1;\fh}((t-\frac12)\pi)$. 
Define $\tilde\Gamma_L(t)=\bL(z_1^{-1}\tilde\Gamma(t))$ and 
$\Gamma_L(t)=\bL(z_1^{-1}\Gamma(t))$ for $t$ 
in some interval $[\frac12-\epsilon_0,\frac12+\epsilon_0]$.
Given $k\in\NN^{\ast}$, take $\epsilon_{1}\in(0,\frac{\epsilon_0}{2^k})$. 
For sufficiently small $\epsilon_2\in(0,\epsilon_1)$, 
we have 
$\tilde\Gamma_L(\frac12-\epsilon_1)\ll \Gamma_L(\frac12-\epsilon_2)$ 
and 
$\Gamma_L(\frac12+\epsilon_2)\ll \tilde\Gamma_L(\frac12+\epsilon_1)$. 
By Lemma 5.3 
of \cite{Goulart-Saldanha0}, there exist convex arcs 
$\Gamma_{k,L,-}:[\frac12-\epsilon_1,\frac12-\epsilon_2]\to\Lo_{n+1}^1$ 
and 
$\Gamma_{k,L,+}:[\frac12+\epsilon_2,\frac12+\epsilon_1]\to\Lo_{n+1}^1$ with 
$\Gamma_{k,L,-}(\frac12-\epsilon_1)=\tilde\Gamma_L(\frac12-\epsilon_1)$, 
$\Gamma_{k,L,-}(\frac12-\epsilon_2)=\Gamma_L(\frac12-\epsilon_2)$, 
$\Gamma_{k,L,+}(\frac12+\epsilon_2)=\Gamma_L(\frac12+\epsilon_2)$, 
$\Gamma_{k,L,+}(\frac12+\epsilon_1)=\tilde\Gamma_L(\frac12+\epsilon_1)$. 
Since $w \dashv \sigma$, there exists a convex arc 
$\Gamma_{k,L,0}:[\frac12-\epsilon_2,\frac12+\epsilon_2]\to\Lo_{n+1}^1$ 
with itinerary $w$ such that 
$\Gamma_{k,L,0}(\frac12\pm\epsilon_2)
=\Gamma_L(\frac12\pm\epsilon_2)$. 
For each $k\in\NN$, define 
\[\Gamma_k(t)=\begin{cases} 
\tilde\Gamma(t), & t\in[0,\frac12-\epsilon_1], \\ 
z_1\bQ(\Gamma_{k,L,-}(t)), & t\in[\frac12-\epsilon_1,\frac12-\epsilon_2], \\
z_1\bQ(\Gamma_{k,L,0}(t)), & t\in[\frac12-\epsilon_2,\frac12+\epsilon_2], \\
z_1\bQ(\Gamma_{k,L,+}(t)), & t\in[\frac12+\epsilon_2,\frac12+\epsilon_1], \\ 
\tilde\Gamma(t), & t\in[\frac12+\epsilon_1,1].
\end{cases}\] 
Of course we have $\lim_{k\to\infty}\Gamma_k=\tilde\Gamma$ 
in $\cL_n^{[H^1]}$, as desired. 
\end{proof}

\begin{rem}
\label{rem:localarcs}
In the statement of Lemma \ref{lemma:tok1} the curves 
$\tilde\Gamma$ and $\Gamma_k$ start at  
$\tilde\Gamma(0)=\Gamma_k(0)=1$ and 
end at $\tilde\Gamma(1)=\Gamma_k(1)=\acute\eta\hat\sigma\acute\eta$. 
The reader will notice, however, that only small convex arcs  
containing the singular sets are relevant to the proof. 
We may therefore apply Lemma \ref{lemma:tok1}
whenever both $\tilde\Gamma$ and $\Gamma_k$ are convex arcs
in the open subset $\cU_{z_1}\subset\Spin_{n+1}$. 
Such arcs have free endpoints in the appropriate connected components of
$\cU_{z_1}\cap\Bru_{\acute\eta}$ and 
$\cU_{z_1}\cap\Bru_{\acute\eta\hat\sigma}$.
This is an equivalent statement since we can always append 
initial and final arcs obtained by projective transformations. 
\end{rem}




\begin{lemma}
\label{lemma:preclosure}
For $\sigma \in S_{n+1} \smallsetminus \{e\}$ and $w \in \Word_n$,
the following conditions are equivalent:
\begin{enumerate}[label=(\roman*)]
\item{$w\dashv \sigma$;}
\item{$w \tok (\sigma)$;}
\item{$\cL_n^{[H^1]}[(\sigma)]\cap
\overline{\cL_n^{[H^1]}[w]}\neq\emptyset$;}
\item{$\cL_n^{[H^1]}[(\sigma)]\subseteq\overline{\cL_n^{[H^1]}[w]}$;}
\item{given $\tilde\Gamma \in \cL_n^{[H^1]}[(\sigma)]$, $\epsilon > 0$,
$\sing(\tilde\Gamma) = \{t_\ast\}$
and an open neighborhood $U \subset \cL_n^{[H^1]}$ 
of $\tilde\Gamma$ there exists
$\Gamma \in U \cap \cL_n^{[H^1]}[w]$
with $\Gamma$ and $\tilde\Gamma$
coinciding outside $(t_\ast-\epsilon,t_\ast+\epsilon)$.}
\end{enumerate}
\end{lemma}

\begin{proof}
Conditions (ii) and (iv) are equivalent by definition.
Condition (iv) clearly  implies (iii);
Lemma \ref{lemma:tok1} shows that (iii) implies (i) 
and that (i) implies (iv). 
The proof of Lemma \ref{lemma:tok1} shows that (i) implies (v). 
Finally, (v) clearly implies (iii).
\end{proof}

\begin{rem}
\label{rem:nonempty2}
The known fact that
convex curves form a connected component of $\cL_n^{[H^1]}$
(as in Lemma \ref{lemma:convex2}) 
gives us a second proof of the fact that  
$(\,)\npreceq (\sigma)$. 
\end{rem}

Lemma \ref{lemma:preclosure} is a 
local version of Theorem \ref{theo:poset}, 
which we are now ready to prove. 

\begin{proof}[Proof of Theorem \ref{theo:poset}]
Condition (iii) implies (i); Condition (i) implies (ii). 
We now show that Condition (ii) implies (iv). 
Indeed, take $\tilde\Gamma\in\cL_n^{[H^1]}[w_1]\cap
\overline{\cL_n^{[H^1]}[w_0]}$ and a sequence 
$(\Gamma_k)_{k\in\NN}$ of curves in $\cL_n^{[H^1]}[w_0]$ 
tending to $\tilde\Gamma$. 
Set $\sing(\tilde\Gamma)=\{t_1<\cdots<t_\ell\}$ and $\epsilon>0$ 
such that $3\epsilon<\min\{t_{i+1}-t_i\,;\,i\in\llbracket \ell\rrbracket\cup\{0\}\}$ (where $t_0=0$ and $t_{\ell+1}=1$, as usual). 
Notice that the intervals $J_i=[t_i-\epsilon,t_i+\epsilon]$ are disjoint. 
By Theorem \ref{theo:Hausdorff}, for sufficiently large $k$, we have 
$\sing(\Gamma_k)\subset\sqcup_i J_i$. 
The restrictions $\Gamma_k|_{J_i}$ tend to 
$\tilde\Gamma|_{J_i}$ and therefore, for large $k$,  
$\iti(\Gamma_k|_{J_i})=\tilde w_i\tok (\sigma_i)$, 
by Lemmas \ref{lemma:tok1} and \ref{lemma:preclosure} 
(see also Remark \ref{rem:localarcs}).  
We have $\tilde w_i\neq (\,)$ by Remark \ref{rem:nonempty2}. 

Now we prove that Condition (iv) implies (iii). 
The idea is to slightly perturb $\tilde\Gamma$ 
about each singular point $t_i$ while leaving the curve 
unchanged outside the supports of these perturbations. 
Implication (ii) to (v) in Lemma \ref{lemma:preclosure} ensures  
that the resulting curve $\Gamma$ can be made to have the 
desired itinerary $w_0$.
\end{proof}

\section{Proof of Theorem \ref{theo:multHk} and the example $[acb]$}
\label{sect:Hk}

The proof of Theorem \ref{theo:multHk} could have been
given immediately after the proof of Theorem \ref{theo:Hausdorff}, 
but we prefer to discuss in this section questions related to 
the $H^r$ metric
for large $r$. 

\begin{proof}[Proof of Theorem \ref{theo:multHk}]
We fix $w_1\in\Word_n$ and $\Gamma_1\in\cL_n^{[H^r]}[w_1]$, 
where we take  
\[r>r_{\bullet}(n)=\left\lfloor \left(\frac{n+1}2\right)^2\right\rfloor
=\max\{\mult_j(\sigma);\,\sigma\in S_{n+1}, j\in\nmesmo\}.\] 
We construct an open neighbourhood $U$ of $\Gamma_1$ 
in $\cL_n^{[H^r]}$ such that $\Gamma\in U$, 
$\iti(\Gamma)=w_0$ implies $\mult(w_0)\leq\mult(w_1)$. 
Write $\sing(\Gamma_1)=\{t_1<\cdots<t_\ell\}$. 
As before, for each $j\in\nmesmo$, consider the function 
$m_{\Gamma_1;j}:[0,1]\to\RR$ 
given by the southwest $j\times j$ minor of the matrix 
$\Pi(\Gamma_1(t))$. 
Set $\mu_{i,j}=\mult_j(\Gamma_1;t_i)$, 
the multiplicity of $t=t_i$ as a zero of $m_{\Gamma_1;j}(t)$, 
so that $\sum_{i}\mu_{i,j}=\mult_j(w_1)$, 
by Theorem 4 of \cite{Goulart-Saldanha0}. 
Notice that  
$m_{\Gamma_1;j}^{(\mu_{i,j})}(t_i)\neq 0$ 
for all $i,j$.  
The value of $r_\bullet(n)$ above was chosen so that 
these derivatives are all known to be continuous. 
 Take $\epsilon>0$ and disjoint open intervals $J_i\ni t_i$ 
such that, for all 
$i\in\llbracket \ell \rrbracket$ and $j\in\nmesmo$, 
we have $|m_{\Gamma_1;j}^{(\mu_{i,j})}(t_i)|>\epsilon$ 
and, for all $t\in J_i$,  
$|m_{\Gamma_1;j}^{(\mu_{i,j})}(t)|>\epsilon/2$. 
Using Theorem \ref{theo:Hausdorff}, we take an open set  
$U\in\cL_n^{[H^r]}$ containing $\Gamma_1$ such that 
$\Gamma\in U$ implies $\sing(\Gamma)\subset\cup_i J_i$ and 
$t\in J_i$ implies $|m_{\Gamma;j}^{(\mu_{i,j})}(t)|>\epsilon/4$. 
The fact that the derivative of order $\mu_{i,j}$ of $m_{\Gamma;j}$ 
has constant sign in $J_i$ implies that the number of zeroes of 
$m_{\Gamma;j}$ in $J_i$ (counted with multiplicity) is at most 
$\mu_{i,j}$. 
Now, Theorem 4 of \cite{Goulart-Saldanha0} 
implies the desired result.
\end{proof}

Given $\sigma\in S_{n+1}\smallsetminus\{e\}$, $w\in\Word_n$, 
we write $w\dashv\sigma\,[H^r]$ if and only if 
there exists a strictly convex curve 
$\Gamma_0:[-1,1]\to\Spin_{n+1}$ of class $H^r$ 
satisfying :
\begin{enumerate}
\item\label{item:dashvHk0}{
$\Gamma_0(-1), \Gamma_0(1) \in \Bru_\eta$,
$\sing(\Gamma_0)=\{0\}$ and $\iti(\Gamma_0)=(\sigma)$;}
\item\label{item:dashvHkepsilon}{for each $\epsilon>0$, 
there is a locally convex curve 
$\Gamma_\epsilon:[-1,1]\to\Spin_{n+1}$ of class $H^r$ 
satisfying $\iti(\Gamma_\epsilon)=w$ and 
$d^{[H^r]}(\Gamma_0,\Gamma_\epsilon)<\epsilon$.}
\end{enumerate}

Of course,
for all $\sigma\in S_{n+1}\smallsetminus\{e\}$, 
there is a strictly convex curve $\Gamma_0$
satisfying Condition~\ref{item:dashvHk0}:
just take $\Gamma_0(t)=\acute\sigma\exp\left(\frac\pi2 t\fh\right)$. 
We stress though that, in the definition above, 
the curve $\Gamma_0$ need not have this special form. 
Under both Conditions \ref{item:dashvHk0} and  \ref{item:dashvHkepsilon} above,
Theorem \ref{theo:Hausdorff} implies that,
for all $\tau>0$, there is $\epsilon_\tau>0$ such that 
$\epsilon<\epsilon_\tau$ implies 
$\sing(\Gamma_\epsilon)\subset(-\tau,\tau)$. 
Also, Lemma \ref{lemma:convex} implies that 
there is $\epsilon_{\conv}>0$ such that 
$\epsilon<\epsilon_{\conv}$ implies $\Gamma_\epsilon$ 
being strictly convex.

\begin{lemma}
\label{lemma:subwordsHk}
Given $w_0,w_1\in\Word_n$, we have 
$\cL_n^{[H^r]}[w_1]\cap\overline{\cL_n^{[H^r]}[w_0]}
\neq\emptyset$ 
if and only if there are 
$\sigma_{1},\ldots,\sigma_{\ell}\in S_{n+1}\smallsetminus\{e\}$ 
and $w_{0,1},\ldots,w_{0,\ell}\in\Word_n$ such that 
$w_0=w_{0,1}\cdots w_{0,\ell}$, 
$w_1=(\sigma_{1},\ldots,\sigma_{\ell})$ and, 
for all $i\in\llbracket\ell\rrbracket$, 
$w_{0,i}\dashv\sigma_{i}\,[H^r]$. 
\end{lemma}

In the proof of Lemma \ref{lemma:subwordsHk} we are going 
to use the following alternate characterization of the relation 
$\dashv\,[H^r]$. 

\begin{lemma}
\label{lemma:dashvHkprime}
Given $\sigma\in S_{n+1}\smallsetminus\{e\}$ and $w\in\Word_n$, 
we have $w\dashv\sigma\,[H^r]$ if and only if 
there exists a strictly convex curve 
$\Gamma_0:[-1,1]\to\Spin_{n+1}$ of class $H^r$ 
for which the following conditions hold: 
\begin{enumerate}[label={\arabic*'}.]
\item\label{item:dashvHk0prime}{
$\Gamma_0(-1), \Gamma_0(1) \in \Bru_\eta$,
$\sing(\Gamma_0)=\{0\}$ and $\iti(\Gamma_0)=(\sigma)$;}
\item\label{item:dashvHkepsilonprime}{given $\epsilon>0$,  
$\tau\in(0,1)$, 
there is a locally convex curve 
$\tilde\Gamma:[-1,1]\to\Spin_{n+1}$ of class $H^r$ 
satisfying $\iti(\tilde\Gamma)=w$, 
$d^{[H^r]}(\Gamma_0,\tilde\Gamma)<\epsilon$ and 
$\tilde\Gamma|_{[-1,-\tau]\cup[\tau,1]}
=\Gamma_0|_{[-1,-\tau]\cup[\tau,1]}$.}
\end{enumerate}
\end{lemma}  

\begin{proof}
The only nontrivial claim is: if $\Gamma_0$ satisfies 
Conditions \ref{item:dashvHk0} and \ref{item:dashvHkepsilon}, 
then it also satisfies Condition \ref{item:dashvHkepsilonprime} 
Given $\epsilon>0$ and $\tau\in(0,1)$, 
take $\Gamma_{\epsilon'}$ satisfying 
Condition~\ref{item:dashvHkepsilon} 
for $\epsilon'\in(0,\epsilon)$ such that 
$\sing(\Gamma_{\epsilon'})\subset\left(-\frac\tau2,\frac\tau2\right)$. 
The idea is now to obtain locally convex arcs 
$\Gamma_-:\left[-\tau,-\frac\tau2\right]\to\Spin_{n+1}$ and 
$\Gamma_+:\left[\frac\tau2,\tau\right]\to\Spin_{n+1}$ 
such that $\tilde\Gamma:[-1,1]\to\Spin_{n+1}$, 
given by 
\[
\tilde\Gamma(t)=
\begin{cases}
\Gamma_0(t), \quad t\in[-1,-\tau]\cup[\tau,1], \\
\Gamma_-(t), \quad t\in\left[-\tau,-\frac\tau2\right], \\
\Gamma_+(t), \quad t\in\left[\frac\tau2,\tau\right], \\
\Gamma_{\epsilon'}(t), \quad t\in\left[-\frac\tau2,\frac\tau2\right], 
\end{cases}
\]
has the desired properties. 
One may do so by producing standard convex arcs 
between $\Gamma_0(\pm\tau)$ and $\Gamma_{\epsilon'}(\pm\frac{\tau}2)$ 
by means of projective transformations and then applying a smoothening procedure. 
We omit the details (but see \cite{gsie}). 
\end{proof}

\begin{proof}[Proof of Lemma \ref{lemma:subwordsHk}]
One direction is easy: 
write $w_1=(\sigma_{1},\ldots,\sigma_{\ell})$, 
$\sigma_i\in S_{n+1}\smallsetminus\{e\}$.  
Given $\Gamma_0\in\cL_n^{[H^r]}[w_1]$, 
let $\sing(\Gamma_0)=\{t_1<\cdots<t_\ell\}$ and 
take $\tau>0$ such that 
$\Gamma_{0,i}=\Gamma_0|_{(t_i-\tau,t_i+\tau)}$ 
is strictly convex and $\iti(\Gamma_{0,i})=(\sigma_i)$ for all $i$. 
Now, assume there is a sequence 
$(\Gamma_k)$, $k\in\NN^\ast$, 
in $\cL_n^{[H^r]}[w_0]$ such that 
$\lim_{k\to\infty}d^{[H^r]}(\Gamma_k,\Gamma_0)=0$ 
and consider the restrictions 
$\Gamma_{k,i}=\Gamma_k|_{(t_i-\tau,t_i+\tau)}$. 
By Theorem \ref{theo:Hausdorff}, for 
sufficiently large $k$, we can assume 
$\sing(\Gamma_{k,i})\subset(t_i-\tau,t_i+\tau)$ for all $i$. 
For each $k$, write $\iti(\Gamma_{k,i})=w_{0,i}^k$ so that 
$w_0=w_{0,1}^k\cdots w_{0,\ell}^k$. 
Since there are finitely many decompositions of $w_0$ in subwords, 
we take a subsequence and assume $(\Gamma_k)$ is such that 
there are fixed subwords 
$w_{0,1},\ldots,w_{0,\ell}\in\Word_n$ with 
$\iti(\Gamma_{k,i})=w_{0,i}$ for all $k$. 
It follows that $w_{0,i}\dashv \sigma_i\,[H^r]$ for all $i$. 
Now for the reciprocal, let 
$\sigma_1,\ldots,\sigma_\ell\in S_{n+1}\smallsetminus\{e\}$ and  
$w_{0,1},\ldots,w_{0,\ell}\in\Word_n$ be such that 
$w_{0,i}\dashv\sigma_i\,[H^r]$ for all $i\in\llbracket\ell\rrbracket$. 
Also, fix once and for all a single $\tau\in(0,1)$ and, 
for each $i$ and each $\epsilon>0$, 
let $\Gamma_{0,i}$, 
$\iti(\Gamma_{0,i})=(\sigma_i)$, and 
$\tilde\Gamma_{\epsilon,i}$, 
$\iti(\tilde\Gamma_{\epsilon,i})=w_{0,i}$, 
be as in Conditions \ref{item:dashvHk0}' and 
\ref{item:dashvHkepsilon}' of 
Lemma \ref{lemma:dashvHkprime} above. 
We shall produce from these ingredients a 
locally convex curve $\tilde\Gamma_0\in\cL_n^{[H^r]}[w_1]$, 
$w_1=(\sigma_1,\ldots,\sigma_\ell)$, 
and a sequence of locally convex curves 
$\tilde\Gamma_k\in\cL_n^{[H^r]}[w_0]$, $k\in\NN^\ast$, 
$w_0=w_{0,1}\cdots w_{0,\ell}$, such that 
$\lim_{k\to\infty}d^{[H^r]}(\tilde\Gamma_k,\tilde\Gamma_0)=0$. 
We begin by setting $\tilde\Gamma_{k,i}=\tilde\Gamma_{\frac1k,i}$ 
(i.e., take $\epsilon=\frac1k$) for all $i\in\llbracket\ell\rrbracket$, 
and all $k\in\NN^\ast$. 
Let $q_0=1$, $q_1=\hat\sigma_1$, 
$q_2=\hat\sigma_1\hat\sigma_2$, 
\ldots, $q_\ell=\hat\sigma_1\cdots\hat\sigma_\ell=\hat w_1
\in\Quat_{n+1}$. 
We have, 
for all $i$ and all $k$, 
$\tilde\Gamma_{k,i}(-1)=\Gamma_{0,i}(-1)
\in\Bru_{\acute\eta q_{i-1}}=\cU_{\acute\eta q_{i-1}}$ and  
$\tilde\Gamma_{k,i}(1)=\Gamma_{0,i}(1)
\in\Bru_{\acute\eta q_i}=\cU_{\acute\eta q_i}$. 
Choose recursively a sequence of matrices 
$U_1,\ldots,U_\ell\in\Up^1_{n+1}$ such that 
$1\ll
\Gamma^{U_1}_{0,1}(-1)$,  
$\Gamma^{U_1}_{0,1}(1)q^{-1}_1
\ll 
\Gamma^{U_2}_{0,2}(-1)q^{-1}_1$,  
\ldots, 
$\Gamma^{U_{\ell-1}}_{0,{\ell-1}}(1)q^{-1}_{\ell-1}
\ll 
\Gamma^{U_\ell}_{0,\ell}(-1)q^{-1}_{\ell-1}$, 
$\Gamma^{U_\ell}_{0,\ell}(1)q^{-1}_{\ell}
\ll
\acute\eta\hat w_1\acute\eta q^{-1}_{\ell}$   
(where $\ll$ is the relation of accessibility in 
$\Bru_{\acute\eta}$ defined in Section \ref{sect:acc}; 
see also Lemma 5.3 of \cite{Goulart-Saldanha0}). 
We now fix, for all 
$i\in\llbracket\ell-1\rrbracket$ and all $k$, 
smooth strictly convex arcs 
$\Gamma_{0,i+\frac12}:[-1,1]\to\Bru_{\acute\eta q_i}$ 
such that $\Gamma_{0,i+\frac12}(-1)=
\Gamma_{0,i}^{U_i}(1)$ and 
$\Gamma_{0,i+\frac12}(1)
=\Gamma_{0,i+1}^{U_{i+1}}(-1)$. 
Of course, there are smooth strictly convex arcs 
$\Gamma_{0,\frac12}:[-1,1]\to\Bru_{\acute\eta}$ and 
$\Gamma_{0,\ell+\frac12}:[-1,1]\to\Bru_{\acute\eta\hat w_1}$ 
such that $\Gamma_{0,\frac12}(-1)=1$, 
$\Gamma_{0,\frac12}(1)=\Gamma_{0,1}^{U_1}(-1)$, 
$\Gamma_{0,\ell+\frac12}(-1)=\Gamma_{0,\ell}^{U_\ell}(1)$ and 
$\Gamma_{0,\ell+\frac12}(1)=\acute\eta\hat w_1\acute\eta$.
Now, consider the concatenations   
$\Gamma_0=
\Gamma_{0,\frac12}\ast
\Gamma_{0,1}^{U_1}\ast
\Gamma_{0,\frac32}\ast
\cdots\ast
\Gamma_{0,\ell}^{U_\ell}\ast
\Gamma_{0,\ell+\frac12}$ and  
$\Gamma_k=
\Gamma_{0,\frac12}\ast
\tilde\Gamma_{k,1}^{U_1}\ast
\Gamma_{0,\frac32}\ast
\cdots\ast
\tilde\Gamma_{k,\ell}^{U_\ell}\ast
\Gamma_{0,\ell+\frac12}$,  
for each $k\in\NN^\ast$ 
(the same piecewise affine reparameterization onto $[0,1]$ is used
in all these concatenations). 
These locally convex curves are of class $H^r$ 
except at the finitely many welding points 
$0<\tau_1<\tau_2<\cdots<\tau_{2\ell}<1$, 
all of them far away from the singular sets 
$\sing(\Gamma_0), \sing(\tilde\Gamma_s)\subset(0,1)$. 
Also, notice that, by Condition \ref{item:dashvHkepsilon}' of 
Lemma \ref{lemma:dashvHkprime}, 
satisfied by the sequences $\tilde\Gamma_{k,i}$ 
(recall we have fixed $\tau>0$ right from the start), 
there is $\delta>0$ such that, for all $j\in\llbracket\ell\rrbracket$,  
we have 
$\Gamma_0|_{[\tau_j-\delta,\tau_j+\delta]}
=\tilde\Gamma_k|_{[\tau_j-\delta,\tau_j+\delta]}$ 
for all $k\in\NN^\ast$. 
Apply a standard smoothening procedure
to each one of 
these $\ell$ arcs (of class $H^r$ except at $t=\tau_j$),   
obtaining the corresponding locally convex arcs  
$\tilde\Gamma_{0,j}:[\tau_j-\delta,\tau_j+\delta]\to\Spin_{n+1}$
of class $H^r$ (coinciding with the original ones 
on an initial and on a final segment). 
It is now easily seen that the maps 
$\tilde\Gamma_k:[0,1]\to\Spin_{n+1}$, 
$k\in\NN$, 
\[
\tilde\Gamma_k(t)=
\begin{cases}
\tilde\Gamma_{0,j}(t), \quad t\in[\tau_j-\delta,\tau_j+\delta], \\
\Gamma_k(t), \quad t\in[0,1]
\smallsetminus\left(\sqcup_j[\tau_j-\delta,\tau_j+\delta]\right)
\end{cases}
\]
satisfy all the desired properties. 
\end{proof}

\begin{coro}
[of Theorem \ref{theo:multHk} and Lemma \ref{lemma:subwordsHk}]
\label{coro:subwordsHk}
Given $w_0,w_1\in\Word_n$, 
if there are 
$\sigma_{1},\ldots,\sigma_{\ell}\in S_{n+1}\smallsetminus\{e\}$ 
and $w_{0,1},\ldots,w_{0,\ell}\in\Word_n$ such that 
$w_0=w_{0,1}\cdots w_{0,\ell}$, 
$w_1=(\sigma_{1},\ldots,\sigma_{\ell})$ and  
for all $i\in\llbracket\ell\rrbracket$, 
$w_{0,i}\dashv\sigma_{i}\,[H^r]$, 
then, $\mult(w_0)\leq\mult(w_1)$. 
\end{coro}



The following statement is already known to be true for $n=2$. 

\begin{conj}
\label{conj:dashv}
For $\dim(\sigma)=\inv(\sigma)-1<n$ and all $w\in\Word_n$, 
we have $w\dashv\sigma\,[H^1]$ (i.e., $w\tok(\sigma)$) if and only if 
$w\dashv\sigma\,[H^r]$ for all $r>2$.
\end{conj}

We now discuss in greater detail the example $[acb]$, 
with emphasis on the $H^r$ metric, $r\geq3$.   
This example has already been 
mentioned in Equation \eqref{equation:acbHk} 
in the Introduction and in Examples 
\ref{example:transversalsectionacb} and 
\ref{example:perturbedtransversalsectionacb}. 
We already know from Example 
\ref{example:perturbedtransversalsectionacb} 
(via Lemma \ref{lemma:subwordsHk}) that 
$\cL_3^{[H^r]}[[acb]]\cap\overline{\cL_3^{[H^r]}[cabca]}
\neq\emptyset$, for all $r$.
By Theorem \ref{theo:poset} 
(or Lemma \ref{lemma:preclosure}), 
we have the inclusion  
$\cL_3^{[H^1]}[[acb]]\subset\overline{\cL_3^{[H^1]}[cabca]}$. 

\begin{prop}
\label{prop:acbHk}
Take $n=3$ and $r \ge 3$. 
There exists a continuous function
$u: \cL_3^{[H^r]}[[acb]] \to \RR$
with the following properties:
\begin{enumerate}
\item{The set $u^{-1}[\{0\}] \subset \cL_3^{[H^r]}[[acb]]$
is a non empty closed subset and
a topological submanifold of codimension $1$.}
\item{The function $u$ is topologically transversal to
$u^{-1}[\{0\}] \subset \cL_3^{[H^r]}[[acb]]$.}
\item{If $\Gamma_0 = \phi_u(0)$ for $\phi_u$
as in Example \ref{example:betaprime} then
$u(\Gamma_0) = u$.}
\item{If $\Gamma_0 \in \cL_3^{[H^r]}[[acb]]$
and $u(\Gamma_0) > 0$ then
the left hand side of Figure \ref{fig:planebetaprime} 
provides a local topological normal form for the itinerary 
of locally convex curves near the curve $\Gamma$. }
\item{Conversely, if $u(\Gamma_0) < 0$ then
the right hand side of Figure \ref{fig:planebetaprime} 
provides a local topological normal form.}
\end{enumerate}
\end{prop}

The concept of a local topological normal form needs clarification.
Let $\mathbf{H}$ be an infinite dimensional separable Hilbert space.
We claim, for instance, that if $u(\Gamma_0) > 0$
then there exists a neighborhood $W \subset \cL_3^{[H^r]}$
of $\Gamma_0$
and a homeomorphism $\tilde\psi: \RR^2 \times \mathbf{H} \to W$
such that the itinerary of $\tilde\psi(x_1,x_2,\ast)$
is given by $(x_1,x_2)$ in the left hand side of Figure \ref{fig:planebetaprime}.

Notice, in particular, that if a curve $\Gamma\in\cL_3^{[H^r]}$, 
$r\geq3$, has a letter $[acb]$ in its itinerary and  
$u>0$, then there exist perturbations 
$\tilde \Gamma$ of $\Gamma$ where the letter $[acb]$ 
splits into the string $acbac$ 
but there are no perturbations $\tilde\Gamma$ 
of $\Gamma$ where $[acb]$ becomes $cabca$. 
Similarly, for $u<0$, the letter $[acb]$ can become 
$cabca$, but not the itinerary $acbac$. 

\begin{rem}
\label{rem:betaprime}
Figure \ref{fig:betaprime} shows the itineraries
near a specific curve 
$\Gamma_0 \in \cL_3^{[H^r]}[[acb]]$
with $u(\Gamma_0) = 0$.
In fact, Figure \ref{fig:betaprime} provides
a local topological normal form
near all such curves.
The proof of this last fact shall not be given,
but is similar (but more laborious)
than that of Proposition \ref{prop:acbHk}.
\end{rem}

\begin{figure}[h]
\def\svgwidth{15cm}
\centerline{\includegraphics[width=15cm]{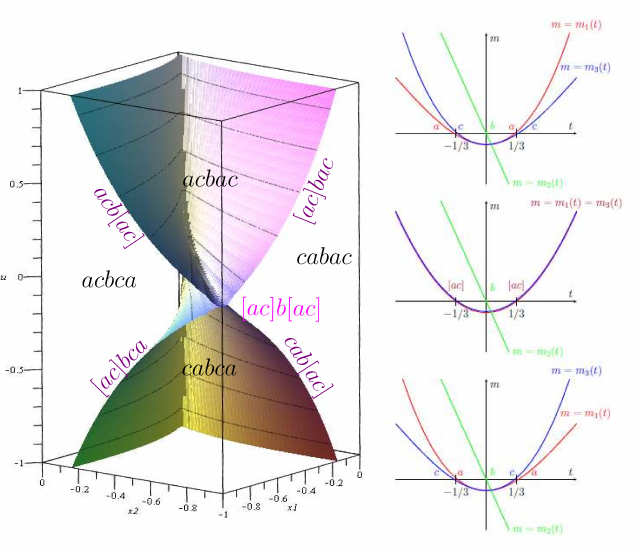}}
\caption{We set 
$\beta_1(t)=1+ut$, $\beta_2(t)=1$, $\beta_3(t)=1-ut$.  
Left: zero locus of $u^{-1}r_{1,3}(x_1,x_2)$ 
and adjacent itineraries. 
Right: graphs of $m_j(t)$ 
for $x_1=x_2=-1/18$ 
and $u=1/2$, $u=0$ and $u=-1/2$. 
}
\label{fig:betaprime}
\end{figure}

\goodbreak



\begin{proof}[Proof of Proposition \ref{prop:acbHk}]
Consider a locally convex curve  
$\Gamma\in\cL^{[H^r]}_3(\cdot\,;\cdot)$ of class $H^r$
with $\iti(\Gamma)=[acb]$ 
and \emph{local triangular presentation} 
$\Gamma_L:[-\epsilon,\epsilon]\to \acute\eta\acute\sigma\Lo^1_4$, 
\begin{equation*}
\label{equation:gij}
\Gamma_L(t)=
\begin{pmatrix}
-g_{1,1}(t) & -1 & 0 & 0 \\
-g_{2,1}(t) & -g_{2,2}(t) & -g_{2,3}(t) & -1 \\ 
-1 & 0 & 0 & 0 \\ 
g_{4,1}(t) & g_{4,2}(t) & 1 & 0 
\end{pmatrix} 
\end{equation*}
with logarithmic derivative 
$(\Gamma_L(t))^{-1}\Gamma'_L(t)=
\beta_1(t)\fl_1+\beta_2(t)\fl_2+\beta_3(t)\fl_3$, i.e., 
\begin{equation}
\label{equation:IVPgij}
\begin{gathered}
g'_{1,1} =\beta_1, \quad
g'_{2,1} =g_{2,2}\beta_1, \quad 
g'_{2,2}=g_{2,3}\beta_2, \quad 
g'_{2,3}=\beta_3, \\
g'_{4,1} =g_{4,2}\beta_1, \quad 
g'_{4,2}=\beta_2. 
\end{gathered}
\end{equation} 

Let as before $m_j=m_j(t)$ be the $j\times j$ southwest minor 
of $\Gamma_L(t)$. 
We have 
\begin{equation}
\label{equation:mj}
\begin{gathered}
m_1=g_{4,1}, \quad
m'_1=g_{4,2}\beta_1, \quad
m_2=-g_{4,2}, \quad
m'_2=-\beta_2, \\
m_3=g_{4,2}g_{2,3}-g_{2,2}, \quad
m'_3=g_{4,2}\beta_3.
\end{gathered}
\end{equation}

The so called \emph{local triangular presentation} 
is obtained as follows. 
Fix a compact subinterval $J\subset(0,1)$ containing the 
singular set $\sing(\Gamma)=\{t_\Gamma\}$ in its interior and 
such that $\Gamma[J]\subset\cU_{\acute\eta\acute\sigma}$ 
and consider an orientation-preserving diffeomorphism  
$\theta:[-\epsilon,\epsilon]\to J$ satisfying $\theta(0)=t_\Gamma$.  
We set 
$\Gamma_L(t)=
\acute\eta\acute\sigma\bL_{\acute\eta\acute\sigma}(\Gamma(\theta(t)))$.  
The positive functions 
$\beta_1,\beta_2,\beta_3:[-\epsilon,\epsilon]\to(0,+\infty)$ 
thus obtained are of class $H^{r-1}$. 
To simplify the computations, 
we assume without loss of generality that the reparameterization $\theta$ 
is chosen so as to produce $\beta_2(t)=1$ constant 
(at least in a small neighborhood of $t=0$). 
We already know from Theorem 4 of 
\cite{Goulart-Saldanha0} that 
$\iti(\Gamma)=[acb]$ 
if and only if 
$m_2(t)$ has a simple zero at $t=0$ and  
both $m_1(t)$ and $m_3(t)$ have a double zero at 
$t=0$. 
From \eqref{equation:mj}, we therefore have  
$g_{4,1}(0)=g_{4,2}(0)=g_{2,2}(0)=0$. 
Notice that the parameters 
$g_{1,1}(0), g_{2,1}(0), g_{2,3}(0)$ 
parameterize the intersection point 
$\Gamma(t_\Gamma)\in\Bru_{\acute\eta\acute\sigma}$, 
but do not change either the singular set or the itinerary of 
$\Gamma$. 
All these simplifications taken into account, 
Equation \eqref{equation:mj} boils down to 
\begin{equation}
\label{equation:mjbeta2const}
\begin{gathered}
m_1(0)=0, \quad
m'_1(t)=t\beta_1(t), \quad
m_2(t)=-t, \\ 
m_3(0)=0, \quad
m'_3(t)=t\beta_3(t).
\end{gathered}
\end{equation}
We are ready to define the desired function $u$:
\[
b_1 = \beta_1(0) > 0, \quad
b_3 = \beta_3(0) > 0, \quad
u= u(\Gamma) = \frac{b_3 \beta'_1(0) - b_1 \beta'_3(0)}{2b_1b_3}.\]
The first three items are clear.

We now study the ways the common double zero of $m_1(t)$ 
and $m_3(t)$ at $t=0$ can possibly split 
as we slightly perturb $\Gamma$
to obtain a curve $\tilde\Gamma$.
Consider the hyperplane    
\[S=
\left\{M(\bfx,\bfy) = \begin{pmatrix}
-y_1 & -1 & 0 & 0 \\
-y_2 & -x_1 & -y_3 &- 1 \\
-1 & 0 & 0 & 0 \\
x_2 & 0 & 1 & 0 \end{pmatrix}
\in\acute\eta\acute\sigma\Lo^1_4 \,;\, 
\begin{array}{l} \bfx=(x_1,x_2)\in\RR^2, \\ 
\bfy=(y_1,y_2,y_3)\in\RR^3 
\end{array}\right\}.\] 
Notice that the locally convex curve $\Gamma$ intersects 
the smooth codimension one submanifold 
$\bQ_{\acute\eta\acute\sigma}[S]
\subset\cU_{\acute\eta\acute\sigma}\subset\Spin_{n+1}$ 
transversally and at a single value of the parameter $t\in(0,1)$:
transversality comes from $g_{4,2}(t)=t$.
Transversality can also be assumed for $\tilde\Gamma$.
For each $\tilde\Gamma$, 
we also consider the corresponding 
local triangular presentation 
$\tilde\Gamma_L$, with logarithmic derivative 
$\tilde\beta_1(t)\fl_1+\fl_2+\tilde\beta_3(t)\fl_3$ 
(as before, we assume the reparameterization $\tilde\theta$ 
is such that $\tilde\beta_2(t)=1$, constant). 
After a reparametrization of $t$ by a translation,
we may assume that $\tilde\Gamma_L$ intersects $S$ at $t = 0$.
The corresponding version of 
Equation \eqref{equation:mjbeta2const} is 
\begin{equation}
\label{equation:tildemjbeta2const}
\begin{gathered}
\tilde m_1(t)=x_2, \quad
\tilde m'_1(t)=t\tilde\beta_1(t), \quad
\tilde m_2(t)=-t, \\ 
\tilde m_3(t)=-x_1, \quad
\tilde m'_3(t)=t\tilde\beta_3(t).
\end{gathered}
\end{equation}
Notice that the graphs of the 
functions $\tilde m_1=\tilde m_1(t)$ 
and $\tilde m_3=\tilde m_3(t)$ are convex in a 
neighborhood of their local minima 
$(0, x_2)$ and $(0, -x_1)$, respectively. 

Assume $u > 0$:
we prove the topological normal form.
If $x_2 > 0$ or $x_1 < 0$ the itinerary
is as in Figure \ref{fig:planebetaprime}.
We are left with studying the quadrant $x_1 > 0$, $x_2 < 0$.
Let $\tilde b_1 = \tilde\beta_1(0) \approx b_1$,
$\tilde b_3 = \tilde\beta_3(0) \approx b_3$.
For fixed $\tilde\beta_1$ and $\tilde\beta_3$,
we consider the line segment
$(x_1,x_2) = (2s \tilde b_3 c,- 2(1-s) \tilde b_1 c)$ 
where $c > 0$ is fixed and $s \in [0,1]$.
When $s$ moves from $0$ to $1$,
the roots of $\tilde m_1$ move monotonically towards $0$.
Similarly,
the roots of $\tilde m_3$ move monotonically away from $0$.

We study the point $s = \frac12$.
We have $\tilde m_1(0) = x_2 = -\tilde b_1c$,
$\tilde m_3(0) = -x_1 = -\tilde b_3c$
and therefore $\tilde b_3\tilde m_1(0) = \tilde b_1\tilde m_3(0)$.
For small $|t|$ and $\tilde\beta_i$ near $\beta_i$, we have
$\tilde b_3\tilde \beta_1'(t) > \tilde b_1\tilde\beta_3'(t)$.
Thus, for small $|t|$, $t \ne 0$, we have
$\tilde b_3t \tilde \beta_1(t) > \tilde b_1t \tilde\beta_3(t)$
and therefore
$\tilde b_3\tilde m_1'(t) > \tilde b_1\tilde m_3'(t)$.
Thus, $t > 0$ implies
$\tilde b_3\tilde m_1(t) > \tilde b_1\tilde m_3(t)$
and $t < 0$ implies
$\tilde b_3\tilde m_1(t) < \tilde b_1\tilde m_3(t)$.
Thus, if $t > 0$ and $\tilde m_1(t) = 0$ we have $\tilde m_3(t) < 0$;
if $t < 0$ and $\tilde m_1(t) = 0$ we have $\tilde m_3(t) > 0$.
At this point the itinerary is therefore $acbac$.

Monotonicity implies that as $s$ goes from $0$ to $1$
the itinerary changes from $a[cb]a$ (at $s = 0$)
to $acbca$ to $acb[ac]$ (for a unique $s_{-} \in (0,\frac12)$)
to $acbac$ (at an open neighborhood of $s = \frac12$)
to $[ac]bac$ (for a unique $s_{+} \in (\frac12,1)$)
to $cabac$ to $c[ab]c$ (at $s = 1$).
A piecewise linear reparamatrization leads to
Figure \ref{fig:planebetaprime}.
The case $u < 0$ is of course similar.
\end{proof}

\section{Final remarks}
\label{section:final}

For $q \in \Quat_{n+1}$, let $\cL_n(q)$ be
the space of locally convex curves $\Gamma$ in $\Spin_{n+1}$
with $\Gamma(0) = 1$ and $\Gamma(1) = q$.
The present paper constructs a stratification of the space $\cL_n(q)$.
We proved the necessary results for the construction
of a homotopy equivalent CW complex.
We detail this construction
and prove some consequences in \cite{Goulart-Saldanha-cw}.
In this final section we discuss some of the methods involved,
state a few consequences proved in
\cite{Goulart-Saldanha-cw, Alves-Goulart-Saldanha}
and also mention a conjectural result.

One important construction in \cite{Shapiro-Shapiro} 
and \cite{Saldanha3} is the add-loop procedure, 
which, in certain cases, is used to loosen up 
compact families of nondegenerate curves
through a homotopy in $\cL_n(q)$. 
The resulting families of curly curves are then maleable:
if a homotopy exists in the space of immersions, another
homotopy exists in the space of locally convex curves. 
In \cite{Saldanha3}, for instance, open dense subsets 
$\mathcal{Y}_{\pm}\subset\cL_2(\pm 1)$ are shown to be 
homotopy equivalent to the space of loops $\Omega\Ss^3$.
This approach is reminiscent of classical constructions such as 
Thurston's eversion of the sphere by corrugations 
\cite{Levy-Maxwell-Munzner} 
and the proof of Hirsch-Smale Theorem \cite{Hirsch, Smale}. 
It can be considered as an elementary instance of the h-principle 
of Eliashberg and Gromov \cite{Eliashberg-Mishachev, Gromov}.
Theorem 3, 
Corollary 1.1 
and Lemma 10.3 
in \cite{Goulart-Saldanha-cw}
are based on this method and
apply to higher dimensions.
We restate here that Corollary 1.1:


\begin{coro}
\label{coro:center}
If $q \in \Quat_{n+1}\smallsetminus Z(\Quat_{n+1})$ then
the inclusion $i_q: \cL_n(q) \to \Omega\Spin_{n+1}$
is a weak homotopy equivalence.
\end{coro}


We now restate the main result from \cite{Alves-Goulart-Saldanha},
which gives the homotopy type of spaces of locally convex curves
in the sphere $\Ss^3$.
A similar result for $\Ss^2$ is the main result in \cite{Saldanha3}.
Recall that $\Spin_4 = \Ss^3 \times \Ss^3$; 
the subgroup $\Quat_4 \subset  \Spin_4$
is generated by $(1,-1)$, $(\mathbf{i},\mathbf{i})$
and $(\mathbf{j},\mathbf{j})$.
The center of $\Quat_4$ is $Z(\Quat_4) = \{(\pm 1, \pm 1)\}$.


\begin{theo}
\label{theo:L3}
We have the following weak homotopy equivalences:
\begin{align*}
\cL_3((+1,+1)) &\approx 
\Omega(\Ss^3 \times \Ss^3) \vee \Ss^4 \vee \Ss^8 \vee \Ss^8
\vee \Ss^{12} \vee \Ss^{12} \vee \Ss^{12} \vee \cdots, \\
\cL_3((-1,-1)) &\approx 
\Omega(\Ss^3 \times \Ss^3) \vee \Ss^2 \vee \Ss^6 \vee \Ss^6
\vee \Ss^{10} \vee \Ss^{10} \vee \Ss^{10} \vee \cdots, \\
\cL_3((+1,-1)) &\approx 
\Omega(\Ss^3 \times \Ss^3) \vee \Ss^0 \vee \Ss^4 \vee \Ss^4
\vee \Ss^{8} \vee \Ss^{8} \vee \Ss^{8} \vee \cdots, \\
\cL_3((-1,+1)) &\approx 
\Omega(\Ss^3 \times \Ss^3) \vee \Ss^2 \vee \Ss^6 \vee \Ss^6
\vee \Ss^{10} \vee \Ss^{10} \vee \Ss^{10} \vee \cdots.
\end{align*}
The above bouquets include one copy of $\Ss^k$,
two copies of $\Ss^{(k+4)}$, \dots, $j+1$ copies of $\Ss^{(k+4j)}$, \dots,
and so on.
\end{theo}

The presence of $\Ss^0$ in the bouquet
indicates the presence of the contractible connected component
of convex curves.

Our methods allow us to study the corresponding problem
for locally convex curves in $\Ss^n$, $n > 3$.
We do not have a conjectural homotopy type in general,
but we hope to be able to prove the following result.

\begin{conj}
\label{conj:nothomotopic}
If $q \in Z(\Quat_{n+1})$
then the space $\cL_n(q)$ is \emph{not} homotopy equivalent
to $\Omega\Spin_{n+1}$.
\end{conj}

\bibliography{gs}
\bibliographystyle{plain}



\noindent
Victor Goulart \\
Departamento de Matem\'atica, UFES, \\
Av. Fernando Ferrari 514; Campus de Goiabeiras, Vit\'oria, ES 29075-910, Brazil. \\
\url{jose.g.nascimento@ufes.br}

\smallskip

\noindent
 Nicolau C. Saldanha \\
Departamento de Matem\'atica, PUC-Rio, \\
R. Marqu\^es de S. Vicente 255, Rio de Janeiro, RJ 22451-900, Brazil.  \\
\url{saldanha@puc-rio.br}

\end{document}